\documentclass{article}

\usepackage{arxiv}
\pdfoutput=1

\usepackage[utf8]{inputenc} 
\usepackage[T1]{fontenc}    
\usepackage{hyperref}       
\usepackage{url}            
\usepackage{booktabs}       
\usepackage{amsfonts}       
\usepackage{nicefrac}       
\usepackage{microtype}      
\usepackage{lipsum}
\usepackage[cmex10]{amsmath}
\usepackage{graphicx}
\usepackage{amssymb}
\usepackage{graphics} 
\usepackage{tikz-cd}
\usepackage{tikz}
\usetikzlibrary{arrows,decorations.pathmorphing,backgrounds,fit,positioning,shapes.symbols,chains}

\newtheorem{theorem}{Theorem}[section]
\newtheorem{proposition}[theorem]{Proposition}
\newtheorem{lemma}[theorem]{Lemma}
\newtheorem{corollary}[theorem]{Corollary}
\newtheorem{definition}[theorem]{Definition}
\newtheorem{remark}[theorem]{Remark}
\newtheorem{example}[theorem]{Example}
\newenvironment{proof}[1][Proof]{\noindent\textbf{#1.} }{\ \rule{0.5em}{0.5em}}

\title{Infinities within Finitely Supported Structures}

\author{
Andrei Alexandru \\
Romanian Academy, Institute of Computer Science\\
 Ia\c si, Romania \\
  \texttt{andrei.alexandru@iit.academiaromana-is.ro} \\
   \And
Gabriel Ciobanu\\
Romanian Academy, Institute of Computer Science\\
and A.I.Cuza University of Ia\c si, Romania\\
  \texttt{gabriel@info.uaic.ro}}

\begin{document}

\maketitle

\begin{abstract}

The theory of finitely supported algebraic structures is related to Pitts 
theory of nominal sets (by equipping finitely supported sets with finitely 
supported internal algebraic laws). It represents a reformulation of 
Zermelo Fraenkel set theory obtained by requiring every set theoretical 
construction to be finitely supported according to a certain action of a 
group of permutations of some basic elements named atoms. Its main purpose 
is to let us characterize infinite algebraic structures, defined involving 
atoms, only by analyzing their finite supports. The first goal of this 
paper is to define and study different kinds of infinities and the notion 
of `cardinality' in the framework of finitely supported structures. We 
present several properties of infinite cardinalities. Some of these 
properties are extended from the non-atomic Zermelo Fraenkel set theory 
into the world of atomic objects with finite support, while other 
properties are specific to finitely supported structures. We also compare 
alternative definitions of `infinite finitely supported set', and we 
finally provide a characterization of finitely supported countable sets.

\end{abstract}

\section{Introduction}

The theory of finitely supported algebraic structures which is known 
under the name of `nominal sets' (when dealing with computer science 
applications) or `Finitely Supported Mathematics' (in some pure set 
theoretical papers related to the foundations of mathematics) represents 
an alternative framework for working with infinite structures 
hierarchically constructed by involving some basic elements (called 
atoms) by dealing only with a finite number of entities that form their 
supports. The theory of nominal sets is presented in a categorical 
manner as a Zermelo-Fraenkel (ZF) alternative to Fraenkel and Mostowski 
1930s permutation models of set theory with atoms \cite{pitts-2}. A 
nominal set is defined as a usual~ZF set endowed with a group action of 
the group of (finitary) permutations over a certain fixed countable ZF 
set~$A$ (also called the set of atoms by analogy with the Fraenkel and 
Mostowski framework) formed by elements whose internal structure is not 
taken into consideration (i.e. by elements that can be checked only for 
equality), satisfying a finite support requirement. This requirement 
states that for any element in a nominal set there should exist a finite 
set of atoms such that any permutation fixing pointwise this set of 
atoms also leaves the element invariant under the related group action. 
Nominal sets represents a categorical mathematical theory of names 
studying scope, binding, freshness and renaming in formal languages 
based upon symmetry. Inductively defined finitely supported sets (that 
are finitely supported elements in the powerset of a nominal set) 
involving the name-abstraction together with Cartesian product and 
disjoint union can encode syntax modulo renaming of bound variables. In 
this way, the standard theory of algebraic data types can be extended to 
include signatures involving binding operators. In particular, there is 
an associated notion of structural recursion for defining 
syntax-manipulating functions and a notion of proof by structural 
induction. Various generalizations of nominal were used in order to 
study automata, languages or Turing machines that operate over infinite 
alphabets; for this a relaxed notion of finiteness, called `orbit 
finiteness', was defined and means `having a finite number of orbits 
under a certain group action' \cite{boj}.

Finitely Supported Mathematics (FSM) is an alternative name for nominal 
algebraic structures, used in theoretical papers focused on the 
foundations of set theory (rather than on applications in computer 
science). In order to describe FSM as a theory of finitely supported 
algebraic structures (that is finitely supported sets \emph{together with 
finitely supported internal algebraic laws}), we use nominal sets (without 
the requirement that the set $A$ of atoms is countable) which by now on 
will be called invariant sets motivated by Tarski's approach regarding 
logicality (i.e. a logical notion is defined by Tarski as one that is 
invariant under the permutations of the universe of discourse). The 
cardinality of the set of atoms \emph{cannot} be internally compared with 
any other ZF cardinality, and so we just say that atoms form an infinite 
set without any specifications regarding its cardinality. In FSM we 
actually study the finitely supported subsets of invariant sets together 
with finitely supported relations (order relations, functions, algebraic 
laws etc), and so FSM becomes a theory of atomic algebraic structures 
constructed/defined according to the finite support requirement. The 
requirement of being finitely supported under a canonical action of the 
group of permutation of atoms (constructed under the rules in Proposition 
\ref{p1}) is actually an axiom adjoined to ZF, and so non-finitely 
supported structures are not allowed (they do not exist) in FSM.

FSM contains the family of `non-atomic' (ordinary) ZF sets (which are 
proved to be trivial FSM sets) and the family of `atomic' sets with 
finite supports (hierarchically constructed from the empty set and the 
fixed ZF set $A$). The main question now is whether a classical ZF 
result (obtained in ZF framework for non-atomic sets) can be adequately 
reformulated by replacing `non-atomic element/set' with `atomic finitely 
supported element/set' (according to the canonical actions of the group 
of one-to-one transformations of $A$ onto itself) in order to be valid 
also for atomic sets with finite supports. The (non-atomic) ZF results 
cannot be directly translated into the framework of atomic finitely 
supported sets, unless we are able to reprove their new formulations 
internally in FSM, i.e. by involving only \emph{finitely supported 
structures} even in the intermediate steps of the proof. This is because 
the family of finitely supported sets is not closed under subset 
constructions, and we cannot use something outside FSM in order to prove 
something in FSM.

The meta-theoretical techniques for the translation of a result from 
non-atomic structures to atomic structures are fully described in 
\cite{book} (or in \cite{pitts-2}, with the mention that, working on 
foundations of mathematics, and so we use a slightly different 
terminology for the same concept). They are based on a refinement of the 
finite support principle form \cite{pitts-2} called ``$S$-finite 
supports principle" claiming that for any finite set $S$ of atoms, 
anything that is definable in higher order logic from $S$-supported 
structures using $S$-supported constructions is also $S$-supported. The 
formal involvement of the $S$-finite support principles implies a 
constructive method for defining the support of a structure by employing 
the supports of the sub-structures of a related structure.

In this paper we introduce the notion of `cardinality' of a finitely 
supported set, and we prove several properties of this concept. Some 
properties are naturally extended from the non-atomic ZF into the world of 
atomic structures. In this sense we prove that Cantor-Schr{\"o}der-Bernstein 
theorem for cardinalities is still valid in FSM. Several other cardinality 
properties are preserved from ZF. 
However, although Cantor-Schr{\"o}der-Bernstein theorem can be 
successfully translated into FSM, its ZF dual is no longer valid in FSM. 
Other specific FSM properties of cardinalities (that do not have 
related~ZF correspondents) are also emphasized. We introduce various 
definition for infinity and we compare them, providing relevant examples 
of atomic sets verifying the conditions of each such definition. Finally, 
we introduce and study the concept of countability in FSM.

\section{Finitely Supported Sets} \label{FMset}

A ZF finite set is referred to a set for which there is a bijection with 
a finite ordinal; a ZF infinite set is a set that is not finite. Adjoin 
to ZF a special infinite set $A$ (called `the set of atoms'; despite 
classical set theory with atoms we do not need to modify the axiom of 
extensionality). Actually, atoms are entities whose internal structure 
is considered to be irrelevant which are considered as basic for a 
higher-order construction, i.e. their internal structure is not taken 
into consideration.

A \emph{transposition} is a function $(a\, b):A\to A$ given 
by~$(a\, b)(a)=b$, $(a\, b)(b)=a$ and $(a\, b)(n)=n$ for $n\neq a,b$. A 
\emph{(finitary) permutation} of $A$ in FSM is a one-to-one transformation 
of $A$ onto itself (a bijection of $A$) generated by composing finitely 
many transpositions. We denote by $S_{A}$ the set of all finitary 
permutations of $A$. According to Proposition 2.6 from \cite{book}, a 
function $f:A \to A$ is a bijection on~$A$ in FSM if and only if it leaves 
unchanged all but finitely many elements of~$A$. Thus, in FSM a function 
is a one-to-one transformation of $A$ onto itself if and only if it is a 
(finitary) permutation of $A$. Thus, the notions `permutation (bijection) 
of $A$' and `finitary permutation of $A$' coincide in FSM.

\begin{definition}\label{2.4} Let $X$ be a  ZF set. 
\begin{enumerate}
\item An \emph{$S_{A}$-action} on $X$ is a function $\cdot:S_{A}\times 
X\rightarrow X$ having the properties that $Id\cdot x=x$ and 
$\pi\cdot(\pi'\cdot x)=(\pi\circ\pi')\cdot x$ for all $\pi,\pi'\in S_{A}$ 
and $x\in X$, where $Id$ is the identity mapping on $A$. An \emph{$S_{A}$-set} is a pair $(X,\cdot)$ where $X$ is a 
ZF set, and $\cdot:S_{A}\times X\to X$ is an $S_{A}$-action on $X$.

\item Let $(X,\cdot)$ be an $S_{A}$-set. We say that \emph{$S\subset A$ 
supports $x$} whenever for each $\pi\in Fix(S)$ 
we have $\pi\cdot x=x$, where $Fix(S)=\{\pi\,|\,\pi(a)=a,\forall a\in S\}$. The least finite set  supporting $x$ (which exists according to Proposition \ref{p11}) is 
called \emph{the support of $x$} and is denoted by $supp(x)$. An empty supported element is called \emph{equivariant}; this means that $x \in X$ is equivariant if and only if $\pi \cdot x=x$, $\forall \pi \in S_{A}$.

\item Let $(X,\cdot)$ be an $S_{A}$-set. We say that $X$ is an 
\emph{invariant set} if for each $x\in X$ there exists a finite set 
\emph{$S_{x}\subset A$ }which supports $x$. 
\end{enumerate}
\end{definition}

\begin{proposition}\label{p11} \cite{book} 
Let $X$ be an $S_{A}$-set and let $x\in X$. If there exists a finite set 
supporting $x$ (particularly, if $X$ is an invariant set), then there 
exists a least finite set supporting $x$ which is constructed as the 
intersection of all finite sets supporting $x$. 
\end{proposition}

\begin{proposition}\label{2.15} \cite{book}
Let $(X,\cdot)$ be an $S_{A}$-set, and~$\pi\in S_{A}$. If $x\in X$ is 
finitely supported, then $\pi\cdot x$ is finitely supported and 
$supp(\pi\cdot x)=\pi(supp(x))$. 
\end{proposition}

\begin{example}\label{2.7} \ \ 
\begin{enumerate}
\item The set $A$ of atoms is an $S_{A}$-set with the $S_{A}$-action
$\cdot:S_{A}\times A\rightarrow A$ defined by $\pi\cdot a:=\pi(a)$
for all $\pi\in S_{A}$ and $a\in A$. $(A,\cdot)$ is an invariant set because
for each $a\in A$ we have that $\{a\}$ supports $a$. Furthermore, $supp(a)=\{a\}$
for each $a\in A$. 
\item The set $S_{A}$ is an $S_{A}$-set with the $S_{A}$-action $\cdot:S_{A}\times S_{A}\rightarrow S_{A}$
defined by $\pi\cdot\sigma:=\pi\circ\sigma\circ\pi^{-1}$ for all
$\pi,\sigma\in S_{A}$. $(S_{A},\cdot)$ is an invariant set because for
each $\sigma\in S_{A}$ we have that the finite set $\{a\in A\,|\,\sigma(a)\neq a\}$
supports~$\sigma$. Furthermore, $supp(\sigma)=\{a\in A\,|\,\sigma(a)\neq a\}$
for each $\sigma\in S_{A}$.

\item Any ordinary (non-atomic) ZF-set $X$ (such as $\mathbb{N},\mathbb{Z},\mathbb{Q}$
or $\mathbb{R}$ for example) is an invariant set with the single possible $S_{A}$-action
$\cdot:S_{A}\times X\rightarrow X$ defined by $\pi\cdot x:=x$ for
all $\pi \in S_{A}$ and $x\in X$. 
\end{enumerate}
\end{example}

\begin{proposition} \label{p1} 
 Let $(X,\cdot)$ and $(Y,\diamond)$ be $S_{A}$-sets.
\begin{enumerate}
\item The Cartesian 
product $X\times Y$ is also an $S_{A}$-set with the $S_{A}$-action 
$\otimes:S_{A}\times(X\times Y)\rightarrow(X\times Y)$ defined by 
$\pi\otimes(x,y)=(\pi\cdot x,\pi\diamond y)$ for all $\pi\in S_{A}$ and all 
$x\in X$, $y\in Y$. If $(X,\cdot)$ and $(Y,\diamond)$ are invariant sets, 
then $(X\times Y,\otimes)$ is also an invariant set.

\item The powerset $\wp(X)=\{Z\,|\, 
Z\subseteq X\}$ is also an $S_{A}$-set with the $S_{A}$-action $\star: 
S_{A}\times\wp(X) \rightarrow \wp(X)$ defined by $\pi\star Z:=\{\pi\cdot 
z\,|\, z\in Z\}$ for all $\pi \in S_{A}$, and all $Z \subseteq X$. For 
each invariant set $(X,\cdot)$, we denote by $\wp_{fs}(X)$ the set formed 
from those subsets of~$X$ which are finitely supported according to the 
action $\star$ .
$(\wp_{fs}(X),\star|_{\wp_{fs}(X)})$ is an invariant set, where 
$\star|_{\wp_{fs}(X)}$ represents the action $\star$ restricted to 
$\wp_{fs}(X)$.

\item The finite powerset of $X$ $\wp_{fin}(X)=\{Y \subseteq X\,|\, Y 
\text{finite}\}$ and the cofinite powerset of $X$ $\wp_{cofin}(X)=\{Y 
\subseteq X\,|\, X\setminus Y \text{finite}\}$ are $S_{A}$-sets with the 
$S_{A}$-action $\star$ defined as in item 2. If $X$ is an invariant set, 
then both $\wp_{fin}(X)$ and $\wp_{cofin}(X)$ are invariant sets.

\item Let $(X,\cdot)$ and $(Y,\diamond)$ be $S_{A}$-sets. We define the
disjoint union of $X$ and $Y$ by $X+Y=\{(0,x)\,|\, x\in X\}\cup\{(1,y)\,|\, y\in Y\}$.
$X+Y$ is an $S_{A}$-set with the $S_{A}$-action $\star:S_{A}\times(X+Y)\rightarrow(X+Y)$
defined by $\pi\star z=(0,\pi\cdot x)$ if $z=(0,x)$ and $\pi\star z=(1,\pi\diamond y)$
if $z=(1,y)$. If $(X,\cdot)$ and $(Y,\diamond)$ are invariant sets, then
$(X+Y,\star)$ is also an invariant set: each $z\in X+Y$ is either of the
form $(0,x)$ and supported by the finite set supporting $x$ in~$X$,
or of the form $(1,y)$ and supported by the finite set supporting
$y$ in~$Y$. 

\end{enumerate}
\end{proposition}

\begin{definition}\label{2.14}
\begin{enumerate}
\item Let $(X,\cdot)$ be an $S_{A}$-set. A subset~$Z$ of $X$ is called 
\emph{finitely supported} if and only if $Z\in\wp_{fs}(X)$ with the notations from Proposition \ref{p1}.  A subset $Z$ of $X$ is \emph{uniformly supported} if all the elements of $Z$ are supported by the same set $S$ (and so $Z$ is itself supported by $S$ as an element of $\wp_{fs}(X)$). Generally, an FSM set is a finitely supported subset (possibly equivariant) of an invariant set. 

\item Let $(X,\cdot)$ be a finitely supported subset of an $S_{A}$- set $(Y, \cdot)$. A subset~$Z$ of $Y$ is called 
\emph{finitely supported subset of $X$} (and we denote this by $Z \in \wp_{fs}(X)$) if and only if $Z\in\wp_{fs}(Y)$ and $Z \subseteq X$. Similarly, we say that a uniformly supported subset of $Y$ contained in $X$ is a \emph{uniformly supported subset of $X$}.
\end{enumerate}
\end{definition} 

From Definition \ref{2.4}, a subset~$Z$ of an invariant set $(X, \cdot)$ 
is finitely supported by a set $S \subseteq A$ if and only if $\pi \star 
Z \subseteq Z$ for all $\pi \in Fix(S)$.  This is because any permutation of atoms should have finite order.

\begin{proposition} \label{4.4-9}
\begin{enumerate}
\item  Let $X$ be a finite subset of an invariant set $(U, \cdot)$. Then $X$ is finitely supported and $supp(X)=\cup\{supp(x)\,|\, x\in X\}$.
\item Let $X$ be a uniformly supported subset of an invariant set $(U, \cdot)$.  Then $X$ is finitely supported and $supp(X)=\cup\{supp(x)\,|\, x\in X\}$.
\end{enumerate}
\end{proposition} 
\begin{proof}  
1.  Let $X=\left\{ x_{1},\ldots, x_{k}\right\}$, and  $S=supp(x_{1})\cup\ldots\cup supp(x_{k})$. Obviously, $S$ supports $X$. Indeed, let us consider $\pi\in Fix(S)$. We have that $\pi\in Fix(supp(x_{i}))$ for each
$i\in\{1,\ldots ,k\}$. Therefore, $\pi\cdot x_{i}=x_{i}$ for each $i\in\{1,\ldots ,k\}$
because $supp(x_{i})$ supports $x_{i}$ for each $i\in\{1,\ldots ,k\}$, and so  $supp(X) \subseteq S$. It remains to prove that $S \subseteq supp(X)$. Consider $a \in S$. This means there exists $j\in\{1,\ldots ,k\}$ such that $a \in supp(x_{j})$. Let $b$ be an atom such that $b \notin supp(X)$ and $b \notin supp(x_{i})$, $\forall i\in\{1,\ldots ,k\}$. Such an atom exists because $A$ is infinite, while $supp(X)$ and $supp(x_{i})$, $ i\in\{1,\ldots ,k\}$, are all finite. We prove by contradiction that $(b\; a) \cdot x_{j} \notin X$. Indeed, suppose that $(b\; a) \cdot x_{j} \in X$. In this case there is $y \in X$ with $(b\; a) \cdot x_{j}=y$. Since  $a \in supp(x_{j})$, we have $b \in (b\; a)(supp(x_{j}))$. However, according to Proposition \ref{2.15}, we have $supp(y)=(b\; a)(supp(x_{j}))$. We obtain that $b \in supp(y)$ for some $y \in X$, which is a contradiction with the choice of $b$. Therefore, $(b\; a) \star X \neq X$, where~$\star$ is the standard $S_{A}$-action on $\wp(U)$ is defined in Proposition \ref{p1}(2).  Since $b \notin supp(X)$, we prove by contradiction  that $a \in supp(X)$. Indeed, suppose that $a \notin supp(X)$. It follows that the transposition $(b\; a)$ fixes each element from $supp(X)$, i.e. $(b\; a) \in Fix(supp(X))$. Since $supp(X)$ supports $X$, by Definition \ref{2.4}, it follows that $(b\; a) \star X=X$, which is a contradiction. Thus, $a \in supp(X)$, and so $S \subseteq supp(X)$. 

2. Since $X$ is uniformly supported, there exists a finite subset of atoms $T$ such that $T$ supports every $x \in X$, i.e. $supp(x) \subseteq T$ for all $x \in X$. Thus, $\cup\{supp(x)\,|\, x\in X\} \subseteq T$. Clearly, $supp(X) \subseteq \cup\{supp(x)\,|\, x\in X\}$. Conversely, let $a \in \cup\{supp(x)\,|\, x\in X\}$. Thus, there exists $x_{0} \in X$ such that $a \in supp(x_{0})$. Let $b$ be an atom such that $b \notin supp(X)$ and $b \notin T$. Such an atom exists because $A$ is infinite, while $supp(X)$ and $T$ are both finite. We prove by contradiction that $(b\; a) \cdot x_{0} \notin X$. Indeed, suppose that $(b\; a) \cdot x_{0}=y \in X$.  Since  $a \in supp(x_{0})$, we have $b =(b\;a)(a) \in (b\; a)(supp(x_{0}))=supp((b\; a) \cdot x_{0})=supp(y)$. Since $supp(y) \subseteq T$, we get $b \in T$, a contradiction. Therefore, $(b\; a) \star X \neq X$.  Since $b \notin supp(X)$, we have that $a \in supp(X)$ as in the above item. 
\end{proof}

\begin{corollary}Let $X$ be a uniformly supported subset of an invariant set. Then $X$ is uniformly supported by $supp(X)$.
\end{corollary}
\begin{proof}Since $supp(X)=\cup\{supp(x)\,|\, x\in X\}$, we have $supp(x) \subseteq supp(X)$ for all $x \in X$ which means $supp(X)$ supports every $x \in X$. 
\end{proof}

\begin{proposition}\label{p111}
We have $\wp_{fs}(A)=\wp_{fin}(A) \cup \wp_{cofin}(A)$.
\end{proposition}

\begin{proof} We know that $B$ is finitely supported with $supp(B)=B$ 
whenever $B \subset A$ and $B$ is finite. If $C \subseteq A$ and~$C$ is 
cofinite, then $C$ is finitely supported by $A \setminus C$ with 
$supp(C)=A \setminus C$. However, if $D \subsetneq A$ is neither finite 
nor cofinite, then $D$ is not finitely supported. Indeed, assume by 
contradiction that there exists a finite set of atoms~$S$ supporting $D$. 
Since $S$ is finite and both $D$ and its complementary $C_{D}$ are 
infinite, we can take $a \in D \setminus S$ and $b \in C_{D} \setminus S$. 
Then the transposition $(a\,b)$ fixes $S$ pointwise, but $(a\,b) \star D 
\neq D$ because $(a\,b)(a)=b \notin D$; this contradicts the assertion 
that $S$ supports $D$.  Therefore, $\wp_{fs}(A)=\wp_{fin}(A) \cup 
\wp_{cofin}(A)$. \end{proof}

\begin{definition}\label{2.10-1}
Let $X$ and $Y$ be invariant sets.
\begin{enumerate} 
\item A function $f:X\rightarrow Y$
is \emph{finitely supported} if $f\in\wp_{fs}(X\times Y)$. The set of all finitely supported functions from $X$ to $Y$ is denoted by $Y^{X}_{fs}$.
\item Let $Z$ be a finitely supported 
subset of $X$ and $T$ a finitely supported 
subset of $Y$. A function $f:Z\rightarrow T$ is \emph{finitely supported} if 
$f\in\wp_{fs}(X\times Y)$. 
The set of all finitely supported functions from $Z$ to $T$ is denoted by~$T^{Z}_{fs}$.
\end{enumerate}
\end{definition}

\begin{proposition}\label{2.18'} \cite{book} 
Let $(X,\cdot)$ and $(Y,\diamond)$ be two invariant sets.
\begin{enumerate}
\item   $Y^{X}$ (i.e. the set of all functions from $X$ to $Y$) is an 
$S_{A}$-set with the $S_{A}$-action $\widetilde{\star}:S_{A}\times 
Y^{X}\rightarrow Y^{X}$ defined by $(\pi \widetilde{\star}f)(x) = 
\pi\diamond(f(\pi^{-1}\cdot x))$ for all $\pi\in S_{A}$, $f\in Y^{X}$ and 
$x\in X$. A function $f:X\rightarrow Y$ is finitely supported in the sense 
of Definition \ref{2.10-1} if and only if it is finitely supported with 
respect the permutation action $\widetilde{\star}$.
\item Let $Z$ be a finitely supported 
subset of $X$ and $T$ a finitely supported 
subset of $Y$. A function $f:Z\rightarrow T$ is 
supported by a finite set $S \subseteq A$ if and only if for all $x \in Z$ and all $\pi 
\in Fix(S)$ we have $\pi \cdot x \in Z$, $\pi \diamond f(x) \in T$ and $f(\pi\cdot x)=\pi\diamond f(x)$. Particularly, a function $f:X\rightarrow Y$ is 
supported by a finite set $S \subseteq A$ if and only if for all $x \in X$ and all $\pi 
\in Fix(S)$ we have $f(\pi\cdot x)=\pi\diamond f(x)$.
\end{enumerate}
\end{proposition}

\section{Cardinalities and Order Properties}

\begin{definition} \label{FM-event struc} \ 
\begin{itemize}
\item An \emph{invariant partially ordered set (invariant poset)} is an invariant set $(P,\cdot)$ together
with an equivariant partial order relation $\sqsubseteq$ on $P$.
An invariant  poset is denoted by $(P,\sqsubseteq,\cdot)$ or
simply $P$.

\item A \emph{finitely supported partially ordered set (finitely supported poset)} is a finitely supported subset $X$ of an invariant set $(P,\cdot)$ together
with a partial order relation $\sqsubseteq$ on $X$ that is finitely supported as a subset of $P\times P$.
\end{itemize}
\end{definition}

Two FSM sets $X$ and $Y$ are called equipollent if there exists a finitely supported 
bijection $f:X \to Y$.  The 
FSM cardinality of  $X$ is defined as the equivalence class of all FSM sets 
equipollent to $X$, and is denoted by $|X|$. This means that for two FSM sets $X$ 
and $Y$ we have $|X|=|Y|$ if and only if there 
exists a finitely supported bijection $f:X \to Y$. On the family of cardinalities we can 
define the relations:
\begin{itemize}
\item  $\leq$ by: \ $|X| \leq |Y|$ if and only if there is a finitely supported
injective mapping $f:X \to Y$;
\item  $\leq^{*}$ by: $|X| \leq^{*} |Y|$ if and only if there is a finitely supported surjective mapping $f:Y \to X$.
\end{itemize}

\begin{theorem} \label{cardord}
\begin{enumerate}
\item The relation $\leq$ is equivariant, reflexive, anti-symmetric and transitive, but it is not total.
\item The relation $\leq^{*}$ is equivariant, reflexive and transitive, but it is not anti-symmetric, nor total.
\end{enumerate}
\end{theorem}
\begin{proof}
\begin{itemize}
\item $\leq$ and $\leq^{*}$ are equivariant because  for any FSM sets $X$ and $Y$, whenever there is a finitely supported injection/ surjection $f:X \to Y$, according to Proposition \ref{2.15}, we have that $\pi \star f:\pi \star X \to \pi \star Y$,  defined by $(\pi \star f)(\pi \cdot x)=\pi \cdot f(x)$ for all $x \in X$,  is a finitely supported injective/surjective mapping, and so $\pi \star X$ is comparable with $\pi \star Y$ (under  $\leq$ or $\leq^{*}$, after case). 
\item  $\leq$ and $\leq^{*}$ are obviously reflexive because for each FSM set $X$, the identity of $X$ is an equivariant  bijection from $X$ to $X$.
\item $\leq$ and $\leq^{*}$ are transitive because for any FSM sets $X$, $Y$ and $Z$, whenever there are two finitely supported injections/surjections $f:X \to Y$ and $g:Y \to Z$, there exists an injection/surjection $g \circ f:X \to Z$ which is finitely supported by $supp(f) \cup supp(g)$. 
\item The anti-symmetry of $\leq$. 

\begin{lemma} \label{lem1}  Let $(B, \cdot)$ and \; $(C, \diamond)$ be two invariant sets. If there 
exist a finitely supported injective mapping $f: B \to C$ and a finitely 
supported injective mapping $g: C \to B$, then there exists a finitely 
supported bijective mapping $h:B \to C$.  Furthermore, $supp(h) \subseteq 
supp(f) \cup supp(g)$.
\end{lemma}

\emph{Proof of Lemma \ref{lem1}.} 
Let us define $F:\wp_{fs}(B) \to \wp_{fs}(B)$ by $F(X)=B-g(C-f(X))$ for all finitely supported subsets $X$ of $B$. 

\textbf{Claim 1:} $F$ is correctly defined, i.e. $Im(F) \subseteq \wp_{fs}(B)$.\\ 
For every finitely supported subset $X$ of $B$, we have that $f(X)$ is 
supported by $supp(f) \cup supp(X)$. Indeed, let $\pi \in Fix(supp(f) \cup 
supp(X))$. Let~$y$ be an arbitrary element from $f(X)$; then $y=f(x)$ for 
some $x \in X$. However, because $\pi \in Fix (supp(X))$, it follows that 
$\pi \cdot x \in X$ and so, because $supp(f)$ supports $f$ and $\pi$ fixes 
$supp(f)$ pointwise, from Proposition~\ref{2.18'} we get $\pi \diamond y= \pi 
\diamond f(x)= f(\pi \cdot x) \in f(X)$. Thus $\pi \widetilde{\star} 
f(X)=f(X)$, where~$\widetilde{\star}$ is the $S_{A}$-action on~$\wp_{fs}(C)$ 
defined as in Proposition~\ref{p1}.  Analogously, $g(Y)$ is finitely 
supported by $supp(g) \cup supp(Y)$ for all $Y \in \wp_{fs}(C)$. It is easy 
to remark that for every finitely supported subset~$X$ of $B$ we have that 
$C-f(X)$ is also supported by $supp(f) \cup supp(X)$, $g(C-f(X))$ is 
supported by $supp(g) \cup supp(f) \cup supp(X)$, and $B-g(C-f(X))$ is 
supported by $supp(g) \cup supp(f) \cup supp(X)$. Thus,~$F$ is well-defined.

\textbf{Claim 2:} $F$ is a finitely supported function.\\ 
We prove that $F$ is finitely supported by $supp(f) \cup supp(g)$. Let us 
consider $\pi \in Fix(supp(f) \cup supp(g))$. Since $\pi \in Fix(supp(f))$ 
and $supp(f)$ supports~$f$, according to Proposition \ref{2.18'} we have 
that $f(\pi \cdot x)=\pi \diamond f(x)$ for all $x \in B$. Thus, for every 
finitely supported subset $X$ of $B$ we have $f(\pi \star X)=\{f(\pi \cdot 
x)\;|\;x \in X\}=\{\pi \diamond f(x)\;|\;x \in X\}=\pi \widetilde{\star} 
f(X)$, where~$\star$ is the $S_{A}$-action on $\wp_{fs}(B)$ and 
$\widetilde{\star}$ is the $S_{A}$-action on $\wp_{fs}(C)$. Similarly, 
$g(\pi \widetilde{\star} Y)=\pi \star g(Y)$ for any finitely supported 
subset $Y$ of $C$. Therefore, $F(\pi \star X)=B-g(C-f(\pi \star 
X))=B-g(C-\pi \widetilde{\star} f(X)) \overset{\pi 
\widetilde{\star}C=C}{=}B-g(\pi \widetilde{\star}(C-f(X)))=B-(\pi \star 
g(C-f(X))) \overset{\pi \star B=B}{=} \pi \star (B- g(C-f(X)))=\pi \star 
F(X)$. From Proposition \ref{2.18'} it follows that $F$ is finitely 
supported. Moreover, because $supp(F)$ is the least set of atoms supporting 
$F$, we have $supp(F) \subseteq supp(f) \cup supp(g)$.

\textbf{Claim 3:} For any $X,Y \in \wp_{fs}(B)$ with $X \subseteq Y$, we have $F(X) \subseteq F(Y)$. This remark follows by direct calculation.

\textbf{Claim 4:} The set $S:=\{X\;|\; X \in \wp_{fs}(B), X \subseteq 
F(X)\}$ is a non-empty finitely supported subset of $\wp_{fs}(B)$. 
Obviously, $\emptyset \in S$. We claim that $S$ is supported by $supp(F)$. 
Let $\pi \in Fix(supp(F))$, and $X \in S$. Then $X \subseteq F(X)$. From the 
definition of $\star$ (see Proposition \ref{p1}) we have $\pi \star X 
\subseteq \pi \star F(X)$. According to Proposition \ref{2.18'}, because 
$supp(F)$ supports $F$, we have $\pi \star X \subseteq \pi \star F(X)=F(\pi 
\star X)$, and so $\pi \star X \in S$. It 
follows that $S$ is finitely supported, and $supp(S) \subseteq supp(F)$.

\textbf{Claim 5:} $T:=\underset{X \in S}{\cup}X$ is finitely supported by $supp(S)$.\\ 
Let $\pi \in Fix(supp(S))$, and $t \in T$. Since $T=\underset{X \in 
S}{\cup}X$, we have that there exists $Z \in S$ such that $t \in Z$. 
Therefore, $\pi \cdot t \in \pi \star Z$. However, since~$\pi$ fixes 
$supp(S)$ pointwise and $supp(S)$ supports $S$, we have 
that $\pi \star Z \in S$. Thus, there exists $Y \in S$ such that $\pi \star 
Z=Y$. Therefore $\pi \cdot t \in Y$, and so $\pi \cdot t \in \underset{X \in 
S}{\cup}X$. It follows that $\underset{X \in S}{\cup}X$ is finitely 
supported, and so $T=\underset{X \in S}{\cup}X \in \wp_{fs}(B)$. 
Furthermore, $supp(T) \subseteq supp(S)$.

\textbf{Claim 6:} We prove that $F(T)=T$.\\ 
Let $X \in S$ arbitrary. We have $X \subseteq F(X) \subseteq F(T)$. By 
taking the supremum on $S$, this leads to $T \subseteq F(T)$.  However, 
because $T \subseteq F(T)$, from Claim 3 we also have $F(T) \subseteq F(F(T))$. 
Furthermore, $F(T)$ is supported by $supp(F) \cup supp(T)$ (i.e. by $supp(f) 
\cup supp(g)$), and so $F(T) \in S$. According to the definition of $T$, we 
get $F(T) \subseteq T$.
\smallskip

We get $T=B-g(C-f(T))$, or equivalently, $B-T=g(C-f(T))$. 
Since~$g$ is injective, we obtain that for each $x \in B-T$, $g^{-1}(x)$ is 
a set containing exactly one element.
Let us define $h:B \to C$ by 
\[
h(x)=\left\{ \begin{array}{ll}
f(x), & \text{for}\: x \in T;\\
g^{-1}(x), & \text{for}\: x \in B-T.\end{array}\right. 
\]

\textbf{Claim 7:} We claim that $h$ is supported by the set $supp(f) \cup 
supp(g) \cup supp(T)$ (more exactly, by $supp(f) \cup supp(g)$, according to 
the previous~claims).
Let $\pi \in Fix(supp(f) \cup supp(g) \cup supp(T))$, and $x$ an arbitrary element of~$B$.

If $x \in T$, because $\pi \in Fix(supp(T))$ and $supp(T)$ supports $T$, 
we have $\pi \cdot x \in T$. Thus, from Proposition~\ref{2.18'} we get 
$h(\pi \cdot x)=f(\pi \cdot x)=\pi \diamond f(x)=\pi \diamond h(x)$.

If $x \in B-T$, we have $\pi \cdot x \in B-T$. Otherwise, we would obtain 
the contradiction $x=\pi^{-1} \cdot (\pi \cdot x) \in T$ because $\pi^{-1}$ 
also fixes $supp(T)$ pointwise. Thus, because $g$ is finitely supported, 
according to Proposition~\ref{2.18'} we have 
$h(\pi \cdot x)=g^{-1}(\pi \cdot x)=\{y \in C\;|\; g(y)$=$\pi \cdot x\}=\{y 
\in C\;|\; \pi^{-1} \cdot g(y)$=$x\}=\{y\in C\;|\; g(\pi^{-1} \diamond y)$=
$x\}\overset{\pi^{-1} \diamond y:= z}{=}\{\pi \diamond z \in C\;|\; 
g(z)$=$x\}=\pi \diamond \{ z \in C\;|\; g(z)$=$x\}=\pi \diamond 
g^{-1}(x)=\pi \diamond h(x)$.
We obtained $h(\pi \cdot x)=\pi \diamond h(x)$ for all $\pi \in 
Fix(supp(f) \cup supp(g) \cup supp(T))$ and all $x \in B$. According to 
Proposition~\ref{2.18'}, we get that $h$ is finitely supported. 
Furthermore, we also have that $supp(h) \subseteq supp(f) \cup supp(g) 
\cup supp(T) \overset{Claim\; 5}{\subseteq} supp(f) \cup supp(g) \cup 
supp(S)$$ \overset{Claim \; 4}{\subseteq} supp(f) \cup supp(g) \cup 
supp(F) \overset{Claim \; 2}{\subseteq} supp(f)$ $ \cup supp(g)$.
 
\textbf{Claim 8:} $h$ is a bijective function. \\
First we prove that $h$ is injective. Let us suppose that $h(x)=h(y)$. We 
claim that either $x,y \in T$ or $x,y \in B-T$. Indeed, let us suppose that 
$x\in T$ and $y \notin T$ (the case $x \notin T$, $y \in T$ is similar). 
We have $h(x)=f(x)$ and $h(y)=g^{-1}(y)$. If we denote $g^{-1}(y)=z$, we 
have $g(z)=y$. However, we supposed that $y \in B-T$, and so there exists $u 
\in C-f(T)$ such that $y=g(u)$. Since $y=g(z)$, from the injectivity of $g$ 
we get $u=z$. This is a contradiction because $u \notin f(T)$, while $z=f(x) 
\in f(T)$.  Since we proved that both $x,y$ are contained either in $T$ or 
in $B-T$, the injectivity of $h$ follows from the injectivity of $f$ or $g$, 
respectively.

Now we prove that $h$ is surjective. Let $y \in C$ be arbitrarily chosen. If 
$y \in f(T)$, then there exists $z \in T$ such that $y=f(z)$, and so 
$y=h(z)$.  If $y \in C-f(T)$, and because $g(C-f(T))=B-T$, there exists $x 
\in B-T$ such that $g(y)=x$. Thus, $y \in g^{-1}(x)$. Since $g$ is 
injective, and so $g^{-1}(x)$ is a one-element set, we can say that 
$g^{-1}(x)=y$ with $x \in B-T$. Thus we have $y=h(x)$.

\begin{lemma} \label{lem2}
Let $(B, \cdot)$ and \; $(C, \diamond)$ be two invariant sets (in 
particular, $B$ and $C$ could coincide), $B_{1}$ a finitely supported subset 
of $B$ and $C_{1}$ a finitely supported subset of $C$. If there exist a 
finitely supported injective mapping $f: B_{1} \to C_{1}$ and a finitely 
supported injective mapping $g: C_{1} \to B_{1}$, then there exists a 
finitely supported bijective mapping $h:B_{1} \to C_{1}$. Furthermore, 
$supp(h) \subseteq supp(f) \cup supp(g) \cup supp(B_{1}) \cup supp(C_{1})$.
\end{lemma}

\emph{Proof of Lemma \ref{lem2}.}
We follow the proof of Lemma \ref{lem1}. We define $F:\wp_{fs'}(B_{1}) 
\to \wp_{fs'}(B_{1})$ by $F(X)=B_{1}-g(C_{1}-f(X))$ for all $X \in 
\wp_{fs'}(B_{1})$, where $\wp_{fs'}(B_{1})$ is a finitely supported subset 
of the invariant set $\wp_{fs}(B)$ (supported by $supp(B_{1})$) defined by 
$\wp_{fs'}(B_{1})=\{X \in \wp_{fs}(B)\;|\;X \subseteq B_{1}\}$. As in the 
previous lemma, but using Proposition \ref{2.18'}, we get that $F$ is 
well-defined, i.e. for every $X \in \wp_{fs'}(B_{1})$ we have that $F(X)$ is 
supported by $supp(f) \cup supp(g) \cup supp(B_{1}) \cup supp(C_{1}) \cup 
supp(X)$ which means $F(X) \in \wp_{fs'}(B_{1})$. Moreover, $F$ is itself 
finitely supported (in the sense of Definition~\ref{2.10-1}) 
by $supp(f) \cup supp(g) \cup supp(B_{1}) \cup supp(C_{1})$. The set 
$S:=\{X\;|\; X \in \wp_{fs'}(B_{1}), \ X \subseteq F(X)\}$ is contained in 
$\wp_{fs'}(B_{1})$ and it is supported by $supp(F)$ as a subset of 
$\wp_{fs}(B)$. The set $T:=\underset{X \in S}{\cup}X \in \wp_{fs'}(B_{1})$ 
is finitely supported by $supp(S)$, and it is a fixed point of $F$.

As in the proof of Lemma \ref{lem1}, 
we define the bijection $h:B_{1} \to C_{1}$ by 
\[
h(x)=\left\{ \begin{array}{ll}
f(x), & \text{for}\: x \in T;\\
g^{-1}(x), & \text{for}\: x \in B_{1}-T.\end{array}\right. \]

According to Proposition \ref{2.18'}, we obtain that $h$ is finitely 
supported by $ supp(f) \cup supp(g) \cup supp(B_{1}) \cup supp(C_{1}) \cup 
supp(T)$, and $supp(h) \subseteq supp(f) \cup supp(g) \cup supp(B_{1}) 
\cup supp(C_{1})$. Thus, $h$ is the required finitely supported 
bijection between $B_{1}$ and $C_{1}$.
\medskip

The anti-symmetry of $\leq$ follows from Lemma \ref{lem1} and Lemma 
\ref{lem2} because FSM sets are actually finitely supported subsets of 
invariant sets. It is worth noting that $\leq^{*}$ is not anti-symmetric.

\begin{lemma} \label{lem3} There are two invariant sets $B$ and $C$ such that there exist both a 
finitely supported surjective mapping $f: C \to B$ and a finitely supported surjective 
mapping $g: B \to C$, but it does not exist a finitely supported bijective 
mapping $h:B \to C$.
\end{lemma}

\emph{Proof of Lemma \ref{lem3}.} 
Let us consider the invariant set $(A, \cdot)$ of atoms. The family $T_{fin}(A)=\{(x_{1}, \ldots, x_{m}) \subseteq (A \times \ldots \times A)\,|\,m \geq 0\}$ of all finite injective tuples 
from $A$ (including the empty tuple denoted by~$\bar{\emptyset}$) 
is an $S_{A}$-set with the $S_{A}$-action $\star:S_{A}\times T_{fin}(A) 
\rightarrow T_{fin}(A)$ defined by $\pi \star (x_{1}, \ldots, x_{m})=(\pi 
\cdot x_{1}, \ldots, \pi \cdot x_{m})$ for all $ (x_{1}, \ldots, x_{m}) \in 
T_{fin}(A)$ and all $\pi \in S_{A}$. Since $A$ is an invariant 
set, we have that $T_{fin}(A)$ is an invariant set. Whenever $X$ is an 
invariant set, we have that each injective tuple $(x_{1}, \ldots, x_{m})$ of 
elements belonging to $X$ is finitely supported, and, furthermore,  $supp(x_{1}, \ldots, 
x_{m})=supp(x_{1}) \cup \ldots \cup supp(x_{m})$. Particularly, we obtain that $supp(a_{1}, \ldots, a_{m})=\{a_{1}, \ldots, a_{m}\}$, for any injective tuple of atoms 
$(a_{1}, \ldots, a_{m})$ (similarly as in Proposition 2.2 from \cite{book}).

Since $supp(\bar{\emptyset})=\emptyset$, it follows that 
$T^{*}_{fin}(A)=T_{fin}(A)\setminus \bar{\emptyset}$ is an equivariant subset of 
$T_{fin}(A)$, and is itself an invariant set. Let us fix an atom $a \in A$. We define $f:T_{fin}(A) \to T_{fin}(A)\setminus \bar{\emptyset}$ by 
\[
f(y)=\left\{ \begin{array}{ll}
y, & \text{if}\: \text{$y$ is an injective non-empty tuple};\\
(a), & \text{if}\: \text{$y=\bar{\emptyset}$}\: .\end{array}\right. \]
Clearly, $f$ is surjective. We claim that $f$ is supported by $supp(a)$. Let 
$\pi \in Fix(supp(a))$, i.e. $a=\pi(a)=\pi \star(a)$. If $y$ is a non-empty 
tuple of atoms, we obviously have $f(\pi \star y)=\pi \star y=\pi \star 
f(y)$. If $y=\bar{\emptyset}$, we have $\pi \star y=\bar{\emptyset}$, and so 
$f(\pi \star y)=(a)=\pi(a)=\pi \star f(y)$. Thus, $f(\pi \star y)=\pi \star 
f(y)$ for all $y \in T_{fin}(A)$. According to Proposition \ref{2.18'}, we 
have that $f$ is finitely supported.

We define an equivariant surjective function $g:T_{fin}(A)\setminus 
\bar{\emptyset} \to T_{fin}(A)$ by
\[
g(y)=\left\{ \begin{array}{ll}
\bar{\emptyset}, & \text{if}\: \text{$y$ is a tuple with exactly one element};\\
y', & \text{otherwise}\: ; \end{array}\right. \]
where $y'$ is a new tuple formed by deleting the first element in tuple~$y$ 
(the first position in a finite injective tuple exists without 
requiring any form of choice).
Clearly, $g$ is surjective. Indeed, $\bar{\emptyset}=g((a))$ for some 
one-element tuple $(a)$ ($A$ is non-empty, and so it has at least one atom). 
For a fixed finite injective non-empty $m$-tuple $y$, we have that $y$ can 
be seen as being ``contained" in an injective $(m+1)$-tuple $z$ of form 
$(b,y)$ (whose first element is a certain atom~$b$, and the following 
elements are precisely the elements of~$y$). The related atom~$b$ exists 
because $y$ is finite, while $A$ is infinite (generally, we can always find 
an atom $b \notin supp(y)=\{y\}$ according to the finite support requirement 
in FSM - more details in Section 2.9 of \cite{book}). We get $y=g (z)$. For proving the surjectivity of $g$ we do not need to `choose'  a precise such an  element $b$ (we do not need to define an inverse function for $g$); it is sufficient to ascertain that $g(b,y)=y$ for every $b \in A\setminus\{y\}$ and $A\setminus\{y\}$ is non-empty (the axiom of choice is not required because for proving only the surjectivity of $g$ we do not involve the construction of a system of representatives for the family $(g^{-1}(y))_{y \in T_{fin}(A)}$).

We claim now that $g$ is equivariant. Let $(x)$ be a one-element tuple 
from~$A$ and $\pi$ an arbitrary permutation from~$S_{A}$. We have that 
$\pi \star (x)=(\pi(x))$ is a one-element tuple from $A$, and so $g(\pi 
\star (x))=\bar{\emptyset} = \pi \star \bar{\emptyset} = \pi \star 
g((x))$. Now, let us consider $(x_{1}, \ldots, x_{m}) \in T_{fin}(A), m 
\geq 2$ and $\pi \in S_{A}$. We have $ g(\pi \star (x_{1}, \ldots, 
x_{m}))=g((\pi \cdot x_{1}, \ldots, \pi \cdot x_{m}))=g((\pi (x_{1}), 
\ldots, \pi (x_{m})))=(\pi (x_{2}), \ldots, \pi (x_{m}))=\pi \star (x_{2}, 
\ldots, x_{m})=\pi \star g(x_{1}, \ldots, x_{m})$. According to 
Proposition \ref{2.18'}, we have that $g$ is empty-supported (equivariant).

We prove by contradiction that there could not exist a finitely supported 
injective $h: T_{fin}(A) \to T_{fin}(A)\setminus \bar{\emptyset}$.  
Let us suppose there is a finitely supported injection $h:T_{fin}(A) 
\rightarrow T_{fin}(A)\setminus \bar{\emptyset}$. We have $\bar{\emptyset} 
\notin Im(h)$ because $Im(h) \subseteq T_{fin}(A) \setminus 
\bar{\emptyset}$. We can form an infinite sequence $\mathcal{F}$ which has 
the first term $y_{0}=\bar{\emptyset}$, and the general term 
$y_{n+1}=h(y_{n})$ for all $n \in \mathbb{N}$. Since $\bar{\emptyset}\notin 
Im(h)$, it follows that $\bar{\emptyset} \neq h(\bar{\emptyset})$. Since $h$ 
is injective and $\bar{\emptyset} \notin Im(h)$, we obtain by induction that 
$h^{n}(\bar{\emptyset}) \neq h^{m}(\bar{\emptyset})$ for all $n,m \in 
\mathbb{N}$ with $n \neq m$.

We prove now that for each $n \in \mathbb{N}$ we have that $y_{n+1}$ is 
supported by $supp(h)\cup supp(y_{n})$. Let $\pi \in Fix(supp(h)\cup 
supp(y_{n}))$. According to Proposition~\ref{2.18'}, because $\pi \in 
Fix(supp(h))$ we have $h(\pi \star y_{n})=\pi \star h(y_{n})$. Since $\pi 
\in Fix(supp(y_{n}))$ we have $\pi \star y_{n}=y_{n}$, and so $h(y_{n})=\pi 
\star h(y_{n})$. Thus, $\pi \star y_{n+1}= \pi \star h(y_{n}) = h(y_{n}) = 
y_{n+1}$. Furthermore, because $supp(y_{n+1})$ is the least set supporting 
$y_{n+1}$, we have $supp(y_{n+1}) \subseteq supp(h)\cup supp(y_{n})$ for all 
$n \in \mathbb{N}$.
Since each $y_{n}$ is a finite injective tuple of atoms, it follows that 
$supp(y_{n})=\{y_{n}\}$ for all $n \in \mathbb{N}$ (where by $\{y_{n}\}$ we 
denoted the set of atoms forming $y_{n}$). We get $\{y_{n+1}\} = 
supp(y_{n+1}) \subseteq supp(h)\cup supp(y_{n}) = supp(h) \cup \{y_{n}\}$. 
By repeatedly applying this result, we get $\{y_{n}\} \subseteq 
supp(h) \cup \{y_{0}\}= supp(h) \cup \emptyset =supp(h)$ for all 
$n\in\mathbb{N}$. Since $supp(h)$ has only a finite number of subsets, we 
contradict the statement that the infinite sequence $(y_{n})_{n}$ never 
repeats. Thus, there does not exist a finitely supported bijection between 
$T_{fin}(A)\setminus \bar{\emptyset}$ and $T_{fin}(A)$. 

\item $\leq$ and $\leq^{\star}$ are not total. 

We prove that whenever $X$ is an infinite ordinary (non-atomic) ZF-set, for any finitely supported function $f : A \to X$
and any finitely supported function $g : X \to A$, $Im(f)$ and $Im(g)$ are finite. As a direct consequence there are no finitely supported injective mappings and no finitely supported surjective mappings between $A$ and $X$.

Let us consider a finitely supported mapping
$f:A \to X$.  Let let us fix an element $b\in A$ with 
$b\notin supp(f)$. Let $c$ be an arbitrary element from 
$A\setminus supp(f)$. Since $b\notin supp(f)$, we have that $(b\, c)$ 
fixes every element from $supp(f)$, i.e. $(b\, c)\in Fix(supp(f))$. 
However, $supp(f)$ supports~$f$, and so, by Proposition \ref{2.18'}, we 
have $f((b\, c)(a))=(b\,c) \diamond f(a)=f(a)$ for all $a\in A$. In 
particular, $f(c)=f((b\,c)(b))=f(b)$. Since~$c$ has been chosen 
arbitrarily from $A \setminus supp(f)$, it follows that $f(c)=f(b)$, for 
all $c \in A \setminus supp(f)$. If $supp(f)=\{a_{1}, \ldots, a_{n}\}$, 
then $Im(f) = \{f(a_{1})\}\cup \ldots \cup \{f(a_{n})\} \cup \{f(b)\}$. Thus, $Im(f)$ 
is finite (because it is a finite union of singletons). 

Let $g : X \to A$ be a finitely supported function. Assume by contradiction that $Im(g)$ is infinite. Pick any atom $a \in Im(g) \setminus supp(g)$ (such an atom exists because $supp(g)$ is finite). There
exists an $x \in X$ such that $g(x) = a$. Now pick any atom $b \in Im(g) \setminus (supp(g) \cup \{a\})$, The transposition $(a\,b)$ fixes $supp(g)$ pointwise, and so $g(x)=g((a\,b) \diamond x)=(a\,b) \cdot g(x) =(a\, b)(a)= b$, contradicting the fact that
$g$ is a function. Thus, $Im(g)$ is finite. 

\end{itemize}
\end{proof}

\begin{corollary}\label{corcor}
There exist two invariant sets $B$ and $C$ such that there is a finitely 
supported bijection between $\wp_{fs}(B)$ and $\wp_{fs}(C)$, but there is 
no finitely supported bijection between $B$ and $C$.
\end{corollary}
\begin{proof} Firstly we prove the following lemma.
\begin{lemma} \label{lemlem} Let $X$ and $Y$ be two FSM sets and $f:X \to Y$ a finitely supported surjective function. Then the mapping $g:\wp_{fs}(Y) \to \wp_{fs}(X)$ defined by $g(V)=f^{-1}(V)$ for all $V \in \wp_{fs}(Y)$ is well defined, injective and finitely supported by $supp(f) \cup supp(X) \cup supp(Y)$. 
\end{lemma}

\emph{Proof of Lemma \ref{lemlem}.} Let $V$ be an arbitrary element from $\wp_{fs}(Y)$.  We claim that $f^{-1}(V) \in 
\wp_{fs}(X)$. Indeed we prove that the set $f^{-1}(V)$ is supported by $supp(f) \cup 
supp(V) \cup supp(X) \cup supp(Y)$. Let $\pi \in Fix(supp(f) \cup supp(V) \cup supp(X) \cup supp(Y))$, and $x \in f^{-1}(V) $. 
This means $f(x) \in V$. According to Proposition \ref{2.18'}, and because $\pi$ fixes $supp(f)$ pointwise and $supp(f)$ supports $f$, we have $f(\pi \cdot x)= \pi \cdot f(x) \in \pi \star 
V = V$, and so $\pi \cdot x \in f^{-1}(V)$ (we denoted the actions on $X$ and $Y$ generically by $\cdot$, and the actions on their powersets by $\star$). Therefore, $f^{-1}(V)$ is 
finitely supported, and so the function $g$ is well defined. We claim that $g$ is supported by 
$supp(f) \cup supp(X) \cup supp(Y)$. Let $\pi \in Fix(supp(f) \cup supp(X) \cup supp(Y))$.  For any arbitrary $V \in \wp_{fs}(Y)$ we get $\pi \star V \in \wp_{fs}(Y)$ and $\pi \star g(V) \in \wp_{fs}(X)$, and by Proposition \ref{2.18'} we have that $\pi^{-1} \in Fix(supp(f))$, 
and so $f(\pi^{-1} \cdot x)=\pi^{-1} \cdot f(x)$ for all $x\in X$. For any 
arbitrary $V \in \wp_{fs}(Y)$, we have that $z \in g(\pi \star V) = 
f^{-1}(\pi \star V) \Leftrightarrow f(z) \in \pi \star V \Leftrightarrow 
\pi^{-1} \cdot f(z) \in V \Leftrightarrow f(\pi^{-1} \cdot z) \in V 
\Leftrightarrow \pi^{-1} \cdot z \in f^{-1}(V) \Leftrightarrow z \in \pi 
\star f^{-1}(V)=\pi \star g(V)$. If follows that $g(\pi \star V)=\pi \star 
g(V)$ for all $V \in \wp_{fs}(Y)$, and so $g$ is finitely supported. 
Moreover, because $f$ is surjective, a simple calculation shows us that $g$ 
is injective. Indeed, let us suppose that $g(U)=g(V)$ for some $U,V \in 
\wp_{fs}(Y)$. We have $f^{-1}(U) = f^{-1}(V)$, and so $f(f^{-1}(U)) = 
f(f^{-1}(V))$. Since $f$ is surjective, we get $U = f(f^{-1}(U)) = 
f(f^{-1}(V)) = V$.

We start the proof of Corollary \ref{corcor}.
As in Lemma \ref{lem3}, we consider the sets $B=T_{fin}(A)\setminus \bar{\emptyset}$ and $C=T_{fin}(A)$. According to Lemma \ref{lem3} there exists a finitely supported surjective function $f:C \to B$ and a finitely supported (equivariant) 
surjection $g:B \to C$. Thus, according to Lemma \ref{lemlem}, there exist a finitely supported injective function $f':\wp_{fs}(B) \to \wp_{fs}(C)$ and  a finitely supported injective function $g':\wp_{fs}(C) \to \wp_{fs}(B)$. According to Lemma \ref{lem1}, there is a finitely supported bijection between $\wp_{fs}(B)$ and $\wp_{fs}(C)$. However, we proved in Lemma \ref{lem3} that there is no finitely supported bijection between $B=T_{fin}(A)\setminus \bar{\emptyset}$ and $C=T_{fin}(A)$.
\end{proof}

The following result communicated by Levy in 1965 for non-atomic ZF sets can be reformulated in the world of finitely supported atomic structures. 
\begin{corollary}Let $X$ and $Y$ be two invariant sets with the property that whenever $|2^{X}_{fs}|=|2^{Y}_{fs}|$ we have $|X|=|Y|$. If $|X|\leq^{\star}|Y|$ and $|Y|\leq^{\star}|X|$, then $|X|=|Y|$.
\end{corollary}
\begin{proof}According to the hypothesis and to Lemma \ref{lemlem} there exist two finitely supported injective functions $f:\wp_{fs}(Y) \to \wp_{fs}(X)$  and $g:\wp_{fs}(X) \to \wp_{fs}(Y)$. According to Lemma \ref{lem1}, there is a bijective mapping $h:\wp_{fs}(X) \to \wp_{fs}(Y)$ . According to Theorem \ref{comp}, we get $|2^{X}_{fs}|=|2^{Y}_{fs}|$, and so we get $|X|=|Y|$.
\end{proof}

\begin{proposition}[Cantor] \label{Cantor} Let $X$ be a finitely supported subset of an invariant set $(Y, \cdot)$. Then $|X| \lneq |\wp_{fs}(X)|$ and $|X| \lneq^{*} |\wp_{fs}(X)|$
\end{proposition}
\begin{proof}First we prove that there is no finitely supported bijection between $X$ and $\wp_{fs}(X)$, and so their cardinalities cannot be equal. Assume, by contradiction, that there is a finitely supported surjective mapping $f: X \to \wp_{fs}(X)$. Let us consider $Z=\{x \in X\,|\,x \notin f(x)\}$. We claim that $supp(X) \cup supp(f)$ supports $Z$. Let $\pi \in Fix(supp(X) \cup supp(f))$. Let $x \in Z$. Then $\pi \cdot x \in X$ and $\pi \cdot x \notin \pi \star f(x)=f(\pi \cdot x)$. Thus, $\pi \cdot x \in Z$, and so $Z \in \wp_{fs}(X)$. Therefore, since $f$ is surjective there is $x_{0} \in X$ such that $f(x_{0})=Z$. However, from the definition of $Z$ we have $x_{0} \in Z$ if and only if $x_{0}\notin f(x_{0})=Z$, which is a contradiction. 

Now, it is clear that the mapping $i: X \to \wp_{fs}(X)$ defined by $i(x)=\{x\}$ is  injective and supported by $supp(X)$. Thus, $|X| \lneq |\wp_{fs}(X)|$.
Let us fix an atom $y \in X$. We define $s:\wp_{fs}(X) \to X$ by 
\[
s(U)=\left\{ \begin{array}{ll}
u, & \text{if}\: \text{$U$ is an one-element set \{u\} };\\
y, & \text{if}\: \text{$U$ has more than one element}\: .\end{array}\right. \]
Clearly, $s$ is surjective. We claim that $s$ is supported by $supp(y) \cup supp(X)$. Let 
$\pi \in Fix(supp(y) \cup supp(X))$. Thus, $y=\pi \cdot y$. If $U$ is of form $U=\{u\}$, we obviously have $s(\pi \star U)=s(\{\pi \cdot u\})=\pi \cdot u=\pi \cdot 
s(U)$. If $U$ has more than one element, then $\pi \star U$ has more than one element, and we have $s(\pi \star U)=y=\pi \cdot y=\pi \cdot s(U)$. Thus,  $\pi \star U \in \wp_{fs}(X)$, $\pi \cdot s(U) \in X$, and $s(\pi \star U)=\pi \cdot 
s(U)$ for all $U \in \wp_{fs}(X)$ . According to Proposition \ref{2.18'}, we 
have that $s$ is finitely supported. Therefore, $|X| \lneq^{*} |\wp_{fs}(X)|$.
\end{proof}

In Proposition \ref{Cantor} we used a technique for constructing a surjection starting from an injection defined in the opposite way, that can be generalized as follows. 

\begin{proposition} \label{pco2}Let $X$ and $Y$ be finitely supported subsets of an invariant set $U$. If  If $|X| \leq |Y|$, then  $|X| \leq ^{\star} |Y|$. The converse is not valid. However, if $|X| \leq ^{\star} |Y|$, then $|X| \leq |\wp_{fs}(Y)|$.
\end{proposition}

\begin{proof} Suppose there exists a finitely supported injective mapping $f: X \to Y$. We consider the case $Y \neq \emptyset$ (otherwise, the result follows trivially).  Fix $x_{0} \in X$. Define the mapping $f':Y \to X$ by \[
f'(y)=\left\{ \begin{array}{ll}
f^{-1}(y), & \text{if}\: \text{$y \in Im(f)$ };\\
x_{0}, & \text{if}\: \text{$y \notin Im(f)$}\: .\end{array}\right. \]

Since $f$ is injective, it follows that $f^{-1}(y)$ is an one-element set 
for each $y \in Im(f)$, and so $f'$ is a function. Clearly,~$f'$ is 
surjective. We claim that $f'$ is supported by the set $supp(f) \cup 
supp(x_{0}) \cup supp(X) \cup supp(Y)$. Indeed, let us consider $\pi \in 
Fix(supp(f) \cup supp(x_{0}) \cup supp(X) \cup supp(Y))$. Whenever $y \in 
Im(f)$ we have $y=f(z)$ for some $z \in X$ and $\pi \cdot y=\pi \cdot 
f(z)=f(\pi \cdot z) \in Im(f)$, which means $Im(f)$ is finitely supported 
by $supp(f)$. Consider an arbitrary $y_{0} \in Im(f)$, and thus $\pi \cdot 
y_{0} \in Im(f)$. Then $f'(y_{0})= f^{-1}(y_{0})=z_{0}$ with 
$f(z_{0})=y_{0}$, and so $f(\pi \cdot z_{0})= \pi \cdot f(z_{0})=\pi \cdot 
y_{0}$, which means $f'(\pi \cdot y_{0})=f^{-1}(\pi \cdot y_{0})=\pi \cdot 
z_{0}=\pi \cdot f^{-1}(y_{0})=\pi \cdot f'(y_{0})$. Now, for $y \notin 
Im(f)$ we have $\pi \cdot y \notin Im(f)$, which means $f'(\pi \cdot 
y)=x_{0}=\pi \cdot x_{0}=\pi \cdot f(y)$ since $\pi$ fixes $x_{0}$ 
pointwise. Thus, $|X| \leq ^{\star} |Y|$. Conversely, from the proof of 
Lemma \ref{lem3}, we know that there is a finitely supported surjection 
$g:T_{fin}(A)\setminus \bar{\emptyset} \to T_{fin}(A)$, but there does not 
exist a finitely supported injection $h: T_{fin}(A) \to 
T_{fin}(A)\setminus \bar{\emptyset}$.

Assume now there is a finitely supported surjective mapping $f:Y \to X$. We proceed similarly as in the proof of Lemma \ref{lemlem}. Fix $x \in X$. Then $f^{-1}(\{x\})$ is supported by $supp(f) \cup 
supp(x) \cup supp(X)$. Indeed, let $\pi \in Fix(supp(f) \cup supp(x) \cup supp(X))$, and $y \in f^{-1}(\{x\})$. 
This means $f(y)=x$. According to Proposition \ref{2.18'}, we have $f(\pi \cdot y)= \pi \cdot f(y)= \pi \cdot x=x$, and so $\pi \cdot y \in f^{-1}(\{x\})$.  Define $g:X \to \wp_{fs}(Y)$ by $g(x)=f^{-1}(\{x\})$. We claim that $g$ is supported by 
$supp(f) \cup supp(X)$. Let $\pi \in Fix(supp(f) \cup supp(X))$. For any 
arbitrary $x \in X$, we have that $z \in g(\pi \cdot x) = 
f^{-1}(\{\pi \cdot x\}) \Leftrightarrow f(z)=\pi \cdot x \Leftrightarrow 
\pi^{-1} \cdot f(z)=x \Leftrightarrow f(\pi^{-1} \cdot z) =x
\Leftrightarrow \pi^{-1} \cdot z \in f^{-1}(\{x\}) \Leftrightarrow z \in \pi 
\star f^{-1}(\{x\})=\pi \star g(x)$. From Proposition \ref{2.18'} it follows that $g$ is finitely supported. Since $g$ is also injective, we get $|X| \leq |\wp_{fs}(Y)|$.
\end{proof}

\begin{proposition} \label{pco1} Let $X,Y,Z$ be finitely supported subsets of an invariant set $U$. The following properties hold.
\end{proposition}
\begin{enumerate}
\item If $|X| \leq |Y|$, then $|X|+|Z| \leq |Y|+|Z|$; 
\item If $|X| \leq |Y|$, then $|X| \cdot |Z| \leq |Y| \cdot |Z|$; 
\item If $|X| \leq |Y|$, then $|X^{Z}_{fs}| \leq |Y^{Z}_{fs}|$; 
\item  If $|X| \leq |Y|$ and $Z\neq \emptyset$, then $|Z^{X}_{fs}| \leq |Z^{Y}_{fs}|$;
\item $|X|+|Y| \leq |X|\cdot|Y|$ whenever both $X$ and $Y$ have more than two elements.

\end{enumerate}

\begin{proof}
1. Suppose there is a finitely supported injective $f: X \to Y$, and 
define the injection $g: X+Z\to Y+Z$ by
\[
g(u)=\left\{ \begin{array}{ll}
(0,f(x)), & \text{if}\: u=(0,x)\: \text{with}\: x \in X;\\
(1,z), & \text{if}\:  u=(1,z)\: \text{with}\: z \in Z .\end{array}\right. 
\]
Since $f$ is finitely supported we have that $f(\pi \cdot x)=\pi \cdot 
f(x)$ for all $x \in X$ and $\pi \in Fix(supp(f))$. By using 
Proposition~\ref{2.18'}, i.e verifying that $g(\pi \star u)=\pi \star g(u)$ for all $u 
\in X+Z$ and all $\pi \in Fix(supp(f) \cup supp(X) \cup supp(Y) \cup 
supp(Z))$, we have that $g$ is also finitely supported.

2. Suppose there exists a finitely supported injective mapping $f: X \to Y$. Define the injection $g: X\times Z\to Y \times Z$ by $g((x,z))=(f(x),z)$ for all $(x,z) \in X \times Z$. Clearly $g$ is injective. Since $f$ is finitely supported we have that $f(\pi \cdot x)=\pi \cdot f(x)$ for all $x \in X$ and $\pi \in Fix(supp(f))$, and so $g(\pi \otimes(x,z))=g((\pi \cdot x, \pi \cdot z))=(f(\pi \cdot x),\pi \cdot z)=(\pi \cdot f(x),\pi \cdot z)=\pi \otimes g((x,z))$ for all $(x,z) \in X \times Z$ and $\pi \in Fix(supp(f) \cup supp(X) \cup supp(Y) \cup supp(Z))$, which means $g$ is supported by $supp(f) \cup supp(X) \cup supp(Y) \cup supp(Z)$.

3.  Suppose there exists a finitely supported injective mapping $f: X \to Y$. Define $g:X^{Z}_{fs} \to Y^{Z}_{fs}$ by $g(h)=f \circ h$. We have that $g$ is injective and for any $\pi \in Fix(supp(f))$ we have $\pi \widetilde{\star} f=f$, and so $g(\pi \widetilde{\star} h)=f \circ (\pi \widetilde{\star} h)= (\pi \widetilde{\star} f) \circ (\pi \widetilde{\star} h)=\pi \widetilde{\star} (f \circ h)=\pi \widetilde{\star} g(h)$ for all $h \in X^{Z}_{fs}$. We used the relation $(\pi \widetilde{\star} f) \circ (\pi \widetilde{\star} h)=\pi \widetilde{\star} (f \circ h)$ for all $\pi \in S_{A}$. This can be proved as follows. Fix $x\in Z$, we have $(\pi\widetilde{\star}(f\circ h))(x)=\pi\cdot(f(h(\pi^{-1}\cdot x)))$.
Also, if we denote $(\pi\widetilde{\star} h)(x)=y$ we have $y=\pi\cdot(h(\pi^{-1}\cdot x))$
and $((\pi\widetilde{\star} f)\circ(\pi\widetilde{\star} h))(x)=(\pi\widetilde{\star} f)(y)=\pi\cdot(f(\pi^{-1}\cdot y))=\pi\cdot(f((\pi^{-1}\circ\pi)\cdot h(\pi^{-1}\cdot x)))=\pi\cdot(f(h(\pi^{-1}\cdot x)))$. We finally obtain that $g$ is supported by $supp(f) \cup supp(X) \cup supp(Y) \cup supp(Z)$.

4.  Suppose there exists a finitely supported injective mapping $f: X \to Y$. According to Proposition \ref{pco2}, there is a finitely supported surjective mapping $f':Y \to X$. Define the injective mapping $g:Z^{X}_{fs} \to Z^{Y}_{fs}$ by $g(h)=h \circ f'$. As in item 3 one can prove that $g$ is finitely supported by $supp(f') \cup supp(X) \cup supp(Y) \cup supp(Z)$.

5. Fix $x_{0}, x_{1} \in X$ with $x_{0} \neq x_{1}$ and $y_{0}, y_{1} \in Y$ with $y_{0}\neq y_{1}$. Define the injection $g: X+Y\to X \times Y$ by 
\[
g(u)=\left\{ \begin{array}{ll}
(x, y_{0}), & \text{if}\: u=(0,x)\: \text{with}\: x \in X, x \neq x_{0};\\
(x_{0},y), & \text{if}\:  u=(1,y)\: \text{with}\: y \in Y;\\
(x_{1}, y_{1}), & \text{if}\: u=(0,x_{0})\end{array}\right. 
\]
It follows that $g$ is supported by $supp(x_{0}) \cup supp(y_{0}) \cup supp(x_{1}) \cup supp(y_{1}) \cup supp(X) \cup supp(Y)$, and $g$ is injective. 
\end{proof}

\begin{theorem} \label{comp} 
Let $(X, \cdot)$ be a finitely supported subset of an invariant set 
$(Z, \cdot)$. There exists a one-to-one mapping from $\wp_{fs}(X)$ onto 
$\{0,1\}^{X}_{fs}$ which is finitely supported by $supp(X)$, where 
$\wp_{fs}(X)$ is considered the family of those finitely supported subsets 
of $Z$ contained in $X$.
\end{theorem}

\begin{proof}
Let $Y$ be a finitely supported subset of $Z$ contained in $X$, and 
$\varphi_{Y}$ be the characteristic function on $Y$, i.e. $\varphi_{Y}:X \to 
\{0,1\}$ is defined by 
\begin{center}
$\varphi_{Y}(x)\overset{def}{=}\left\{ 
\begin{array}{ll} 1 & \text{for}\: x \in Y\\ 0 & \text{for}\: x\in X 
\setminus Y \end{array}\right.$.\\ 
\end{center} 
We prove that $\varphi_{Y}$ is a finitely supported function from $X$ to 
$\{0,1\}$ (according to Proposition \ref{p1}, $\{0,1\}$ is a trivial invariant 
set), and the mapping $Y \mapsto \varphi_{Y}$ defined on $\wp_{fs}(X)$ 
is also finitely supported in the sense of Definition~\ref{2.10-1}.

First we prove that $\varphi_{Y}$ is supported by $supp(Y) \cup supp(X)$. 
Let $\pi \in Fix(supp(Y) \cup supp(X))$. Thus $\pi \star Y=Y$ (where $\star$ represents the canonical permutation action on $\wp(Z)$), and so $\pi 
\cdot x \in Y$ if and only if $x \in Y$. Since we additionally have $\pi 
\star X=X$, we obtain $\pi \cdot x \in X \setminus Y$ if and only if $x 
\in X \setminus Y$. Thus, $\varphi_{Y}(\pi\cdot x)=\varphi_{Y}(x)$ for all 
$x \in X$. Furthermore, because $\pi$ fixes $supp(X)$ pointwise we have $\pi 
\cdot x \in X$ for all $x \in X$, and from Proposition \ref{2.18'} we get 
that $\varphi_{Y}$ is supported by $supp(Y) \cup supp(X)$.

We remark that $\{0,1\}^{X}_{fs}$ is a finitely supported 
subset of the set $(\wp_{fs}(Z \times \{0,1\}), \widetilde{\star})$. Let $\pi \in 
Fix(supp(X))$ and $f:X \to \{0,1\}$ finitely supported. We have 
$\pi\widetilde{\star} f=\{(\pi\cdot x,$ $\pi\diamond y)\,|\,(x,y)\in 
f\}=\{(\pi\cdot x,$ $y)\,|\,(x,y)\in f\}$ because $\diamond$ is the trivial 
action on $\{0,1\}$. Thus, $\pi\widetilde{\star}f $ is a function with the 
domain $\pi\star X=X$ which is finitely supported as an element of $(\wp (Z 
\times \{0,1\}), \widetilde{\star})$ according to Proposition~\ref{2.15}. 
Moreover, $(\pi \widetilde{\star} f)(\pi\cdot x)= f(x)$ for all $x \in X$ (1).

According to Proposition \ref{2.18'}, to prove that the function $g:=Y 
\mapsto \varphi_{Y}$ defined on $\wp_{fs}(X)$ (with the codomain contained 
in $\{0,1\}^{X}_{fs}$) is supported by $supp(X)$, we have to prove that $\pi 
\widetilde{\star}g(Y)=g(\pi \star Y)$ for all $\pi \in Fix(supp(X))$ and all 
$Y \in \wp_{fs}(X)$ (where $\widetilde{\star}$ symbolizes the induced 
$S_{A}$-action on $\{0,1\}^{X}_{fs}$). This means that we need to verify the 
relation $\pi \widetilde{\star} \varphi_{Y} = \varphi_{\pi \star Y}$ for all 
$\pi \in Fix(supp(X))$ and all $Y \in \wp_{fs}(X)$. Let us consider $\pi \in 
Fix(supp(X))$ (which means $\pi \cdot x \in X$ for all $x \in X$) and $Y \in 
\wp_{fs}(X)$. For any $x \in X$, we know that $x \in \pi \star Y$ if and 
only if $\pi^{-1} \cdot x \in Y$. Thus, $\varphi_{Y}(\pi^{-1} \cdot x 
)=\varphi_{\pi \star Y}(x)$ for all $x \in X$, and so $(\pi 
\widetilde{\star} \varphi_{Y})(x) \overset{(1)}{=}\varphi_{Y}(\pi^{-1} \cdot 
x )=\varphi_{\pi \star Y}(x)$ for all $x\in X$. Moreover, from Proposition 
\ref{2.15}, $\pi \star Y$ is a finitely supported subset of $Z$ contained in 
$\pi \star X=X$, and $\{0,1\}^{X}_{fs}$ can be represented as a finitely 
supported subset of $\wp_{fs}(Z \times \{0,1\})$ (supported by $supp(X)$). 
According to Proposition \ref{2.18'} we have that~$g$ is a finitely 
supported function from $\wp_{fs}(X)$ to $\{0,1\}^{X}_{fs}$.

Obviously, $g$ is one-to-one. Now we prove that $g$ is onto. 
Let us consider an arbitrary finitely supported function $f: X \to \{0,1\}$. 
Let $Y_{f}\overset{def}{=}\{x \in X\;|\; f(x)=1\}$. We claim that $Y_{f} \in 
\wp_{fs}(X)$. Let $\pi \in Fix(supp(f))$. According to Proposition 
\ref{2.18'} we have $\pi \cdot x \in X$ and $f(\pi \cdot x)=f(x)$ for all 
$x \in X$. Thus, for each $x \in Y_{f}$, we have $\pi \cdot x \in Y_{f}$. 
Therefore $\pi \star Y_{f}=Y_{f}$, and so $Y_{f}$ is finitely supported by 
$supp(f)$ as a subset of $Z$, and it is contained in $X$. 
A simple calculation show us that $g(Y_{f})=f$, and so $g$ is onto. 
\end{proof}

One can easy verify that the properties of $\leq$ presented in Proposition \ref{pco1} (1), (2) and (4) also hold for $\leq^{\star}$. We left the details to the reader.   

\begin{theorem} \label{cardord1} There exists an invariant set $X$ (particularly the set $A$ of atoms) having the following properties.
\begin{enumerate}
\item $|X \times X| \nleq^{*} |\wp_{fs}(X)|$;
\item $|X \times X| \nleq |\wp_{fs}(X)|$;
\item $|X \times X| \nleq^{*} |X|$;
\item $|X \times X| \nleq |X|$;
\item For each $n \in \mathbb{N}, n \geq 2$ we have $|X| \lneq |\wp_{n}(X)| \lneq |\wp_{fs}(X)|$, where $\wp_{n}(X)$ is the family of all $n$-sized subsets of $X$;
\item For each $n \in \mathbb{N}$ we have $|X| \lneq^{*} |\wp_{n}(X)| \lneq^{*} |\wp_{fs}(X)|$;
\item $|X| \lneq |\wp_{fin}(X)| \lneq |\wp_{fs}(X)|$;
\item $|X| \lneq^{*} |\wp_{fin}(X)| \lneq^{*} |\wp_{fs}(X)|$;
\item $|\wp_{fs}(X) \times \wp_{fs}(X)| \nleq^{*} |\wp_{fs}(X)|$;
\item $|\wp_{fs}(X) \times \wp_{fs}(X)| \nleq |\wp_{fs}(X)|$;
\item$|X+X| \lneq^{*} |X\times X|$;
\item $|X+X| \lneq |X\times X|$.

\end{enumerate}
\end{theorem}

\begin{proof}

1. We prove that that there does not exist a finitely supported surjective mapping  $f: \wp_{fs}(A) \to A \times A$.  Suppose, by contradiction, that there is a finitely supported surjective mapping  $f: \wp_{fs}(A) \to A \times A$.  Let us consider two atoms $a,b\notin supp(f)$ with $a 
\neq b$. These atoms exist because $A$ is infinite, while $supp(f) \subseteq 
A$ is finite.  It follows that the transposition $(a\, b)$ fixes each 
element from $supp(f)$, i.e. $(a\, b) \in Fix(supp(f))$. Since $f$ is surjective, it follows that there exists an element $X \in \wp_{fs}(A)$ 
such that $f(X)=(a,b)$. Since $supp(f)$ supports $f$ and $(a\, b) \in 
Fix(supp(f))$, from Proposition \ref{2.18'} we have $f((a\,b) \star X)=(a\,b) 
\otimes f(X)=(a\,b) \otimes (a,b)=((a\,b)(a), (a\,b)(b))=(b,a)$. Due to the functionality of $f$ we should have $(a\,b) \star X \neq X$. Otherwise, we would obtain $(a,b)=(b,a)$.

We claim that if both $a,b \in supp(X)$, then $(a\,b)\star X=X$. Indeed, suppose $a,b \in supp(X)$. Since $X$ is a finitely supported subset of $A$, then $X$ is either finite or cofinite. If $X$ is finite, then $supp(X)=X$, and so $a,b \in X$. Moreover, $(a\, 
b)(a)=b$, $(a\, b)(b)=a$, and $(a\, b)(c)=c$ for all $c \in X$ with $c\neq a,b$. Therefore, $(a\,b) \star X=\{(a\,b)(x)\,|\,x \in X\}=\{(a\,b)(a)\} \cup \{(a\,b)(b)\} \cup \{(a\,b)(c)\,|\,c \in X \setminus\{a,b\}\}=\{b\} \cup \{a\} \cup (X \setminus \{a,b\})=X$. Now, if $X$ is cofinite, then $supp(X)=A \setminus X$, and so $a,b \in A \setminus X$. Since $a,b \notin X$, we have $a,b \neq   x$ for all $x \in X$, and so $(a\,b)(x)=x$ for all $x \in X$. Thus, in this case we also have $(a\,b) \star X=X$.

Since when both $a,b \in supp(X)$ we have $(a\,b)\star X=X$, it follows that one of $a$ or $b$ does not belong to $supp(X)$. Suppose $b \notin supp(X)$ (the other case is analogue). Let us consider $c\neq a,b$, $c \notin supp(f)$, $c \notin supp(X)$. Then $(b\, c) \in Fix(supp(X))$, and, because $supp(X)$ supports $X$, we have $(b\,c)\star X=X$. Furthermore, $(b\, c) \in Fix(supp(f))$, and by Proposition \ref{2.18'} we have $(a,b)=f(X)=f((b\,c) \star X)=(b\,c) 
\otimes f(X)=(b\,c) \otimes (a,b)=((b\,c)(a), (b\,c)(b))=(a,c)$ which is a contradiction because $b\neq c$. 
Thus, $|A \times A| \nleq^{*} |\wp_{fs}(A)|$.

2. We prove that there does not exist a finitely supported injective 
mapping $f: A \times A \to \wp_{fs}(A)$. Suppose, by contradiction, that 
there is a finitely supported injective mapping $f: A \times A \to 
\wp_{fs}(A)$. According to Proposition~\ref{pco2}, one can define a 
finitely supported surjection $g: \wp_{fs}(A) \to A \times A$. This 
contradicts the above item. Thus, $|A \times A| \nleq |\wp_{fs}(A)|$.

3. We prove that there does not exist a finitely supported surjection $f: A \to A \times A$. Since there exists a surjection $s$ from $\wp_{fs}(A)$ onto $A$ defined by \[
s(X)=\left\{ \begin{array}{ll}
x, & \text{if}\: \text{$X$ is an one-element set $\{x\}$ };\\
a, & \text{if}\: \text{$X$ is not an one-element set,}\: \end{array}\right. 
\] 
where $a$ is a fixed atom, and $s$ is finitely supported (by$\{a\}$), the result follows from item 1. Thus, $|A \times A| \nleq^{*} |A|$.

4. We prove that there does not exist a finitely supported injection $f: A \times A \to A$.
Since there exists an equivariant injection from $A$ into $\wp_{fs}(A)$ defined as $x \mapsto \{x\}$, the result follows from item 2.  Thus, $|A \times A| \nleq |A|$;

Alternatively, one can prove that there does not exist a one-to-one mapping from $A \times A$ to $A$ (and so neither a finitely supported one). Suppose, by contradiction, that there is a an injective mapping  $i: A \times A \to A$. Let us fix two atoms $x$ and $y$ with $x \neq y$. The sets $\{i(a,x)\,|\,a \in A\}$ and $\{i(a,y)\,|\,a \in A\}$ are disjoint and infinite. Thus, $\{i(a,x)\,|\,a \in A\}$ is a infinite and coinfinite subset of $A$, which contradicts the fact that any subset of $A$ is either finite or cofinite.

5. We prove that $|A| \lneq |\wp_{n}(A)| \lneq |\wp_{fs}(A)|$ for all $n \in \mathbb{N}, n \geq 2$. 
Consider $a_{1},a_{2}, \ldots, a_{n-1}$ $, a^{1}_{1}, \ldots, a^{n}_{1}, \ldots, a^{1}_{n-1}, \ldots, a^{n}_{n-1} \in A$ a family of pairwise different elements. Then $i: A \to \wp_{n}(A)$ defined by
 \[i(x)=\left\{ \begin{array}{ll}
\{x,a_{1},a_{2}, \ldots, a_{n-1}\}, & \text{if}\: \text{$x \neq a_{1}, \ldots, a_{n-1}$ };\\
\{ a^{1}_{1}, \ldots, a^{n}_{1}\}, & \text{if}\: \text{$x=a_{1}$ }\\ \vdots \\

\{ a^{1}_{n-1}, \ldots, a^{n}_{n-1}\}, & \text{if}\: \text{$x=a_{n-1}$ }
\: \end{array}\right.\]  
is obviously an  injective mapping from $(A, \cdot)$ to $(\wp_{n}(A), \star)$. Furthermore, we can easy check that $i$ is supported by the finite set \{$a_{1},a_{2}, \ldots, a_{n-1} $ $, a^{1}_{1}, \ldots, a^{n}_{1}, \ldots, a^{1}_{n-1}, \ldots, a^{n}_{n-1}\}$, and so $|A| \leq |\wp_{n}(A)|$ in FSM.

We claim that there does not exist a finitely supported injection from $\wp_{n}(A)$ into $A$. Assume on the contrary
that there exists an finitely supported  injection $f:\wp_{n}(A) \to A$.

First, we claim that, for any $Y \in \wp_{n}(A)$ which is disjoint from $supp(f)$, we have $f(Y) \notin Y$.  Assume by contradiction that $f(Y) \in Y$ for a fixed $Y$ with $Y \cap supp(f)=\emptyset$. Let $\pi$ 
be a permutation of atoms which fixes $supp(f)$ pointwise, and interchanges all the elements of $Y$ (e.g. $\pi$ is a cyclic permutation of $Y$). Since $\pi$ permutes all the elements of $Y$, we have $\pi \cdot 
f(Y)=\pi(f(Y)) \neq f(Y)$. However, $\pi \star Y =\{\pi(a_{1}), 
\ldots, \pi(a_{n})\}=\{a_{1}, \ldots, a_{n}\}=Y$. Since $\pi$ fixes $supp(f)$ pointwise and~$supp(f)$ supports $f$, 
we have $\pi (f(Y))=\pi \cdot f(Y)= f(\pi \star Y)=f(Y)$, a contradiction.

Since $supp(f)$ is finite,  there are infinitely many such $Y$ with the property that $Y \cap supp(f)=\emptyset$. Thus, because it is injective, $f$ takes infinitely many values on those $Y$. Since $supp(f)$ is finite, there should exist at least one element in $\wp_{n}(A)$, denoted by $Z$ such that $Z \cap supp(f)=\emptyset$ and $f(Z) \notin supp(f)$. Thus, 
$f(Z) = a$ for some $a \in A \setminus (Z \cup supp(f))$. Let $b \in A \setminus (supp(f) \cup Z \cup \{a\})$ and also let $\pi = (a\, b)$.
Then $\pi \in Fix(supp(f) \cup Z)$, and hence $f(Z) =f((a\,b) \star Z)=(a\,b)(f(Z))= b$, a contradiction. We obtained that $|A| \neq |\wp_{n}(A)|$ in FSM, and so $|A|<|\wp_{n}(A)|$. 

 We obviously have $|\wp_{n}(A)| \leq |\wp_{fs}(A)|$. We prove below  that there does not exist a finitely supported injective mapping from $\wp_{fs}(A)$ onto one of its finitely supported proper subsets, i.e. any finitely supported injection $f:\wp_{fs}(A) \rightarrow 
\wp_{fs}(A)$ is also surjective.  Let us consider a finitely supported injection $f:\wp_{fs}(A) \rightarrow \wp_{fs}(A)$. Suppose, by contradiction, $Im(f) \subsetneq \wp_{fs}(A)$. This means that there 
exists $X_{0}\in\wp_{fs}(A)$ such that $X_{0}\notin Im(f)$. Since $f$ is injective, we can define an infinite 
sequence $\mathcal{F}=(X_{n})_{n}$ starting from $X_{0}$, with distinct terms of form $X_{n+1}=f(X_{n})$ for all $n \in \mathbb{N}$. Furthermore, according to Proposition \ref{2.18'}, for a fixed $k \in \mathbb{N}$ and $\pi \in Fix(supp(f) \cup supp(X_{k}))$, we have $\pi \star X_{k+1}=\pi \star f(X_{k})=f(\pi \star X_{k})=f(X_{k})=X_{k+1}$. Then, $supp(X_{n+1}) \subseteq supp(f)\cup supp(X_{n})$ for all 
$n \in \mathbb{N}$, and by induction on $n$ we have that $supp(X_{n}) \subseteq supp(f)\cup supp(X_{0})$ for all $n \in \mathbb{N}$. We obtained that each element $X_{n}\in \mathcal{F}$ is supported by the 
same finite set $S:=supp(f)\cup supp(X_{0})$. However, there could exist 
only finitely many subsets of $A$ (i.e. only finitely many elements in $\wp_{fs}(A)$) supported by $S$, namely the subsets of 
$S$ and the supersets of $A\setminus S$ (where a superset of $A \setminus S$ 
is of form $A\setminus X$ with $X \subseteq S$). We contradict the statement 
that the infinite sequence $(X_{n})_{n}$ never repeats. Thus, $f$ is surjective, and so there could not exist a bijection between $\wp_{fin}(A)$ and $\wp_{fs}(A)$, which means $|\wp_{n}(A)| \neq |\wp_{fs}(A)|$. 

6.  Fix $n \in \mathbb{N}$. As in the above item there does not exist neither a finitely supported bijection between $\wp_{n}(A)$ and $\wp_{fs}(A)$, nor a finitely supported bijection between $A$ and $\wp_{n}(A)$. However, there exists a finitely supported injection  $i: A \to \wp_{n}(A)$. Fix an atom $a \in A$. The mapping $s: \wp_{n}(A) \to A $ defined by 
\[
s(X)=\left\{ \begin{array}{ll}
i^{-1}(X), & \text{if}\: \text{$X\in Im(i)$ };\\
a, & \text{if}\: \text{$X \notin Im(i)$}\: \end{array}\right. \]
is supported by $supp(i) \cup \{a\}$ and is surjective. 

Now, fix $n$ atoms $x_{1}, \ldots, x_{n}$.  The mapping $g:\wp_{fs}(A) \to \wp_{n}(A)$ defined by 
\[g(X)=\left\{ \begin{array}{ll}
X, & \text{if}\: \text{$X\in \wp_{n}(A)$ };\\
\{x_{1}, \ldots x_{n}\}, & \text{if}\: \text{$X \notin \wp_{n}(A)$}\: \end{array}\right. 
\] is supported by $\{x_{1}, \ldots x_{n}\}$ and is surjective.  
 
7. We prove that $|A| \lneq |\wp_{fin}(A)| \lneq |\wp_{fs}(A)|$. We obviously have that $|A| \leq |\wp_{fin}(A)|$ by taking the equivariant injective mapping $f:A \to \wp_{fin}(A)$ defined by $f(a)=\{a\}$ for all $a \in A$. We prove, by contradiction, that there is no finitely supported surjection from $A$ onto $\wp_{fin}(A)$. Assume that $g:A \to \wp_{fin}(A)$ is a finitely supported surjection. Let us fix two atoms $x$ and $y$. We define the function $h: \wp_{fin}(A) \to \wp_{2}(A)$ by 
$
h(X)=\left\{ \begin{array}{ll}
X, & \text{if}\: \text{$|X|=2$ };\\
\{x,y\}, & \text{if}\: \text{$|X| \neq 2$}\: .\end{array}\right. $. Since for every $\pi \in S_{A}$ and $X \in \wp_{fin}(A)$ we have $|\pi \star X|=|X|$, we conclude that $h$ is finitely supported by $\{x,y\}$. Thus, $h \circ g$ is a surjection from $A$ onto $\wp_{2}(A)$ supported by $supp(g)\cup\{x,y\}$, which contradicts the previous item. Therefore, $|A|<|\wp_{fin}(A)|$. Since every element in $\wp_{fin}(A)$ belongs to $\wp_{fs}(A)$, but there does not exist a finitely supported injective mapping from $\wp_{fs}(A)$ onto one of its finitely supported proper subsets, we also have  $|\wp_{fin}(A)|<|\wp_{fs}(A)|$.

8. As in the above item there does not exist neither a finitely supported bijection between $\wp_{fin}(A)$ and $\wp_{fs}(A)$, nor a finitely supported bijection between $A$ and $\wp_{fin}(A)$. Fix an atom $a \in A$. The mapping $s: \wp_{fin}(A) \to A$ defined by 
\[
s(X)=\left\{ \begin{array}{ll}
x, & \text{if}\: \text{$X$ is an one-element set $\{x\}$ };\\
a, & \text{if}\: \text{$X$ is not an one-element set}\: \end{array}\right. \]
is supported by $\{a\}$ and is surjective. 

Now, fix an atom $b$.  The mapping $g:\wp_{fs}(A) \to \wp_{fin}(A)$ defined by 
\[g(X)=\left\{ \begin{array}{ll}
X, & \text{if}\: \text{$X\in \wp_{fin}(A)$ };\\
\{b\}, & \text{if}\: \text{$X \notin \wp_{fin}(A)$}\: \end{array}\right. \] is supported by $\{b\}$ and is surjective.  

9. According to Theorem \ref{cardord1}(1) there is no finitely supported surjection from $\wp_{fs}(A)$ onto $A \times A$.
Suppose there is a finitely supported surjective mapping $f: \wp_{fs}(A) \to \wp_{fs}(A) \times \wp_{fs}(A)$. Obviously, there exists a supported surjection $s:\wp_{fs}(A) \to A$ defined by \[
s(X)=\left\{ \begin{array}{ll}
a, & \text{if}\: \text{$X$ is an one-element set \{a\} };\\
x, & \text{if}\: \text{$X$ has more than one element}\: .\end{array}\right. \] where $x$ is a fixed atoms of $A$. The surjection $s$ is supported by $supp(x)=x$. Thus, we can define a surjection $g:\wp_{fs}(A) \times \wp_{fs}(A) \to A \times A$ by $g(X,Y)=(s(X),s(Y))$ for all $X,Y \in \wp_{fs}(A)$. Let $\pi \in Fix(supp(s))$. Since $supp(s)$ supports $s$, by Proposition \ref{2.18'} we have  $g(\pi \otimes_{\star} (X,Y))=g(\pi \star X,\pi \star Y)=(s(\pi \star X),s(\pi \star Y))=(\pi \cdot s(X),\pi \cdot s(Y))=\pi \otimes (s(X),s(Y))$ for all $X,Y \in \wp_{fs}(A)$, where $\otimes_{\star}$ and $\otimes$ represent the $S_{A}$-actions on $\wp_{fs}(A) \times \wp_{fs}(A)$ and $A \times A$, respectively. Thus, $supp(s)$ supports $g$, and so $supp(g) \subseteq supp(s)$. Furthermore, the function $h=g \circ f: \wp_{fs}(A) \to A \times A$ is surjective and finitely supported by $supp(s) \cup supp(f)$. This is a contradiction, and so $|\wp_{fs}(A) \times \wp_{fs}(A)| \nleq^{*} |\wp_{fs}(A)|$.

10.  Suppose, by contradiction, that there is a finitely supported injective mapping  $f:  \wp_{fs}(A) \times \wp_{fs}(A) \to \wp_{fs}(A)$. In the view of Proposition \ref{pco2}, let us fix two finitely supported subsets of $A$, namely $U$ and $V$. We define the function $g: \wp_{fs}(A) \to \wp_{fs}(A) \times \wp_{fs}(A)$ by 
\[
g(X)=\left\{ \begin{array}{ll}
f^{-1}(X), & \text{if}\: \text{$X\in Im(f)$ };\\
(U,V), & \text{if}\: \text{$X \notin Im(f)$}\: .\end{array}\right. 
\]
Clearly, $g$ is surjective. 
Furthermore, $g$ is supported by $supp(f) \cup supp(U) \cup supp(V)$ (the proof uses the fact that $Im(f)$ is a subset of $\wp_{fs}(A)$ supported by $supp(f)$). This contradicts the above item, and so $|\wp_{fs}(A) \times \wp_{fs}(A)| \nleq |\wp_{fs}(A)|$.

11.  In the view of Proposition \ref{pco1}(5) there is a finitely supported injection from $A+A$ into $A \times A$, and a finitely supported surjection from $A \times A$ onto $A+A$ according to Proposition \ref{pco2}. Thus $|A+A|\leq |A \times A|$ and $|A+A|\leq^{*}|A \times A|$  Fix three different atoms $a,b,c \in A$. Define the mapping $f:A+A \to \wp_{fs}(A)$ by 
\[
f(u)=\left\{ \begin{array}{ll}
\{x\}, & \text{if}\: u=(0,x)\: \text{with}\: x \in A;\\
\{a,y\}, & \text{if}\:  u=(1,y)\: \text{with}\: y \in A, y \neq a;\\
\{b,c\}, & \text{if}\: u=(1,a)\end{array}\right. 
\]
One can directly prove that $f$ is injective and supported by $\{a,b,c\}$. According to Proposition \ref{pco2}, we have $|A+A|\leq^{*}|\wp_{fs}(A)|$. If we had $|A \times A|=|A+A|$, we would obtain $|A\times A| \leq^{*} |\wp_{fs}(A)|$ which contradicts item 1. 

12. According to the above item $|A+A|\leq |\wp_{fs}(A)|$.  If we had $|A \times A|=|A+A|$, we would obtain $|A\times A| \leq |\wp_{fs}(A)|$ which contradicts item 2. 
\end{proof}

\begin{proposition} \label{cardord2'}  There exists an invariant set $X$ having the following properties:
\begin{enumerate}
\item $|X| \lneq |X|+|X|$;
\item $|X| \lneq^{*} |X|+|X|$.
\end{enumerate}
\end{proposition}
\begin{proof}

1. First we prove that in FSM we have $|\wp_{fs}(A)|=2|\wp_{fin}(A)|$.
Let us consider the function $f:\wp_{fin}(A) \to \wp_{cofin}(A)$ defined by $f(U)=A \setminus U$ for all $U \in \wp_{fin}(A)$. Clearly, $f$ is bijective. We claim that $f$ is equivariant. Indeed, let $\pi \in S_{A}$. To prove that $f(\pi \star U)$=$\pi \star f(U)$ for all $U \in \wp_{fin}(A)$, we have to prove that $A \setminus (\pi \star U)=\pi \star (A \setminus U)$ for all $U \in \wp_{fin}(A)$. Let $y \in A \setminus (\pi \star U)$. We can express $y$ as $y=\pi \cdot (\pi^{-1} \cdot y)$. If $\pi^{-1} \cdot y \in U$, then $y \in \pi \star U$, which is a contradiction. Thus, $\pi^{-1} \cdot y \in (A \setminus U)$, and so $y \in \pi \star (A \setminus U)$. Conversely, if  
$y \in \pi \star (A \setminus U)$, then $y=\pi\cdot x$ with $x \in A \setminus U$. Suppose $y \in \pi \star U$. Then $y=\pi\cdot z$ with $z \in U$. Thus, $x=z$ which is a contradiction, and so $y \in A \setminus (\pi \star U)$. Since $f$ is equivariant and bijective, it follows that $|\wp_{fin}(A)|=|\wp_{cofin}(A)|$. However, every finitely supported subset of $A$ is either finite or cofinite, and so $\wp_{fs}(A)$ is the union of the disjoint subsets $\wp_{fin}(A)$ and $\wp_{cofin}(A)$. Thus, $|\wp_{fs}(A)|=2|\wp_{fin}(A)|$.  Moreover, there exists an equivariant injection $i: \wp_{fin}(A) \to \wp_{fs}(A)$ defined by $i(U)=U$ for all $U \in \wp_{fin}(A)$. However, there does not exist a finitely supported one-to-one mapping from $\wp_{fs}(A)$ onto one of its finitely supported proper subsets. Thus, there could not exist a bijection $f: \wp_{fs}(A) \to \wp_{fin}(A)$. Therefore, $|\wp_{fin}(A)| \neq |\wp_{fs}(A)|=2|\wp_{fin}(A)|$. We can consider $X=\wp_{fin}(A)$ or $X=\wp_{cofin}(A)$.

2. It remains to prove that there is a finitely supported surjection 
from $\wp_{fs}(A)$ onto $\wp_{fin}(A)$. We either use Proposition~\ref{pco2} 
or effectively construct the surjection as below. Fix $a \in A$. We define $g:\wp_{fs}(A) \to \wp_{fin}(A)$ by
\[ g(U)=\left\{ \begin{array}{ll}
U, & \text{if}\: \text{$U\in \wp_{fin}(A)$ };\\
\{a\}, & \text{if}\: \text{$U \notin \wp_{fin}(A)$}\: .\end{array}\right. 
\] Clearly, $g$ is supported by $\{a\}$ and surjective. We can consider $X=\wp_{fin}(A)$ or  $X=\wp_{cofin}(A)$. 
\end{proof}

\section{Forms of Infinite in Finitely Supported Structures} \label{chap9}

The equivalence of various definitions for infinity is provable in ZF 
under the consideration of the axiom of choice. Since in FSM the axiom of 
choice fails, our goal is to study various FSM forms of infinite and to 
provide several relations between them.

\begin{definition}
Let $X$ be a finitely supported subset of an invariant set. 

\begin{enumerate}
\item $X$ is called \emph{FSM usual infinite} if $X$ does not correspond one-to-one and onto to a finite ordinal. 
We simply call \emph{infinite} an FSM usual infinite set. 
\item $X$ is \emph{FSM covering infinite} if there is a finitely supported directed family $\mathcal{F}$ of finitely supported sets
with the property that $X$ is contained in the union of the members of $\mathcal{F}$, but there does not exist $Z\in \mathcal{F}$ such that $X\subseteq Z$.
\item $X$ is called \emph{FSM Tarski I infinite} if there exists a finitely supported one-to-one mapping of $X$ onto $X \times X$.
\item $X$ is called \emph{FSM Tarski II infinite} if there exists a finitely supported family of finitely supported subsets of $X$, totally ordered by inclusion, having no maximal element.
\item $X$ is called \emph{FSM Tarski III infinite} if $|X|=2|X|$.
\item $X$ is called \emph{FSM Mostowski infinite} if there exists an infinite finitely supported  totally ordered subset of $X$. 
\item $X$ is called \emph{FSM Dedekind infinite} if there exist a finitely supported one-to-one mapping of $X$ onto a finitely supported proper subset of $X$. 
\item $X$ is \emph{FSM ascending infinite} if there is a finitely supported increasing countable chain of finitely supported sets $X_{0}\subseteq X_{1}\subseteq\ldots\subseteq X_{n}\subseteq\ldots$
with $X\subseteq\cup X_{n}$, but there does not exist $n\in\mathbb{N}$ such
that $X\subseteq X_{n}$;
\end{enumerate}
\end{definition}

Note that in the definition of FSM Tarski II infinity for a certain $X$, 
the existence of a finitely supported family of finitely supported subsets 
of $X$ is required, while in the definition of FSM ascending infinity for 
$X$, the related family of finitely supported subsets of $X$ has to be FSM 
countable (i.e. the mapping $n \mapsto X_{n}$ should be finitely 
supported). It is immediate that if $X$ is FSM ascending infinite, then it 
is also FSM Tarski II infinite.

\begin{theorem}
Let $X$ be a finitely supported subset of an invariant set. Then $X$ is FSM usual infinite if and only if $X$ is FSM covering infinite.
\end{theorem}

\begin{proof}
Let us suppose that $X$ is FSM usual infinite. Let $\mathcal{F}$ be the 
family of all FSM usual non-infinite (FSM usual finite) subsets of $X$ 
ordered by inclusion. Since $X$ is finitely supported, it follows that 
$\mathcal{F}$ is supported by $supp(X)$. Moreover, since all the elements 
of $\mathcal{F}$ are finite sets, it follows that all the elements of 
$\mathcal{F}$ are finitely supported. Clearly,~$\mathcal{F}$ is directed 
and $X$ is the union of the members of $\mathcal{F}$. Suppose by 
contradiction, that $X$ is not FSM covering infinite. Then there exists 
$Z\in\mathcal{F}$ such that $X\subseteq Z$. Therefore, $X$ should by FSM 
usual finite which is a contradiction with our original assumption.

Conversely, assume that $X$ is FSM covering infinite. Suppose, by contradiction that $X$ is FSM usual finite, i.e. $X=\{x_{1}, \ldots x_{n}\}$.  Let $\mathcal{F}$ be a directed
family such that $X$ is contained in the union of the members of
$\mathcal{F}$ (at least one such a family exists, for example $\wp_{fs}(X)$). Then for each $i \in \{1, \ldots, n\}$ there exists $F_{i} \in \mathcal{F}$ such that $x_{i} \in F_{i}$. Since $\mathcal{F}$ is directed, there is $Z \in \mathcal {F}$ such that $F_{i} \subseteq Z$ for all $i \in \{1, \ldots, n\}$, and so $X \subseteq Z$ with $Z\in \mathcal{F}$, which is a contradiction. 
\end{proof}

\begin{theorem} \label{ti1} The following properties of FSM Dedekind infinite sets hold.
\begin{enumerate}
\item Let $X$ be a finitely supported subset of an invariant set $Y$. Then $X$ is FSM Dedekind infinite if and only if there exists a finitely supported one-to-one mapping $f: \mathbb{N} \to X$. As a consequence, an FSM superset of an FSM Dedekind infinite set is FSM Dedekind infinite, and an FSM subset of an FSM set that is not Dedekind infinite is also not FSM Dedekind infinite.
\item Let $X$ be an infinite finitely supported subset of an invariant set $Y$. Then the sets $\wp_{fs}(\wp_{fin}(X))$ and $\wp_{fs}(T_{fin}(X))$ are FSM Dedekind infinite.
\item Let $X$ be an infinite finitely supported subset of an invariant set $Y$. Then the set $\wp_{fs}(\wp_{fs}(X))$ is FSM Dedekind infinite.
\item Let $X$ be a finitely supported subset of an invariant set $Y$ such that $X$ does not contain an infinite subset $Z$ with the property that all the elements of $Z$ are supported by the same set of atoms. Then $X$ is not FSM Dedekind infinite. 
\item Let $X$ be a finitely supported subset of an invariant set $Y$ such that $X$ does not contain an infinite subset $Z$ with the property that all the elements of $Z$ are supported by the same set of atoms. Then $\wp_{fin}(X)$ is not FSM Dedekind infinite.
\item Let $X$ and $Y$ be two finitely supported subsets of an invariant set $Z$. If neither $X$ nor $Y$ is FSM Dedekind infinite, then $X \times Y$ is not FSM Dedekind infinite. 
\item Let $X$ and $Y$ be two finitely supported subsets of an invariant set $Z$. If neither $X$ nor $Y$ is FSM Dedekind infinite, then $X + Y$ is not FSM Dedekind infinite. 
\item Let $X$ be a finitely supported subset of an invariant set $Y$. Then $\wp_{fs}(X)$ is FSM Dedekind infinite if and only if $X$ is FSM ascending infinite.
\item Let $X$ be a finitely supported subset of an invariant set $Y$. If $X$ is FSM Dedekind infinite, then $X$ is FSM ascending infinite. The reverse implication is not valid.

\end{enumerate}
\end{theorem}

\begin{proof}
\begin{enumerate}
\item Let us suppose that $(X, \cdot)$ is FSM Dedekind infinite, and 
$g: X \rightarrow X$ is an injection supported by the finite set $S 
\subsetneq A$ with the property that $Im(g) \subsetneq X$. This means 
that there exists $supp(g) \subseteq S$ and there exists $x_{0}\in X$ such that $x_{0}\notin Im(g)$.  We can form a sequence of elements from~$X$ which has the first term $x_{0}$ and the general term 
$x_{n+1}=g(x_{n})$ for all $n \in \mathbb{N}$.  Since $x_{0}\notin Im(g)$ it follows that $x_{0} 
\neq g(x_{0})$. Since $g$ is injective and $x_{0} \notin Im(g)$, by 
induction we obtain that $g^{n}(x_{0}) \neq g^{m}(x_{0})$ for all $n,m \in 
\mathbb{N}$ with $n \neq m$. Furthermore,~$x_{n+1}$ 
is supported by $supp(g)\cup supp(x_{n})$ for all $n \in \mathbb{N}$. 
Indeed, let $\pi \in Fix(supp(g)\cup supp(x_{n}))$. According to 
Proposition~\ref{2.18'}, $\pi \cdot x_{n+1}= \pi \cdot g(x_{n})=g(\pi 
\cdot x_{n})=g(x_{n})=x_{n+1}$. Since $supp(x_{n+1})$ is the least set 
supporting $x_{n+1}$, we obtain $supp(x_{n+1}) \subseteq supp(g)\cup 
supp(x_{n})$ for all $n \in \mathbb{N}$. By finite recursion, we have 
$supp(x_{n}) \subseteq supp(g)\cup supp(x_{0})$ for all 
$n\in\mathbb{N}$. Since all $x_{n}$ are supported by the same set of 
atoms $supp(g)\cup supp(x_{0})$, we have that the function $f:\mathbb{N} \to X$, defined by 
$f(n)=x_{n}$, is also finitely supported (by the set $supp(g)\cup 
supp(x_{0}) \cup supp(X)$ not depending on $n$). Indeed, for any $\pi \in Fix(supp(g)\cup supp(x_{0}) \cup supp(X))$ we 
have $f(\pi \diamond n)=f(n)=x_{n}=\pi \cdot x_{n}=\pi\cdot f(n)$, 
$\forall n \in \mathbb{N}$, where by $\diamond$ we denoted the trivial 
$S_{A}$-action on $\mathbb{N}$. Furthermore, because $\pi$ fixes $supp(X)$ pointwise we have $\pi \cdot f(n) \in X$ for all $n \in \mathbb{N}$. From Proposition 
\ref{2.18'} we have that $f$ is finitely supported. Obviously, $f$ is also injective.

Conversely, suppose there exists a finitely supported injective mapping 
$f: \mathbb{N} \to X$. According to Proposition~\ref{2.18'}, it follows 
that for any $\pi \in Fix(supp(f))$ we have $\pi \cdot f(n)=f(\pi \diamond 
n)=f(n)$ and $\pi \cdot f(n) \in X$ for all $n \in \mathbb{N}$. Let us 
define $g:X \to X$ by \[ g(x)=\left\{ \begin{array}{ll} f(n+1), & 
\text{if}\: \text{there exists $n \in \mathbb{N}$ with $x=f(n)$};\\ x, & 
\text{if}\: \text{$x \notin Im(f)$}\: .\end{array}\right. \] We claim that 
$g$ is supported by $supp(f) \cup supp(X)$. Indeed, let us consider $\pi 
\in Fix(supp(f) \cup supp(X))$ and $x \in X$. If there is some $n$ such 
that $x=f(n)$, we have that $\pi \cdot x=\pi \cdot f(n)=f(n)$, and so 
$g(\pi \cdot x)=g(f(n))=f(n+1)=\pi \cdot f(n+1)=\pi \cdot g(x)$. If $x 
\notin Im(f)$, we prove by contradiction that $\pi \cdot x \notin Im(f)$. 
Indeed, suppose that $\pi \cdot x \in Im(f)$. Then there is $y \in 
\mathbb{N}$ such that $\pi \cdot x=f(y)$ or, equivalently, $x = \pi^{-1} 
\cdot f(y)$. However, since $\pi \in Fix(supp(f))$, from Proposition 
\ref{2.18'} we have $\pi^{-1} \cdot f(y)=f(\pi^{-1} \diamond y)$, and so 
we get $x=f(\pi^{-1} \diamond y)=f(y) \in Im(f)$ which contradicts the 
assumption that $x \notin Im(f)$. Thus, $\pi \cdot x \notin Im(f)$, and so 
$g(\pi \cdot x)=\pi \cdot x=\pi \cdot g(x)$. We obtained that $g(\pi \cdot 
x)=\pi \cdot x=\pi \cdot g(x)$ for all $x \in X$ and all $\pi \in 
Fix(supp(f) \cup supp(X))$. Furthermore, $\pi \cdot g(x) \in \pi \star 
X=X$ (where by $\star$ we denoted the $S_{A}$-action on $\wp_{fs}(Y)$), 
and so $g$ is finitely supported. Since $f$ is injective, it follows 
immediately that $g$ is injective. Furthermore, $Im(g)=X \setminus 
\{f(0)\}$ which is a proper subset of $X$, finitely supported by 
$supp(f(0)) \cup supp(X)=supp(f) \cup supp(X)$.

\item The family $\wp_{fin}(X)$ represents the family of those finite subsets of $X$ (these  subsets of $X$ are finitely supported as subsets of the invariant set $Y$ in the sense of Definition \ref{2.14}). Obviously, $\wp_{fin}(X)$ is a finitely supported subset of the invariant set $\wp_{fs}(Y)$, supported by $supp(X)$. This is because whenever $Z$ is an element of $\wp_{fin}(X)$ (i.e. whenever $Z$ is a finite subset of $X$) and $\pi$ fixes $supp(X)$ pointwise, we have that $\pi \star Z$ is also a finite subset of $X$. The family $\wp_{fs}(\wp_{fin}(X))$ represents the family of those subsets of $\wp_{fin}(X)$ which are finitely supported as subsets of the invariant set $\wp_{fs}(Y)$ in the sense of Definition \ref{2.14}. As above, according to Proposition \ref{2.15}, we have that $\wp_{fs}(\wp_{fin}(X))$ is a finitely supported subset of the invariant set $\wp_{fs}(\wp_{fs}(Y))$, supported by $supp(\wp_{fin}(X)) \subseteq supp(X)$.

Let $X_{i}$ be the set of all $i$-sized subsets from $X$, i.e. $X_{i}=\{Z 
\subseteq X\,|\,|Z|=i\}$. Since $X$ is infinite, it follows that each 
$X_{i}, i \geq 1$ is non-empty. Obviously, we have that any $i$-sized 
subset $\{x_{1}, \ldots, x_{i}\}$ of $X$ is finitely supported (as a 
subset of $Y$) by $supp(x_{1}) \cup \ldots \cup supp(x_{i})$. Therefore, 
$X_{i} \subseteq \wp_{fin}(X)$ and $X_{i} \subseteq \wp_{fs}(Y)$ for all 
$i \in \mathbb{N}$. Since $\cdot$ is a group action, the image of an 
$i$-sized subset of $X$ under an arbitrary permutation is an $i$-sized 
subset of $Y$. However, any permutation of atoms that fixes $supp(X)$ 
pointwise also leaves $X$ invariant, and so for any permutation $\pi \in 
Fix(supp(X))$ we have that $\pi \star Z$ is an $i$-sized subset of $X$ 
whenever $Z$ is an $i$-sized subset of $X$. Thus, each $X_{i}$ is a subset 
of $\wp_{fin}(X)$ finitely supported by $supp(X)$, and so~$X_{i} \in 
\wp_{fs}(\wp_{fin}(X))$.

We define $f: \mathbb{N} \to \wp_{fs}(\wp_{fin}(X))$ by $f(n)=X_{n}$. We claim that $supp(X)$ supports $f$. Indeed, let $\pi \in Fix(supp(X))$. Since $supp(X)$ supports $X_{n}$ for all $n \in \mathbb{N}$, we have $\pi \star f(n)=\pi \star X_{n}=X_{n}=f(n)= f(\pi \diamond n)$ (where $\diamond$ is the trivial $S_{A}$-action on $\mathbb{N}$) and $\pi \star f(n)=\pi \star X_{n}=X_{n} \in \wp_{fs}(\wp_{fin}(X))$ for all $n \in \mathbb{N}$. According to Proposition \ref{2.18'}, we have that $f$ is finitely supported. Furthermore, $f$ is injective and, by  item 1, we have that $\wp_{fs}(\wp_{fin}(X))$ is FSM Dedekind infinite.

If we consider $Y_{i}$ the set of all $i$-sized injective tuples formed by elements of $X$, we have that each $Y_{i}$ is a subset of $T_{fin}(X)$ supported by $supp(X)$, and the family $(Y_{i})_{i \in \mathbb{N}}$ is a countably infinite, uniformly supported, subset of $\wp_{fs}(T_{fin}(X))$. From item 1 we get that $\wp_{fs}(T_{fin}(X))$ is FSM Dedekind infinite. 

\item The proof is actually the same as in the above item because every $X_{i} \in \wp_{fs}(\wp_{fs}(A))$.

\item If there does not exist a uniformly supported subset of $X$, then there does not exist a finitely supported injective mapping $f:\mathbb{N} \to X$, and so $f$ cannot be FSM Dedekind infinite.

\item We prove the following lemma:

\begin{lemma} \label{lem4} Let $X$ be a finitely supported subset of an invariant set $Y$ such that $X$ does not contain an infinite uniformly supported subset. Then the set $\wp_{fin}(X)=\{Z\!\subseteq\! X\,|\, Z\, \text{finite}\}$ does not contain an infinite uniformly supported subset. 
\end{lemma}

\emph{Proof of Lemma \ref{lem4}.} Suppose, by contradiction, that the set 
$\wp_{fin}(X)$ contains an infinite subset $\mathcal{F}$ such that all the 
elements of $\mathcal{F}$ are different and supported by the same finite 
set $S$. Therefore, we can express $\mathcal{F}$ as 
$\mathcal{F}=(X_{i})_{i \in I} \subseteq \wp_{fin}(X)$ with the properties 
that $X_{i} \neq X_{j}$ whenever $i \neq j$ and $supp(X_{i}) \subseteq S$ 
for all $i \in I$. Fix an arbitrary $j \in I$. However, from Proposition 
\ref{4.4-9}, because $supp(X_{j})=\underset{x \in X_{j}}{\cup}supp(x)$, we 
have that $X_{j}$ has the property that $supp(x) $ $\subseteq S$ for all 
$x \in X_{j}$. Since $j$ has been arbitrarily chosen from $I$, it follows 
that every element from every set of form $X_{i}$ is supported by $S$, and 
so $\underset{i}{\cup}X_{i}$ is an uniformly supported subset of $X$ (all 
its elements being supported by $S$). Furthermore, $\underset{i \in 
I}{\cup}X_{i}$ is infinite because the family $(X_{i})_{i \in I}$ is 
infinite and $X_{i} \neq X_{j}$ whenever $i \neq j$.  Otherwise, if 
$\underset{i}{\cup}X_{i}$ was finite, the family $(X_{i})_{i \in I}$ would 
be contained in the finite set $\wp (\underset{i}{\cup}X_{i})$, and so it 
couldn't be infinite with the property that $X_{i} \neq X_{j}$ whenever $i 
\neq j$. We were able to construct an infinite uniformly supported subset 
of $X$, namely $\underset{i}{\cup}X_{i}$, and this contradicts the 
hypothesis that~$X$ does not contain an infinite uniformly supported subset.
\smallskip

\emph{Proof of this item} According to the above lemma, if $X$ does not 
contain an infinite uniformly supported subset, then $\wp_{fin}(X)$ does 
not contain an infinite uniformly supported subset. Suppose, by 
contradiction, that $\wp_{fin}(X)$ is FSM Dedekind infinite. According to 
item 1, there exists a finitely supported injective mapping $f: \mathbb{N} 
\to \wp_{fin}(X)$. Thus, because $\mathbb{N}$ is a trivial invariant set, 
according to Proposition \ref{2.18'}, there exists an infinite injective 
(countable) sequence $f(\mathbb{N})=(X_{i})_{i \in \mathbb{N}} \subseteq 
\wp_{fin}(X)$ having the property $supp(X_{i}) \subseteq supp(f)$ for all 
$i \in \mathbb{N}$. We obtained that $\wp_{fin}(X)$ contains an infinite 
uniformly supported subset $(X_{i})_{i \in \mathbb{N}}$, which is a 
contradiction.

\item Suppose, by contradiction, that $X \times Y$ is FSM Dedekind 
infinite. According to item 1, there exists a finitely supported injective 
mapping $f: \mathbb{N} \to X \times Y$ Thus, according to Proposition 
\ref{2.18'}, there exists an infinite injective sequence 
$f(\mathbb{N})=((x_{i},y_{i}))_{i \in \mathbb{N}} \subseteq X \times Y$ 
with the property that $supp((x_{i}, y_{i})) \subseteq supp(f)$ for all $i 
\in \mathbb{N}$ (1). Fix some $j \in \mathbb{N}$. We claim that 
$supp((x_{j}, y_{j})) =supp(x_{j}) \cup supp(y_{j})$.  Let $U=(x_{j}, 
y_{j})$, and $S=supp(x_{j})\cup supp(y_{j})$. Obviously, $S$ supports $U$. 
Indeed, let us consider $\pi\in Fix(S)$. We have that $\pi\in 
Fix(supp(x_{j}))$ and also $\pi\in Fix(supp(y_{j}))$ Therefore, $\pi\cdot 
x_{j}=x_{j}$ and $\pi\cdot y_{j}=y_{j}$, and so $\pi \otimes (x_{j}, 
y_{j})=(\pi \cdot x_{j}, \pi \cdot y_{j})=(x_{j}, y_{j})$, where $\otimes$ 
represent the $S_{A}$ action on $X \times Y$ described in Proposition 
\ref{p1}. Thus, $supp(U) \subseteq S$. It remains to prove that $S 
\subseteq supp(U)$. Fix $\pi \in Fix(supp(U))$. Since $supp(U)$ supports 
$U$, we have $\pi \otimes (x_{j}, y_{j})=(x_{j}, y_{j})$, and so $(\pi 
\cdot x_{j}, \pi \cdot y_{j})=(x_{j}, y_{j})$, from which we get $\pi 
\cdot x_{j}=x_{j}$ and $\pi \cdot y_{j}= y_{j}$. Thus, $supp(x_{j}) 
\subseteq supp(U)$ and $supp(y_{j}) \subseteq supp(U)$. Hence 
$S=supp(x_{j})\cup supp(y_{j}) \subseteq supp(U)$.

According to relation (1) we obtain, $supp(x_{i})\cup supp(y_{i}) 
\subseteq supp(f)$ for all $i \in \mathbb{N}$. Thus, $supp(x_{i}) 
\subseteq supp(f)$ for all $i \in \mathbb{N}$ and $supp(y_{i}) \subseteq 
supp(f)$ for all $i \in \mathbb{N}$ (2). Since the sequence 
$((x_{i},y_{i}))_{i \in \mathbb{N}}$ is infinite and injective, then at 
least one of the sequences $(x_{i})_{i \in \mathbb{N}}$ and $(y_{i})_{i 
\in \mathbb{N}}$ is infinite. Assume that $(x_{i})_{i \in \mathbb{N}}$ is 
infinite. Then there exists an infinite subset $B$ of $\mathbb{N}$ such 
that $(x_{i})_{i \in B}$ is injective, and so there exists an injection 
$u: B \to X$ defined by $u(i)=x_{i}$ for all $i \in B$ which is supported 
by $supp(f)$ (according to relation (2) and Proposition \ref{2.18'}).  
However, since $B$ is an infinite subset of $\mathbb{N}$, there exists a 
ZF bijection $h: \mathbb{N} \to B$. The construction of $h$ requires only 
the fact that $\mathbb{N}$ is well-ordered which is obtained from the 
Peano construction of~$\mathbb{N}$ and does not involve a form of the 
axiom of choice.  Since both $B$ and $\mathbb{N}$ are trivial invariant 
sets, it follows that~$h$ is equivariant. Thus, $u \circ h$ is an 
injection from $\mathbb{N}$ to $X$ which is finitely supported by $supp(u) 
\subseteq supp(f)$. This contradicts the assumption that $X$ is not FSM 
Dedekind infinite.

\begin{remark} \label{rrr} 
Analogously, using the relation $supp(x) \cup supp(y)=supp((x,y))$ for all $x \in X$ and $y \in Y$ derived from Proposition \ref{4.4-9}, it can be proved that $X \times Y$ does not contain an infinite uniformly supported subset if neither $X$ nor $Y$ contain an infinite uniformly supported subset.
\end{remark}

\item Suppose, by contradiction, that $X + Y$ is FSM Dedekind infinite. According to item 1, there exists a finitely supported injective mapping $f: \mathbb{N} \to X + Y$. Thus, there exists an infinite injective sequence $(z_{i})_{i \in \mathbb{N}} \subseteq X+ Y$ such that $supp(z_{i}) \subseteq supp(f)$ for all $i \in \mathbb{N}$.   According to the construction of the disjoint union of two $S_{A}$-sets (see Proposition \ref{p1}), as in the proof of item 6,  there should exist an infinite subsequence of $(z_{i})_{i}$ of form $((0, x_{j}))_{x_{j} \in X}$ which is uniformly supported by $supp(f)$, or an infinite sequence of form $((1, y_{k}))_{y_{k} \in Y}$ which is uniformly supported by $supp(f)$.  Since $0$ and $1$ are constants, this means there should exist at least an infinite uniformly supported sequence of elements from $X$, or an infinite uniformly supported sequence of elements from $Y$. This contradicts the hypothesis neither $X$ nor $Y$ is FSM Dedekind infinite.

\begin{remark}  Analogously, it can be proved that $X + Y$ does not contain an infinite uniformly supported subset if neither $X$ nor $Y$ contain an infinite uniformly supported subset.
\end{remark}

\item Assume, by contradiction, that $(X_{n})_{n \in \mathbb{N}}$ is an 
infinite countable family of different subsets of $X$ such that the 
mapping $n\mapsto X_{n}$ is finitely supported. Thus, each $X_{n}$ is 
supported by the same set $S=supp(n \mapsto X_{n})$. We define a countable 
family $(Y_{n})_{n \in \mathbb{N}}$ of subsets of $X$ that are non-empty 
and pairwise disjoint. A ZF construction of such a family belongs to 
Kuratowski and can also be found in Lemma 4.11 from \cite{herrlich}. This 
approach works also in FSM in the view of the $S$-finite support principle 
because every $Y_{k}$ is defined only involving elements in the family 
$(X_{n})_{n \in \mathbb{N}}$, and so whenever $(X_{n})_{n \in \mathbb{N}}$ 
is uniformly supported (meaning that all $X_{n}$ are supported by the same 
set of atoms), we get that $(Y_{n})_{n \in \mathbb{N}}$ is uniformly 
supported. Formally the sequence $(Y_{n})_{n \in \mathbb{N}}$ is 
recursively constructed as below. For $n \in \mathbb{N}$, assume that 
$Y_{m}$ is defined for any $m<n$ such that the set $\{ X_{k} \setminus 
\underset{m<n}\cup Y_{m}\,|\,k \geq n\}$ is infinite. Define 
$n'=min\{k\,|\,k \geq n \:\text{and}\: X_{k} \setminus \underset{m<n}\cup 
Y_{m} \neq \emptyset \:\text{and}\: (X \setminus X_{k}) \setminus 
\underset{m<n}\cup Y_{m} \neq \emptyset\}$. We define \[Y_{n}=\left\{ 
\begin{array}{ll} X_{n'} \setminus \underset{m<n}\cup Y_{m}, & \text{if}\: 
\{X_{k}\setminus (X_{n'} \cup \underset{m<n}\cup Y_{m}) \,|\, 
k>n'\}\:\text{is infinite};\\ (X \setminus X_{n'})\setminus 
\underset{m<n}\cup Y_{m}, & \text{otherwise} .\end{array}\right. \]
Obviously, $Y_{1}$ is supported by $S \cup supp(X)$. By induction, assume 
that $Y_{m}$ is supported by $S \cup supp(X)$ for each $m<n$. Since 
$Y_{n}$ is defined as a set combination of $X_{i}$'s (which are all 
$S$-supported) and $Y_{m}$'s with $m<n$, we get that $Y_{n}$ is supported 
by $S \cup supp(X)$ according to the $S$-finite support principle. 
Therefore the family $(Y_{i})_{i \in \mathbb{N}}$ is uniformly supported 
by $S \cup supp(X)$.  Let $U_{i}=Y_{0} \cup \ldots \cup Y_{i}$ for all $i 
\in \mathbb{N}$. Clearly all $U_{i}$ are supported by $S \cup supp(X)$, 
and $U_{0} \subsetneq U_{1} \subsetneq U_{2} \subsetneq \ldots \subsetneq 
X$. Let $V_{n}=(X\setminus\underset{i\in\mathbb{N}}{\cup}U_{i})\cup 
U_{n}$. Clearly, $X=\underset{n\in \mathbb{N}}{\cup} V_{n}$. Moreover, 
$V_{n}$ is supported by $S \cup supp(X)$ for all $n \in \mathbb{N}$. 
Therefore, the mapping $n\mapsto V_{n}$ is finitely supported. Obviously, 
$V_{0} \subsetneq V_{1} \subsetneq V_{2} \subsetneq \ldots \subsetneq X$. 
However, there does not exist $n\in\mathbb{N}$ such that $X=V_{n}$, and so 
$X$ is FSM ascending infinite.

The converse holds since  if $X$ is FSM ascending infinite, there is a finitely supported increasing countable chain of finitely supported sets $X_{0}\subseteq X_{1}\subseteq\ldots\subseteq X_{n}\subseteq\ldots$ with $X\subseteq\cup X_{n}$, but there does not exist $n\in\mathbb{N}$ such that $X\subseteq X_{n}$. In this sequence there should exist infinitely many different elements of form $X_{i}$ (otherwise their union will be a term of the sequence), and the result follows from Proposition \ref{cou}.  

\item Suppose $X$ is FSM Dedekind infinite. Therefore, $\wp_{fs}(X)$ is 
FSM Dedekind infinite. According to item 8, we have that $X$ is FSM 
ascending infinite. The reverse implication is not valid because, as it is
proved in Proposition~\ref{tari}, $\wp_{fin}(A)$ is FSM ascending infinite, 
but not FSM Dedekind infinite.

\end{enumerate}
\end{proof}

\begin{corollary} \label{ti2}
The following sets and all of their FSM usual infinite subsets are 
FSM usual infinite, but they are not FSM Dedekind infinite.
\begin{enumerate}
\item The invariant set $A$ of atoms.
\item The powerset $\wp_{fs}(A)$ of the set of atoms.
\item The set $T_{fin}(A)$ of all finite injective tuples of atoms.
\item The invariant set $A^{A}_{fs}$ of all finitely supported functions from $A$ to $A$.
\item The invariant set of all finitely supported functions $f:A \to A^{n}$, where $n \in \mathbb{N}$.
\item  The invariant set of all finitely supported functions $f:A \to T_{fin}(A)$.
\item  The invariant set of all finitely supported functions $f:A \to \wp_{fs}(A)$.
\item The sets
 $\wp_{fin}(A)$, $\wp_{cofin}(A)$, $\wp_{fin}(\wp_{fs}(A))$, $\wp_{fin}(\wp_{cofin}(A))$, $\wp_{fin}(\wp_{fin}(A))$, $\wp_{fin}(A^{A}_{fs})$.
\item Any construction of finite powersets of form
$\wp_{fin}(\ldots \wp_{fin}(A))$, $\wp_{fin}(\ldots \wp_{fin}(P(A)))$, or $\wp_{fin}(\ldots \wp_{fin}(\wp_{fs}(A)))$.
\item Every finite Cartesian combination between the set $A$, $\wp_{fin}(A)$, $\wp_{cofin}(A)$, $\wp_{fs}(A)$ and $A^{A}_{fs}$\ .
\item The disjoint unions $A+A^{A}_{fs}$, $A+\wp_{fs}(A)$,  $\wp_{fs}(A)+A^{A}_{fs}$ and $A+\wp_{fs}(A)+A^{A}_{fs}$ and all finite disjoint unions between  $A$, $A^{A}_{fs}$ and  $\wp_{fs}(A)$.
\end{enumerate}
\end{corollary}

\begin{proof}
\begin{enumerate}
\item $A$ does not contain an infinite uniformly supported subset, and so it is not FSM Dedekind infinite (according to Theorem \ref{ti1}(4)).
\item $\wp_{fs}(A)$ does not contain an infinite uniformly supported subset because for any finite set $S$ of atoms there exist only finitely many elements of $\wp_{fs}(A)$ supported by $S$, namely the subsets of $S$ and the supersets of $A\setminus S$. Thus, $\wp_{fs}(A)$ it is not FSM Dedekind infinite (Theorem \ref{ti1}(4)).
\item $T_{fin}(A)$ does not contain an infinite uniformly supported subset 
because the finite injective tuples of atoms supported by a finite set $S$ are 
only those injective tuples formed by elements of $S$, being at most 
$1+A_{|S|}^{1}+A_{|S|}^{2}+\ldots+A_{|S|}^{|S|}$ such tuples, where 
$A_{n}^{k}=n(n-1)\ldots (n-k+1)$.
\item We prove the following lemmas.
\begin{lemma} \label{lem''} 
Let $S=\{s_{1},\ldots,s_{n}\}$ be a finite subset of an invariant set $(U, \cdot)$ and $X$ a finitely supported subset of an invariant set $(V, \diamond)$. Then if $X$ is does not contain an infinite uniformly supported subset, we have that $X^{S}_{fs}$ does not contain an infinite uniformly supported subset.
\end{lemma}
\emph{Proof of Lemma \ref{lem''}}.
First we prove that there is an FSM injection $g$ from $X^{S}_{fs}$ into 
$X^{|S|}$.  For $f \in X^{S}_{fs}$ define $g(f)=(f(s_{1}),\ldots, 
f(s_{n}))$. Clearly $g$ is injective (and it is also surjective). Let $\pi 
\in Fix(supp(s_{1}) \cup \ldots \cup supp(s_{n}) \cup supp(X))$. Thus, 
$g(\pi \widetilde{\star} f)=(\pi \diamond f(\pi^{-1} \cdot s_{1}),\ldots, 
\pi \diamond f(\pi^{-1} \cdot s_{n}))=(\pi \diamond f( s_{1}),\ldots, \pi 
\diamond f(s_{n}))$ $=\pi \otimes g(f)$ for all $f \in X^{S}_{fs}$, where 
$\otimes$ is the $S_{A}$-action on $X^{|S|}$ defined as in Proposition 
\ref{p1}. Hence $g$ is finitely supported, and the conclusion follows from 
Theorem \ref{ti1}(1) and by repeatedly applying similar arguments as in 
Theorem \ref{ti1}(6) (if we slightly modify the proof of the theorem, 
using the fact that $supp(x) \cup supp(y)=supp((x,y))$ for all $x,y \in 
X$, we show that the $|S|$-time Cartesian product of $X$, i.e. $X^{|S|}$ 
does not contain an infinite uniformly supported subset; otherwise $X$ 
should contain itself an infinite uniformly supported subset, which 
contradicts the hypothesis).

\begin{lemma} \label{lemyy} 
Let $S=\{s_{1},\ldots,s_{n}\}$ be a finite subset of an invariant set $(U, \cdot)$ and $X$ a finitely supported subset of an invariant set $(V, \diamond)$. Then if $X$ is not FSM Dedekind infinite, we have that $X^{S}_{fs}$ is not FSM Dedekind-infinite.
\end{lemma}
\emph{Proof of Lemma \ref{lemyy}}
First we proved that there is an FSM injection $g$  from $X^{S}_{fs}$ into $X^{|S|}$.  The conclusion follows from Theorem \ref{ti1}(1) and by repeatedly applying Theorem \ref{ti1}(6) (from which we know that the $|S|$-time Cartesian product of $X$, i.e. $X^{|S|}$, is not FSM Dedekind infinite).

\begin{lemma} \label{lem'''}Let $f:A \to A$ be a function that is finitely supported by a certain finite set of atoms $S$. Then either $f|_{A \setminus S}=Id$ or $f|_{A \setminus S}$ is an one-element subset of $S$.
\end{lemma}
\emph{Proof of Lemma \ref{lem'''}}
Let $f:A \to A$ be a function that is finitely supported by the finite set of atoms $S$. We distinguish two cases:

I. There is $a \notin S$ with $f(a)=a$.  Then for each $b\notin S$ we have that $(a\, b) \in Fix(S)$, and so $f(b)=f((a\, b)(a))=(a\, b)(f(a))=(a\, b)(a)=b$. Thus, $f|_{A \setminus S}=Id$.

II.  For all $a \notin S$ we have $f(a) \neq a$.  We claim that $f(a) \in S$ for all $a \notin S$. Suppose, by contradiction, that $f(a)=b \in A \setminus S$ for a certain $a \notin S$. Thus, $(a\, b) \in Fix(S)$, and so $f(b)=f((a\, b)(a))=(a\, b)(f(a))=(a\, b)(b)=a$. Let us consider $c \in A \setminus S$, $c \neq a,b$.  Thus, $(a\, c) \in Fix(S)$, and so $f(c)=f((a\, c)(a))=(a\, c)(f(a))=(a\, c)(b)=b$. Furthermore, $(b\, c) \in Fix(S)$, and so $f(b)=f((b\, c)(c))=(b\, c)(f(c))=(b\, c)(b)=c$. However, $f(b)=a$ which contradicts the functionality of $f$. Thus $f(a) \in S$ for any $a \notin S$. If $x,y \notin S$, then we should have $f(x),f(y) \in S$, and so, because $(x\,y) \in Fix(S)$, we get $f(x)=f((x\, y)(y))=(x\, y)(f(y))=f(y)$ since both $x$ and $y$ belong to $A\setminus S$ which means they are different from $f(y)$ belonging to $S$. Therefore there is $x_{0} \in S$ such that $f|_{A \setminus S}$ $=\{x_{0}\}$.

\emph{Proof of this item}. Assume, by contradiction, that $A^{A}_{fs}$ contains an infinite, uniformly supported subset, meaning that there are infinitely many functions from $A$ to $A$ supported by the same finite set $S$. According to Lemma \ref{lem'''}, any $S$-supported function $f:A \to A$ should have the property that either $f|_{A \setminus S}=Id$ or $f|_{A \setminus S}$ is an one-element subset of $S$. A function from $A$ to $A$ is precisely characterized by the set of values it takes on the elements of $S$ and on the elements of  $A\setminus S$, respectively. For each possible definition of such an $f$ on $S$ we have at most $|S|+1$ possible ways to define $f$ on $A \setminus S$. Since we assumed that there exist infinitely many finitely supported functions from $A$ to $A$ supported by the same set $S$, there should exist infinitely many finitely supported functions from $S$ to $A$ supported by the set $S$. But this is a contradiction according to Lemma \ref{lem''} which states that $A^{S}_{fs}$ is does not contain an infinite uniformly supported subset (because $A$ does not contain an infinite uniformly supported subset). 

\item There is an equivariant bijective mapping between $(A^{n})^{A}_{fs}$ and $(A^{A}_{fs})^{n}$ defined as follows. If $f:A \to A^{n}$ is a finitely supported function with $f(a)=(a_{1},\ldots, a_{n})$, we associate to $f$ the Cartesian pair $(f_{1},\ldots, f_{n})$ where for each $i \in \mathbb{N}$, $f_{i}:A \to A$ is defined by $f_{i}(a)=a_{i}$ for all $a \in A$. We omit technical details since they are based only on the application of Proposition \ref{2.18'}. We proved above that $A^{A}_{fs}$ does not contain an infinite uniformly supported subset, and so neither $(A^{A}_{fs})^{n}$ contains an infinite uniformly supported subset by involving a similar proof as of Theorem \ref{ti1}(6) (see the proof of Lemma \ref{lem''}).

\item Assume by contradiction that $T_{fin}(A)^{A}$ contains an infinite 
$S$-uniformly supported subset. If $f:A \to T_{fin}(A)$ is a function 
supported by $S$, then consider $f(a)=x$ for some $a \notin S$. For $b 
\notin S$ we have $(a\,b) \in Fix(S)$, and so 
$f(b)=f((a\,b)(a))=(a\,b)\otimes f(a)=(a\,b)\otimes x$ which means 
$|f(a)|=|f(b)|$ for all $a,b \notin S$.  Each $S$-supported function $f:A 
\to T_{fin}(A)$ is fully described the values it takes on the elements of 
$S$ and on the elements of $A\setminus S$, respectively, i.e., by the 
elements of $f(S)$ and of $f(A\setminus S)$. More precisely, each 
$S$-supported function $f:A \to T_{fin}(A)$ can be uniquely decomposed 
into two $S$-supported functions $f|_{S}$ and $f|_{A \setminus S}$ (this 
follows from Proposition \ref{2.18'} and because both $S$ and $A \setminus 
S$ are supported by $S$). However, $f(A\setminus S) \subseteq A'^{n}$ for 
some $n \in \mathbb{N}$, where $A'^{n}$ is the set of all injective 
$n$-tuples of $A$. According to Lemma~\ref{lem''} we have at most finitely 
many $S$-supported functions from $S$ to $T_{fin}(A)$. According to item 5, 
we have at most finitely many $S$-supported functions from $A\setminus S$ 
to $A'^{n}$ for each fixed $n \in \mathbb{N}$. This is because $A'^{n}$ is 
a subset of $A^{n}$ and $A \setminus S$ is a subset of $A$, and so by 
involving Proposition \ref{pco1}(3) and (4) we find a finitely supported 
injection $\varphi$ from $(A'^{n})^{A \setminus S}$ and $(A^{n})^{A}$; if 
$\mathcal{K}$ was an infinite subset in $(A'^{n})^{A \setminus S}$ 
uniformly supported by $T$, then $\varphi(\mathcal{K})$ would be an 
infinite subset of $(A^{n})^{A}$ uniformly supported by $T \cup 
supp(\varphi)$.  Therefore, there should exist an infinite subset $M 
\subseteq \mathbb{N}$ such that we have at least one $S$-supported 
function $g:A\setminus S \to A'^{k}$ for any $k \in M$. We do not need to 
find a set of representatives for such $g$'s; we consider all of them.  
Fix $a \in A \setminus S$. For each of the above $g$'s (that form an 
$S$-supported family $\mathcal{F}$) we have that $g(a)$'s form an 
uniformly supported family (by $S \cup \{a\}$) of $T_{fin}(A)$, which is 
also infinite because tuples having different cardinalities are different 
and $M$ is infinite. However, we contradict the proof of item 3 stating 
that $T_{fin}(A)$ does not contain an infinite uniformly supported subset. 
Alternatively, one can remark that if $|S \cup \{a\}|=l$ with $l$ fixed, 
then there is $m\in M$ fixed with $m>l$. Moreover, $g(a)$ for some 
$g:A\setminus S \to A'^{m}$ in $\mathcal{F}$ (we need to select only a 
function from those functions $g:A\setminus S \to A'^{m}$ with $m$ fixed 
depending only on the fixed $l$, \emph{and not} a set of representatives 
for the entire family $(\{g:A\setminus S \to A'^{k}\})_{k \in M}$), which 
is an injective $m$-tuple of atoms, cannot be supported by $S \cup \{a\}$; 
thus, the set of all $g(a)$'s cannot be infinite and uniformly supported.

\item We can use a similar approach as in item 6, to prove that there exist at most finitely many $S$-supported functions from $A$ to $\wp_{fin}(A)$. For this we just replace $A'^{n}$ with the set of all $n$-sized subsets of $A$, $\wp_{n}(A)$. All it remains is to prove that, for each $n \in \mathbb{N}$, there cannot exist infinitely many functions $g:A \to \wp_{n}(A)$ supported by the same set $S'$. Fix $n \in \mathbb{N}$. Assume, by contradiction that there exist infinitely many functions $g:A \to \wp_{n}(A)$ supported by the same set $S'$. According to Lemma \ref{lem''} there are only finitely many functions from $S'$ to $\wp_{n}(A)$ supported by the same set of atoms, and so there should exist infinitely many functions $g:(A\setminus S') \to \wp_{n}(A)$ supported by $S'$. For such a $g$, let us fix an element $a\in A$ with $a\notin S'$. There exist $x_{1}, \ldots, x_{n} \in A$ fixed (depending only on the fixed $a$) and different such that $g(a)=\{x_{1}, \ldots, x_{n}\}$. Let $b$ be an arbitrary element 
from $A\setminus S'$, and so $(a\, b)\in Fix(S')$ which means $g(b)=g((a\,b)(a))=(a\,b) \star g(a)=(a\,b) \star \{x_{1}, \ldots, x_{n}\}=\{(a\,b)(x_{1}),\ldots, (a\,b)(x_{n})\}$. 

We analyze the two possibilities:

Case 1: One of $x_{1}, \ldots, x_{n}$ coincides to $a$. Suppose $x_{1}=a$. 
We claim that $x_{2}, \ldots, x_{n} \in S'$. Assume the contrary, that is, 
there exists $i \in \{2,\ldots,n\}$ such that $x_{i} \notin S'$. Without 
losing the generality suppose $x_{2} \notin S'$, which means $(a\,x_{2}) 
\in Fix(S')$, and so $g(x_{2})=g((a\,x_{2})(a))=(a\,x_{2}) \star 
g(a)=(a\,x_{2}) \star \{a,x_{2}, \ldots, x_{n}\}=\{a,x_{2}, \ldots, 
x_{n}\}$. Let $c \in A\setminus S'$ with $c$ different from $a,x_{2}, 
\ldots, x_{n}$. We have $g(c)=g((a\,c)(a))=(a\,c) \star g(a)=(a\,c) \star 
\{a,x_{2}, \ldots, x_{n}\}=\{c,x_{2}, \ldots, x_{n}\}$, and hence 
$g(x_{2})=g((c\,x_{2})(c))=(c\,x_{2}) \star g(c)=(c\,x_{2}) \star 
\{c,x_{2}, \ldots, x_{n}\}=\{c,x_{2}, \ldots, x_{n}\}$ which contradicts 
the functionality of $g$. Therefore, $g(b)=(b,x_{2},\ldots, x_{n})$ for 
all $b \in A \setminus S'$, and so only the selection of $x_{2}, \ldots 
x_{n}$ provides the distinction between $g$'s.  Since $S'$ is finite, 
$\{x_{2}, \ldots, x_{n}\}$ can be selected in $C_{|S'|}^{n-1}$ ways if 
$|S'|\geq n-1$, or in~$0$ ways otherwise.

Case 2: Consider now that all $x_{1}, \ldots, x_{n}$ are different from $a$.

Then $g(b)= \left\{ \begin{array}{ll}
\{x_{1}, \ldots, x_{n}\}, & \text{if}\: \text{$b \neq x_{1}, \ldots, x_{n}$ };\\
\{a,x_{2}, \ldots, x_{n}\}, & \text{if}\: \text{$x_{1} \notin S'$ and $b=x_{1}$ };\\ \ldots\\
\{x_{1},\ldots, x_{n-1}, a\}, & \text{if}\: \text{$x_{n} \notin S'$ and $b=x_{n}$}
\: .\end{array}\right.$ 

Since $x_{1}, \ldots, x_{n},a$ are fixed atoms, then $g(A \setminus S')$ 
is finite. However, $Im(g)$ should be supported by~$S'$. According to 
Proposition \ref{4.4-9}, since $Im(g)$ is finite, it should be uniformly 
supported by $S'$.  We obtain that $x_{1}, \ldots, x_{n} \in S'$, and so 
$g(A \setminus S') =\{x_{1}, \ldots, x_{n}\}$. Otherwise, if some $x_{i} 
\notin S'$, we would get $\{x_{1}, \ldots, a, \ldots, x_{n}\} \in Im(g)$ 
(where $a$ replaces $x_{i}$) and so $\{x_{1}, \ldots, a, \ldots, x_{n}\}$ 
is supported by $S'$. Again by Proposition \ref{4.4-9} we would have that 
$a$ is supported by $S'$ which means $\{a\}=supp(a)\subseteq S'$ 
contradicting the choice of $a$. Alternatively, for proving that all 
$x_{1}, \ldots, x_{n} \in S'$, assume by contradiction that one of them 
(say $x_{1}$) does not belong to $S'$. Let $c$ be an atom from $A 
\setminus S'$ with $c$ different from $a, x_{1},x_{2}, \ldots, x_{n}$. We 
have $g(c)=g((a\,c)(a))=(a\,c) \star g(a)=(a\,c) \star \{x_{1}, \ldots, 
x_{n}\}=\{x_{1}, \ldots, x_{n}\}$, and hence 
$g(x_{1})=g((c\,x_{1})(c))=(c\,x_{1}) \star g(c)=(c\,x_{1}) \star \{x_{1}, 
x_{2}, \ldots, x_{n}\}=\{c,x_{2}, \ldots, x_{n}\}$. However 
$g(x_{1})=\{a,x_{2}, \ldots, x_{n}\}$ which contradicts the functionality 
of $g$. Since $S'$ is finite, $\{x_{1}, \ldots, x_{n}\}$ can be selected 
in $C_{|S'|}^{n}$ ways $|S'|\geq n$ or in $0$ ways otherwise.

In either case, there couldn't exist infinitely many $g$'s supported by $S'$, and so for each $n \in \mathbb{N}$, there exist at most finitely many functions from $A$ to $\wp_{n}(A)$ supported by the same set of atoms. 

 Assume by contradiction that $\wp_{fin}(A)^{A}$ contains an infinite $S$-uniformly supported subset. If $f:A \to \wp_{fin}(A)$ is a function supported by $S$, then we have $|f(a)|=|(a\,b)\star f(a)|=|f((a\,b)(a))|=|f(b)|$ for all $a,b \notin S$.  According to Proposition \ref{2.18'} (since both $S$ and $A \setminus S$ are supported by $S$),  each $S$-supported function $f:A \to \wp_{fin}(A)$  is uniquely decomposed into two $S$-supported functions $f|_{S}$ and $f|_{A \setminus S}$. However $f(A\setminus S) \subseteq \wp_{n}(A)$ for some $n \in \mathbb{N}$. According to Lemma \ref{lem''} there are at most finitely many $S$-supported functions from $S$ to $\wp_{fin}(A)$. Furthermore, there exist at most finitely many $S$-supported functions from $A\setminus S$ to $\wp_{n}(A)$ for each fixed $n \in \mathbb{N}$.  Therefore, there should exist an infinite subset $M \subseteq \mathbb{N}$ such that we have at least one $S$-supported function $g:A\setminus S \to \wp_{k}(A)$ for any $k \in M$. Fix $a \in A \setminus S$. For each of the above $g$'s (that form an $S$-supported family $\mathcal{F}$) we have that $g(a)$'s form an uniformly supported family (by $S \cup \{a\}$) of $\wp_{fin}(A)$. If  $|S \cup \{a\}|=l$ with $l$ fixed, then there is $m\in M$ fixed with $m>l$. Moreover, $g(a)$ for  $g:A\setminus S \to \wp_{m}(A) \in \mathcal{F}$, which is an  $m$-sized subset of atoms, cannot be supported by $S \cup \{a\}$ (according to Proposition \ref{4.4-9}); thus, the set of all $g(a)$'s cannot be infinite and uniformly supported.

Analogously, there there exist at most finitely many $S$-supported functions from $A$ to $\wp_{cofin}(A)$ (using eventually the fact that there is an equivariant bijection $X \mapsto A\setminus X$ between  $\wp_{fin}(A)$ and $\wp_{cofin}(A)$).
Assume by contradiction that $\wp_{fs}(A)^{A}$ contains an infinite $S$-uniformly supported subset.  If $f:A \to \wp_{fs}(A)$ is a function supported by $S$, then consider $f(a)=X$ for some $a \notin S$. For $b \notin S$  we have $f(b)=(a\,b)\star X$ which means $f(A \setminus S)$ is formed only by finite subsets of atoms  if $X$ is finite, and $f(A \setminus S)$ is formed only by cofinite subsets of atoms if $X$ is cofinite. Thus, whenever $f:A \to \wp_{fs}(A)$ is a function supported by $S$, we have either $f(A \setminus S) \subseteq \wp_{fin}(A)$ or $f(A \setminus S) \subseteq \wp_{cofin}(A)$. Each $S$-supported function $f:A \to \wp_{fs}(A)$  is fully described by $f(S)$ and $f(A\setminus S)$.  According to Lemma \ref{lem''} we have at most finitely many $S$-supported functions from $S$ to $\wp_{fs}(A)$. Furthermore, we have at most finitely many $S$-supported functions from $A\setminus S$ to $\wp_{fin}(A)$, and at most finitely many $S$-supported functions from $A\setminus S$ to $\wp_{cofin}(A)$. Thus, $\wp_{fs}(A)^{A}$ does not contain an infinite uniformly supported subset.

\item The sets $\wp_{fin}(A)$, $\wp_{cofin}(A)$, $\wp_{fin}(\wp_{fs}(A))$, $\wp_{fin}(\wp_{cofin}(A))$, $\wp_{fin}(\wp_{fin}(A))$, $\wp_{fin}(A^{A}_{fs})$ do not contain  infinite uniformly supported subsets, and so they are not FSM Dedekind infinite (Theorem \ref{ti1}(5)).

\item Directly from Theorem \ref{ti1}(5).

\item According to Theorem \ref{ti1}(6).

\item According to Theorem \ref{ti1}(7).
\end{enumerate}
\end{proof}

\begin{corollary} There exist two FSM sets that whose cardinalities 
incomparable via the relation $\leq$ on cardinalities, and none of them is 
FSM Dedekind infinite. 
\end{corollary}

\begin{proof}
According to Corollary \ref{ti2}, none of the sets $A \times A$ and 
$\wp_{fs}(A)$ is FSM Dedekind infinite. According to Theorem 
\ref{cardord1}, there does not exist a finitely supported injective 
mapping $f: A \times A \to \wp_{fs}(A)$. According to Lemma~11.10 from 
\cite{jech} that is preserved in FSM (proof omitted) there does not exist 
a finitely supported injective mapping $f: \wp_{fs}(A) \to A \times A$.
\end{proof}

\begin{corollary} \label{ti3} The following sets and all of their supersets, their powersets and the families of their finite subsets, are both FSM usual infinite and  FSM Dedekind infinite. 
\begin{enumerate}
\item The invariant sets $\wp_{fs}(\wp_{fs}(A))$,  $\wp_{fs}(\wp_{fin}(A))$ and $\mathbb{N}$.
\item The set of all finitely supported mappings from $X$ to $Y$, and the set of all finitely supported mappings from $Y$ to $X$, where $X$ is a finitely supported subset of an invariant set with at least two elements, and $Y$ is an FSM Dedekind infinite set. 
\item The set of all finitely supported functions $f:\wp_{fin}(Y) \to X$ and the set of all finitely supported functions $f:\wp_{fs}(Y) \to X$, where $Y$ is an infinite finitely supported subset of an invariant set, and  $X$ is a finitely supported subset of an invariant set with at least two elements.
\item The set $T^{\delta}_{fin}(A)=\underset{n\in \mathbb{N}}{\cup}A^{n}$ of all finite tuples of atoms (not necessarily injective). 
\end{enumerate}
\end{corollary}

\begin{proof}
\begin{enumerate}
\item This follows from Theorem \ref{ti1}(3) and Theorem \ref{ti1} (2).

\item Let $(y_{n})_{n \in \mathbb{N}}$ be an injective, uniformly 
supported, countable sequence in $Y$ (that exists from Theorem~\ref{ti1}(1)). 
Thus, each $y_{n}$ is supported by the same set $S$ of 
atoms. In $Y^{X}$ we consider the injective family $(f_{n})_{n \in 
\mathbb{N}}$ of functions from $X$ to $Y$ where for each $i \in 
\mathbb{N}$ we define $f_{i}(x)=y_{i}$ for all $x \in X$. According to 
Proposition~\ref{2.18'}, each $f_{i}$ is supported by $S$, and so is the 
infinite family $(f_{n})_{n \in \mathbb{N}}$, meaning that there is an 
$S$-supported injective mapping from $\mathbb{N}$ to $Y^{X}$. In this case 
it is necessary to require only that $X$ is non-empty.

Fix two different elements $x_{1}, x_{2} \in X$. Take  $\mathcal{F}=(y_{n})_{n \in \mathbb{N}}$  an injective,  uniformly supported, countable sequence in $Y$. In $X^{Y}$ we consider the injective family $(g_{n})_{n \in \mathbb{N}}$ of functions from  $Y$ to $X$  where for each $i \in \mathbb{N}$ we define $g_{i}(y)=\left\{ \begin{array}{ll}
x_{1} & \text{if}\: y=y_{i}\\
x_{2} & \text{if}\: y=y_{j} \; \text{with}\; j \neq i, \; \text{or}\;  y\notin \mathcal{F}\end{array}
\right.$. According to  Proposition \ref{2.18'}, each $g_{i}$ is supported by the finite set $supp(x_{1}) \cup supp(x_{2}) \cup supp(\mathcal{F})$, and so the infinite family $(g_{n})_{n \in \mathbb{N}}$ is uniformly supported meaning that there is an injective mapping from $\mathbb{N}$ to $X^{Y}$ supported by $supp(x_{1}) \cup supp(x_{2}) \cup supp(\mathcal{F})$. 

\item From Theorem \ref{comp}, there exists a one-to-one mapping from $\wp_{fs}(U)$ onto 
$\{0,1\}^{U}_{fs}$ for an arbitrary finitely supported subset of an invariant set $U$. Fix two distinct elements $x_{1},x_{2} \in X$. There exists a finitely supported (by $supp(x_{1}) \cup supp(x_{2})$) bijective mapping from $\{0,1\}^{U}_{fs}$ to $\{x_{1},x_{2}\}^{U}_{fs}$ which associates to each $f \in \{0,1\}^{U}_{fs}$ an element $g \in \{x_{1},x_{2}\}^{U}_{fs}$ defined by $g(x)=\left\{ \begin{array}{ll}
x_{1} & \text{for}\: f(x)=0\\
x_{2} & \text{for}\: f(x)=1 \end{array}\right.$ for all $x \in U$ and supported by $supp(x_{1}) \cup supp(x_{2}) \cup supp(f)$. Obviously, there is a finitely supported injection between $\{x_{1},x_{2}\}^{U}_{fs}$ and $X^{U}_{fs}$. Thus, there is a finitely supported injection from $\wp_{fs}(U)$ into $X^{U}_{fs}$. If we take $U=\wp_{fin}(Y)$ or $U=\wp_{fs}(Y)$, the result follows from Theorem \ref{ti1}(1), Theorem \ref{ti1}(2) and Theorem \ref{ti1}(3).

\item Fix $a \in A$ and $i \in \mathbb{N}$. We consider the tuple $x_{i}=(a,\ldots,a) \in A^{i}$. Clearly $x_{i}$ is supported by $\{a\}$ for each $i \in \mathbb{N}$, and so $(x_{n})_{n \in \mathbb{N}}$ is a uniformly supported subset of $T^{\delta}_{fin}(A)$.
\end{enumerate}
\end{proof}

\begin{proposition} \label{propro} Let $X$ be a finitely supported subset of an invariant set such that $\wp_{fs}(X)$ is not FSM Dedekind infinite. Then each finitely supported surjective mapping $f:X \to X$ should be injective. 
\end{proposition} 

\begin{proof}Let $f: X \to X$ be a finitely supported surjection.
Since $f$ is surjective, we can define the function $g:\wp_{fs}(X) 
\rightarrow \wp_{fs}(X)$ by $g(Y) = f^{-1}(Y)$ for all $Y\in \wp_{fs}(X)$ which is finitely supported and injective according to Lemma \ref{lemlem}. Since $\wp_{fs}(X)$ is not FSM Dedekind infinite, it follow that $g$ is surjective.

Now let us consider two elements $a,b \in X$ such that $f(a)=f(b)$. We prove by 
contradiction that $a=b$. Suppose that $a \neq b$. Let us consider $Y=\{a\}$ 
and $Z=\{b\}$. Obviously, $Y,Z \in \wp_{fs}(X)$. Since $g$ is surjective, 
for $Y$ and $Z$ there exist $Y_{1}, Z_{1} \in \wp_{fs}(X)$ such that 
$f^{-1}(Y_{1})=g(Y_{1})=Y$ and $f^{-1}(Z_{1})=g(Z_{1})=Z$. We know that 
$f(Y) \cap f(Z)= \{f(a)\}$. Thus, $f(a) \in f(Y)=f(f^{-1}(Y_{1})) \subseteq 
Y_{1}$. Similarly, $f(a) =f(b) \in f(Z)=f(f^{-1}(Z_{1})) \subseteq Z_{1}$, 
and so $f(a) \in Y_{1} \cap Z_{1}$.  Thus, $a \in f^{-1}(Y_{1} \cap 
Z_{1})=f^{-1}(Y_{1}) \cap f^{-1}(Z_{1})=Y \cap Z$. However, since we assumed 
that $a \neq b$, we have that $Y \cap Z = \emptyset$, which represents a 
contradiction. It follows that $a=b$, and so $f$ is injective. 
\end{proof}

\begin{proposition} \label{propro''}
\begin{enumerate}
\item Let $X$ be a finitely supported subset of an invariant set. If $\wp_{fin}(X)$ is FSM Dedekind infinite, then $X$ should be FSM non-uniformly amorphous, meaning that $X$ should contain two disjoint, infinite, uniformly supported subsets. 
\item Let $X$ be a finitely supported subset of an invariant set. If $\wp_{fs}(X)$ is FSM Dedekind infinite, then $X$ should be FSM non-amorphous, meaning that $X$ should contain two disjoint, infinite, finitely supported supported subsets. The reverse implication is not valid.
\end{enumerate}
\end{proposition}
\begin{proof}
1. Assume that 
$(X_{n})_{n \in \mathbb{N}}$ is a countable family of different finite  subsets of $X$ 
such that the mapping $n\mapsto X_{n}$ is finitely supported. Thus, each $X_{n}$ is supported by the same set $S=supp(n \mapsto X_{n})$. Since each $X_{n}$ is finite (and the support of a finite set coincides with the union of the supports of its elements), as in the proof of Lemma  \ref{lem4}, we have that $\underset{n \in \mathbb{N}}\cup X_{n}$
is uniformly supported by $S$. Furthermore, $\underset{n \in \mathbb{N}}\cup X_{n}$ is infinite since all $X_{i}$ are pairwise different.  Moreover, the countable sequence $(Y_{n})_{n \in \mathbb{N}}$ defined by $Y_{n}=X_{n} \setminus \underset{m<n}\cup X_{m}$ is a uniformly supported (by $S$) sequence of pairwise disjoint uniformly supported sets with $\underset{n \in \mathbb{N}}\cup X_{n}=\underset{n \in \mathbb{N}}\cup Y_{n}$. Again since each $Y_{n}$ is finite (and the support of a finite set coincides with the union of the supports of its elements), any element belonging to a set from the sequence $(Y_{n})_{n \in \mathbb{N}}$ is $S$-supported. Since the union of all  $Y_{n}$ is infinite, and each $Y_{n}$ is finite, there should exist infinitely many terms from the sequence $(Y_{n})_{n \in \mathbb{N}}$ that are non-empty. Assume that $(Y_{n})_{n \in M \subseteq \mathbb{N}}$ with $M$ infinite is a subset of $(Y_{n})_{n \in \mathbb{N}}$ formed by non-empty terms. Let $U_{1}=\{\cup Y_{k}\,|\, k \in M, k \;\text{is odd}\}$ and $U_{2}=\{\cup Y_{k}\,|\, k \in M, k\; \text{is even}\}$. Then $U_{1}$ and $U_{2}$ are disjoint, uniformly $S$-supported and infinite subsets of $X$. 

2. Assume that $\wp_{fs}(X)$ is FSM Dedekind infinite. As in the proof of Theorem \ref{ti1}(8), we can define a uniformly supported, countable family $(Y_{n})_{n \in \mathbb{N}}$  of subsets of $X$ that are non-empty and pairwise disjoint. Let $V_{1}=\{\cup Y_{k}\,|\,  k \;\text{is odd}\}$ and $V_{2}=\{\cup Y_{k}\,|\, k\; \text{is even}\}$. Then $V_{1}$ and $V_{2}$ are disjoint, infinite subsets of $X$. Since each $Y_{i}$ is supported by $S'=supp(n \mapsto Y_{n})$ we have $\pi \star Y_{i}=Y_{i}$ for all $i \in \mathbb{N}$ and $\pi \in Fix(S')$.  Fix $\pi \in Fix(S')$ and $x \in V_{1}$. Thus, there is $l \in \mathbb{N}$ such that $x \in Y_{2l+1}$. We obtain $\pi \cdot x \in \pi \star Y_{2l+1}=Y_{2l+1}$, and so $\pi \cdot x \in V_{1}$. Thus, $V_{1}$ is supported by $S'$. Analogously, $V_{2}$ is supported by $S'$, and so $X$ is FSM non-amorphous. 

Conversely, the set $A+A=\{0,1\}\times A$ (the disjoint union of $A$ and $A$) is obviously non-amorphous  because $\{(0,a)\,|\,a \in A\}$ is equivariant, infinite and coinfinite. One can define the equivariant bijection $f: \wp_{fs}(A) \times \wp_{fs}(A) \to  \wp_{fs}(\{0,1\}\times A)$ by $f(U,V)=\{(0,x)\,|\,x \in U\} \cup \{(1,y)\,|\,y \in V\}$ for all $U,V \in \wp_{fs}(A)$. Clearly $f$ is equivariant because for each $\pi \in S_{A}$ we have $f(\pi \star U,\pi \star V)=\pi \star f(U,V)$. However, $\wp_{fs}(A) \times \wp_{fs}(A)$ is not FSM Dedekind infinite according to Corollary \ref{ti2}(2) and Theorem \ref{ti1}(6).
\end{proof}

It is also worth noting that non-uniformly amorphous FSM sets are non-amorphous FSM sets since uniformly supported sets are obviously finitely supported. The converse however is not valid since $\wp_{fin}(A)$ is non-amorphous but it has no infinite uniformly supported subset (the only finite subsets of atoms supported by a finite set $S$ of atoms being the subsets of $S$), and so it cannot be non-uniformly amorphous. 

\begin{corollary} \label{propro'} 
Let $X$ be a finitely supported amorphous subset of an invariant set (i.e. 
any finitely supported subset of~$X$ is either finite or cofinite). Then 
each finitely supported surjective mapping $f:X \to X$ should be injective. 
\end{corollary}

\begin{proof}
Since any finitely supported subset of $X$ is either finite or cofinite, 
then any uniformly supported subset of $X$ is either finite or cofinite. 
From Proposition \ref{propro''}, $\wp_{fin}(X)$ is not FSM Dedekind 
infinite. For the rest of the proof we follow step-by-step the proof of 
Proposition \ref{propro} (and of Lemma \ref{lemlem}). If $X$ is finite, we 
are done, so assume $X$ is infinite.  If $Y \in \wp_{fin}(X)$, then 
$f^{-1}(Y) \in \wp_{fs}(X)$ (supported by $supp(f) \cup supp(X) \cup 
supp(Y)$). Since $X$ is amorphous, it follows that $f^{-1}(Y)$ is either 
finite or cofinite. If $f^{-1}(Y)$ is cofinite, then its complementary 
$\{x \in X \;|\; f(x) \notin Y\}$ is finite. This means that all but 
finitely many elements in $X$ would have their image under $f$ belonging 
to the finite set~$Y$. Therefore, $Im(f)$ would be a finite subset of $X$, 
which contradicts the surjectivity of $f$. Thus, $f^{-1}(Y)$ is a finite 
subset of $X$. In this sense we can well-define define the function 
$g:\wp_{fin}(X) \rightarrow \wp_{fin}(X)$ by $g(Y) = f^{-1}(Y)$ which is 
supported by $supp(f) \cup supp(X)$ and injective. Since $\wp_{fin}(X)$ is 
not FSM Dedekind infinite, it follow that $g$ is surjective, and so $f$ is 
injective exactly as in the last paragraph of the proof of Proposition 
\ref{propro}. 
\end{proof}

\begin{proposition}
\begin{enumerate}
\item Let $X$ be an FSM Dedekind infinite set. Then there exists a 
finitely supported surjection $j:X \to \mathbb{N}$. The reverse 
implication is not valid.

\item If $X$ is a finitely supported subset of an invariant set such that there exists a finitely supported surjection $j:X \to \mathbb{N}$, then $\wp_{fs}(X)$ is FSM Dedekind infinite. The reverse implication is also valid. 
\end{enumerate}
\end{proposition}

\begin{proof}
1. Let $X$ be an FSM Dedekind infinite set. According to Theorem \ref{ti1}(1), there is a finitely supported injection $i:\mathbb{N} \to X$. Let us fix $n_{0} \in \mathbb{N}$. We define the function $j: X \to \mathbb{N}$ by 
\[
j(x)=\left\{ \begin{array}{ll}
i^{-1}(x), & \text{if}\: \text{$x\in Im(i)$ };\\
n_{0}, & \text{if}\: \text{$x\notin Im(i)$}\: .\end{array}\right. \]

Since $Im(i)$ is supported by $supp(i)$ and $n_{0}$ is empty supported, by verifying the condition in Proposition \ref{2.18'}  we have that $j$ is supported by $supp(i) \cup supp(X)$. Indeed, when   $\pi \in Fix(supp(i) \cup supp(X))$, then $x \in Im(i) \Leftrightarrow \pi \cdot x \in Im(i)$, and $n=i^{-1}(\pi \cdot x) \Leftrightarrow i(n)=\pi \cdot x \Leftrightarrow \pi^{-1} \cdot i(n)=x \Leftrightarrow i(\pi^{-1} \diamond n)=x \Leftrightarrow i(n)=x \Leftrightarrow n=i^{-1}(x)$, where $\diamond$ is the trivial action on $\mathbb{N}$; similarly $y \notin Im(i) \Leftrightarrow \pi \cdot y \notin Im(i)$ and $j(\pi \cdot y)=n_{0}=\pi \diamond n_{0}=\pi \diamond j(y)$. Clearly, $j$ is surjective. However, the reverse implication is not valid because the mapping $f:\wp_{fin}(A) \to \mathbb{N}$ defined by $f(X)=|X|$ for all $X \in \wp_{fin}(A)$ is equivariant and surjective, but $\wp_{fin}(A)$ is not FSM Dedekind infinite.

2. Suppose now there exists a finitely supported surjection $j:X \to \mathbb{N}$. Clearly, for any $n \in \mathbb{N}$, the set $j^{-1}(\{n\})$ is non-empty and supported by $supp(j)$. Define $f: \mathbb{N} \to \wp_{fs}(X)$ by $f(n)=j^{-1}(\{n\})$. For $\pi \in Fix(supp(j))$ and an arbitrary $n \in \mathbb{N}$ we have $j(x)=n \Leftrightarrow j(\pi^{-1} \cdot x)=n$, and so $x \in j^{-1}(\{n\}) \Leftrightarrow \pi^{-1} \cdot x \in j^{-1}(\{n\})$, which means $f(n) = \pi \star f(n)$ for all $n \in \mathbb{N}$, and so $f$ is supported by $supp(j)$. Since $f$ is also injective, by Theorem \ref{ti1}(1) we have that $\wp_{fs}(X)$ is FSM Dedekind infinite.

Conversely, assume that $\wp_{fs}(X)$ is FSM Dedekind infinite. As in the proof of Theorem \ref{ti1}(8), we can define a uniformly supported, countable family $(Y_{n})_{n \in \mathbb{N}}$  of subsets of $X$ that are non-empty and pairwise disjoint. The mapping $f$ can be defined by 
$
f(x)=\left\{ \begin{array}{ll}
n, & \text{if}\; \exists n. x\in Y_{n} ;\\
0, & \text{otherwise} \end{array}\right.$, and, obviously, $f$ is supported by $supp(n \mapsto Y_{n})$.
\end{proof}

\begin{proposition}Let $X$ be an infinite finitely supported subset of an invariant set. Then there exists a finitely supported surjection $f:\wp_{fs}(X) \to \mathbb{N}$.
\end{proposition}

\begin{proof}Let $X_{i}$ be the set of all $i$-sized subsets from $X$, 
i.e. $X_{i}=\{Z \subseteq X\,|\,|Z|=i\}$. 
The family $(X_{i})_{i \in \mathbb{N}}$ is uniformly supported by $supp(X)$ 
and all $X_{i}$ are non-empty and pairwise disjoint. 
Define the mapping $f$ by 
$f(Y)=\left\{ \begin{array}{ll}
n, & \text{if}\: Y\in X_{n};\\
0, & \text{if}\; Y\; \text{is infinite} .\end{array}\right.$ 
According to Proposition \ref{2.18'}, $f$ is supported by $supp(X)$ (since 
any $X_{n}$ is supported by $supp(X)$) and it is surjective. We actually 
proved the existence of a finitely supported surjection from 
$\wp_{fin}(X)$ onto~$\mathbb{N}$. 
\end{proof}

The sets $A$ and $\wp_{fin}(A)$ are both FSM usual infinite and none of 
them is FSM Dedekind infinite. We prove below that~$A$ is not FSM 
ascending infinite, while $\wp_{fin}(A)$ is FSM ascending infinite.

\begin{proposition} \label{tari}
\begin{itemize}
\item The set $A$ is not FSM ascending infinite.
\item Let $X$ be a finitely supported subset of an invariant set $U$. If $X$ is FSM usual infinite, then the set $\wp_{fin}(X)$ is FSM ascending infinite. 
\end{itemize}
\end{proposition}
\begin{proof}
In order to prove that $A$ is not FSM ascending infinite, we prove firstly that each finitely supported increasing countable chain of finitely supported subsets of $A$ must be stationary. Indeed, if there exists an increasing countable chain
$X_{0}\subseteq X_{1}\subseteq\ldots\subseteq A$ such that $n \mapsto X_{n}$ is finitely supported, then, according to Proposition \ref{2.18'} and because $\mathbb{N}$ is a trivial invariant set, each element $X_{i}$ of the chain must be
supported by the same $S=supp(n \mapsto X_{n})$. However, there are only finitely many such
subsets of $A$ namely the subsets of $S$ and the supersets of $A\setminus S$.
Therefore the chain is finite, and, because it is ascending, there exists 
$n_{0}\in\mathbb{N}$ such that $X_{n}=X_{n_{0}},\forall n\geq n_{0}$. Now, 
let $Y_{0}\subseteq Y_{1}\subseteq\ldots\subseteq Y_{n}\subseteq\ldots$ be 
a finitely supported countable chain with $A\subseteq \underset{n \in 
\mathbb{N}}{\cup} Y_{n}$. Then $A \cap Y_{0}\subseteq A \cap 
Y_{1}\subseteq\ldots\subseteq A \cap Y_{n}\subseteq\ldots \subseteq A$ is 
a finitely supported countable chain of subsets of $A$ (supported by 
$supp(n \mapsto Y_{n})$) which should be stationary (finite). Furthermore, 
since $\underset{i \in \mathbb{N}}{\cup}(A\cap Y_{i})=A \cap(\underset{i 
\in \mathbb{N}}{\cup} Y_{i})=A$, there is some $k_{0}$ such that $A \cap 
Y_{k_{0}}=A$, and so $A \subseteq Y_{k_{0}}$. Thus,~$A$ is not FSM 
ascending infinite.

We know that $\wp_{fin}(X)$ is a subset of the invariant set 
$\wp_{fin}(U)$ supported by $supp(X)$. Let us consider 
$X_{n}\!=\!\{Z\!\in\!\wp_{fin}(X)\,|\,|Z|\!\leq\! n\}$. Clearly, 
$X_{0}\subseteq X_{1}\subseteq\ldots\subseteq X_{n}\subseteq\ldots$. 
Furthermore, because permutations of atoms are bijective, we have that for 
an arbitrary $k \in \mathbb{N}$, $|\pi \star Y| =|Y|$ for all $\pi \in 
S_{A}$ and all $Y \in X_{k}$, and so $\pi \star Y \in X_{k}$ for all $\pi 
\in Fix(supp(X))$ and all $Y \in X_{k}$. Thus, each $X_{k}$ is a subset of 
$\wp_{fin}(X)$ finitely supported by $supp(X)$, and so $(X_{n})_{n \in 
\mathbb{N}}$ is finitely (uniformly) supported by $supp(X)$. Obviously, 
$\wp_{fin}(X)=\underset{n \in \mathbb{N}}{\cup} X_{n}$. However, there 
exists no $n\in\mathbb{N}$ such that $\wp_{fin}(X)=X_{n}$. Thus, 
$(\wp_{fin}(X), \star)$ is FSM ascending infinite. 
\end{proof}

\begin{theorem} \label{Tt} Let $X$ be a finitely supported subset of an invariant set $(Z, \cdot)$. 

\begin{enumerate}\item If $X$ is FSM Dedekind infinite, then $X$ is FSM Mostowski infinite. 
\item If $X$ is FSM Mostowski infinite, then $X$ is FSM Tarski II infinite. The reverse implication is not valid. 
\end{enumerate}
\end{theorem}
\begin{proof}
1. Suppose $X$ is FSM Dedekind infinite. According to Theorem \ref{ti1}(1) there exists an uniformly supported infinite injective sequence $T=(x_{n})_{n \in \mathbb{N}}$ of elements from $X$. Thus, each element of $T$ is supported by $supp(T)$ and there is a bijective correspondence between $\mathbb{N}$ and $T$ defined as $n \mapsto x_{n}$ which is supported by $supp(T)$.  If we define the relation $\sqsubset$ on $T$ by: $x_{i} \sqsubset x_{j}$ if and only if $i<j$, we have that $\sqsubset$ is a (strict) total order relation supported by $supp(T)$. Thus $T$ is an infinite finitely supported (strictly) totally ordered subset of $X$, and so $X$ is FSM Mostowski infinite since any strict total order can be extended to a total order. 

2. Suppose that $X$ is not FSM Tarski II infinite. Then every non-empty 
finitely supported family of finitely supported subsets of $X$ which is 
totally ordered by inclusion has a maximal element under inclusion. Let 
$(U, <)$ be a finitely supported strictly totally ordered subset of $X$ 
(any total order relation induces a strict total order relation). We prove 
that~$U$ is finite, and so $X$ is not FSM Mostowski infinite. In this 
sense it is sufficient to prove that $<$ and $>$ are well-orderings.  
Since both of them are (strict) total orderings, we need to prove that any 
finitely supported subset of $U$ has a least and a greatest element wrt 
$<$, i.e. a minimal and a maximal element (because $<$ is total). Let $Y$ 
be a finitely supported subset of $U$. The set $\downarrow z=\{y \in 
Y\,|\,y<z\}$ is supported by $supp(z) \cup supp(Y) \cup supp(<)$ for all 
$z \in Y$. The family $T=\{\downarrow z\,|\,z \in Y\}$ is itself finitely 
supported by $supp(Y) \cup supp(<)$ because for all $\pi \in Fix(supp(Y) 
\cup supp(<))$ we have $\pi \cdot \downarrow z=\downarrow \pi \cdot z$. 
Since $<$ is transitive, we have that $T$ is (strictly) totally ordered by 
inclusion, and so it has a maximal element, which means $Y$ has a maximal 
element. Analogously, the set $\uparrow z=\{y \in Y\,|\,z<y\}$ is 
supported by $supp(z) \cup supp(Y) \cup supp(<)$ for all $z \in Y$ and the 
family $T'=\{\uparrow z\,|\,z \in Y\}$ is itself finitely supported by 
$supp(Y) \cup supp(<)$ because for all $\pi \in Fix(supp(Y) \cup supp(<))$ 
we have $\pi \cdot \uparrow z=\uparrow \pi \cdot z$. The family $T'$ is 
(strictly) totally ordered by inclusion, and so it has a maximal element, 
from which $Y$ has a minimal element. We used the obvious properties $z<t$ 
if and only if $\downarrow z \subset \downarrow t$, and $z<t$ if and only 
if $\uparrow t \subset \uparrow z$.

Conversely, according to Proposition \ref{tari}, $\wp_{fin}(A)$ is FSM 
ascending infinite, and so it is FSM Tarski II infinite. However,
$\wp_{fin}(A)$ is not FSM Mostowski infinite, according to Corollary~\ref{pTII}.
\end{proof}

\begin{proposition} \label{propamo} 
Let $X$ be a finitely supported subset of an invariant set $(Z, \cdot)$. 
If $X$ is FSM Mostowski infinite, then~$X$ is non-amorphous meaning that 
$X$ can be expressed as a disjoint union of two infinite finitely 
supported subsets. The reverse implication is not valid. 
\end{proposition}

\begin{proof}Suppose that there is an infinite finitely supported totally ordered subset $(Y, \leq)$ of $X$. Assume, by contradiction, that $Y$ is amorphous, meaning that any finitely supported subset of $Y$ is either finite or cofinite. As in the proof of Theorem \ref{Tt} (without making the requirement that $\leq$ is strict, which anyway would not essentially change the proof), for $z \in Y$ we define the finitely supported subsets $\downarrow z=\{y \in Y\,|\,y \leq z\}$ and $\uparrow z=\{y \in Y\,|\,z \leq y\}$ for all $z \in Y$. We have that the mapping $z \mapsto \downarrow z$ from $Y$ to $T=\{\downarrow z\,|\,z \in Y\}$ is itself finitely supported by $supp(Y) \cup supp(\leq)$. Furthermore it is bijective, and so $T$ is amorphous. Thus, any subset $Z$ of $T$ is either finite or cofinite, and obviously any subset $Z$ of $T$ is finitely supported. Analogously, the mapping $z \mapsto \uparrow z$  from $Y$ to $T'=\{\uparrow z\,|\,z \in Y\}$ is finitely supported and bijective, which means that any subset of $T'$ is either finite or cofinite, and clearly any subset of $T'$ is finitely supported. 

We distinguish the following two cases:

1. There are only finitely many elements $x_{1}, \ldots, x_{n} \in Y$ such 
that $\downarrow x_{1}, \ldots, \downarrow x_{n}$ are finite. Thus, for $y 
\in U= Y \setminus \{x_{1}, \ldots, x_{n}\}$ we have $\downarrow y$ 
infinite. Since $\downarrow y$ is a subset of $Y$, it should be cofinite, 
and so $\uparrow y$ is finite (because $\leq$ is a total order relation). 
Let $M=\{\uparrow y\,|\,y \in U\}$. As in Theorem \ref{Tt} we have that 
$M$ is totally ordered with respect to sets inclusion. Furthermore, for an 
arbitrary $y \in U$ we cannot have $y \leq x_{k}$ for some $k \in 
\{1,\ldots,n\}$ because $\downarrow y$ is infinite, while $\downarrow 
x_{k}$ is finite, and so $\uparrow y$ is a subset of $U$. Thus, $M$ is an 
infinite, finitely supported (by $supp(U) \cup supp(\leq)$), totally 
ordered family formed by finite subsets of $U$. Since $M$ is finitely 
supported, for each $y \in U$ and each $\pi \in Fix(supp(M))$ we have $\pi 
\cdot \uparrow y \in M$. Since $\uparrow y$ is finite, we have that $\pi 
\cdot \uparrow y$ is finite having the same number of elements as 
$\uparrow y$. Since $\pi \cdot \uparrow y$ and $\uparrow y$ are comparable 
via inclusion, they should be equal. Thus,~$M$ is uniformly supported. 
Since $\leq$ is a total order, for $\pi \in Fix(supp(\uparrow y))$ we have 
$\uparrow \pi \cdot y=\pi \cdot \uparrow y=\uparrow y$, and so $\pi \cdot 
y=y$, from which $supp(y) \subseteq supp(\uparrow y)$. Thus, $U$ is 
uniformly supported.  Since any element of $U$ has only a finite number of 
successors (leading to the conclusion that $\geq$ is an well-ordering on 
$U$ uniformly supported by $supp(U)$) and~$U$ is \emph{uniformly 
supported}, we can define an order monomorphism between $\mathbb{N}$ and 
$U$ which is supported by $supp(U)$. For example, choose $u_{0}\neq u_{1} 
\in U$, then let $u_{2}$ be \emph{the greatest element} (w.r.t. $\leq$) in 
$U\setminus \{u_{0}, u_{1}\}$, $u_{3}$ be \emph{the greatest element} in 
$U\setminus \{u_{0}, u_{1}, u_{2}\}$ (no choice principle is used since 
$\geq$ is an well-ordering, and so such a \emph{greatest} element is 
precisely defined), and so on, and find an infinite, uniformly supported 
countable sequence $u_{0}, u_{1}, u_{2},\ldots$.  Since $\mathbb{N}$ is 
non-amorphous (being expressed as the union between the even elements and 
the odd elements), we conclude that $U$ is non-(uniformly) amorphous 
containing two infinite uniformly supported disjoint subsets.

2. We have cofinitely many elements $z$ such that $\downarrow z$ is 
finite. Thus, there are only finitely many elements $y_{1}, \ldots, y_{m} 
\in Y$ such that $\downarrow y_{1}, \ldots, \downarrow y_{m}$ are 
infinite. Since every infinite subset of $Y$ is cofinite, only $\uparrow 
y_{1}, \ldots, \uparrow y_{m}$ are finite. Let $z \in Y \setminus \{y_{1}, 
\ldots, y_{m}\}$ which means $\uparrow z$ infinite. Since $\uparrow z$ is 
a subset of $Y$ it should be cofinite, and so $\downarrow z$ is finite. As 
in the above item, the set $M'=\{\downarrow z\,|\,z \in Y \setminus 
\{y_{1}, \ldots, y_{m}\} \}$ is an infinite, finitely supported, totally 
ordered (by inclusion) family of finite sets, and so it has to be 
uniformly supported, from which $Y \setminus \{y_{1}, \ldots, y_{m}\}$ is 
uniformly supported, and so $\leq$ is an FSM well ordering on $Y \setminus 
\{y_{1}, \ldots, y_{m}\}$. Therefore, $Y \setminus \{y_{1}, \ldots, 
y_{m}\}$ has an infinite, uniformly supported, countable subset, and so $Y 
\setminus \{y_{1}, \ldots, y_{m}\}$ is non-(uniformly) amorphous 
containing two infinite uniformly supported disjoint subsets.
Thus $Y$ is non-amorphous, and so $X$ is non-amorphous.

Conversely, the set $A+A$ (the disjoint union of $A$ and $A$) is obviously non-amorhpous because because $\{(0,a)\,|\,a \in A\}$ is equivariant, infinite and coinfinite. However, if we assume there exists a finitely supported total order relation on an infinite subset of $A+A$, then there should exist an infinite, finitely supported, total order on at least one of the sets $\{(0,a)\,|\,a \in A\}$ or $\{(1,a)\,|\,a \in A\}$, which leads to an infinite finitely supported total order relation on $A$. However $A$ is not FSM Mostowski infinite by Corollary \ref{pTII}. 
\end{proof}

\begin{theorem} \label{Tti} Let $X$ be a finitely supported subset of an invariant set $(Z, \cdot)$. If $X$ contains no infinite uniformly supported subset, then $X$ is not FSM Mostowski infinite. 
\end{theorem}
\begin{proof}Assume, by contradiction, that $X$ is FSM Mostowski infinite, meaning that $X$ contains an infinite, finitely supported, totally ordered subset $(Y, \leq)$.  We claim that $Y$ is uniformly supported by $supp(\leq) \cup supp(Y)$. Let $\pi \in Fix (supp(\leq) \cup supp(Y))$ and let $y \in Y$ an arbitrary element. Since $\pi$ fixes $supp(Y)$ pointwise and $supp(Y)$ supports $Y$, we obtain that $\pi \cdot y \in Y$, and so we should have either $y<\pi \cdot y$, or $y=\pi \cdot y$, or $\pi \cdot y<y$. If $y<\pi \cdot y$, then, because $\pi$ fixes $supp(\leq)$ pointwise and because the mapping $z\mapsto \pi \cdot z$ is bijective from $Y$ to $\pi \star Y$, we get $y<\pi \cdot y<\pi^{2} \cdot y<\ldots < \pi^{n} \cdot y$ for all $n \in \mathbb{N}$. However, since any permutation of atoms interchanges only finitely many atoms, it has a finite order in the group $S_{A}$, and so there is $m \in \mathbb{N}$ such that $\pi^{m}=Id$. This means $\pi^{m} \cdot y=y$, and so we get $y<y$ which is a contradiction. Analogously, the assumption $\pi \cdot y<y$, leads to the relation $\pi^{n} \cdot y<\ldots<\pi \cdot y<y$ for all $n \in \mathbb{N}$ which is also a contradiction since $\pi$ has finite order. Therefore, $\pi \cdot y=y$, and because $y$ was arbitrary chosen form $Y$, $Y$ should be a uniformly supported infinite subset of $X$. 
\end{proof}

Looking to the proof of Proposition \ref{propamo}, the following result 
follows directly.

\begin{corollary}
Let $X$ be a finitely supported subset of an invariant set $(Z, \cdot)$. 
If $X$ is FSM Mostowski infinite, then~$X$ is non-uniformly amorphous meaning that $X$ has two disjoint, infinite, uniformly supported subsets. 
\end{corollary}

\begin{remark} \label{remarema} 
In a permutation model of set theory with atoms, a set can be well-ordered if and only if there is a one-to-one mapping of the related set into the kernel of the model. Also it is noted that axiom of choice is valid in the kernel of the model \cite{jech}. Although FSM/nominal is somehow related to (has connections with) permutation models of set theory with atoms, it is independently developed over ZF without being necessary to relax the axioms of extensionality or foundation. FSM sets are ZF sets together with group actions, and such a theory makes sense over ZF without being necessary to require the validity of the axiom of choice on ZF sets. Thus, FSM is the entire ZF together with atomic sets with finite support (where the set of atoms is a fixed ZF formed by element whose internal structure is ignored and which are basic in the higher order construction). There may exist infinite ZF sets that do not contain infinite countable subsets, and as well there may exist infinite uniformly supported FSM sets (particularly such ZF sets) that do not contain infinite countable, uniformly supported, subsets.
\end{remark}

\begin{corollary} \label{pTII} 
\begin{enumerate}
\item The sets $A$, $A+A$ and $A \times A$ are FSM usually infinite, but there are not FSM Mostowski infinite, nor FSM Tarski II infinite.
\item None of the sets $\wp_{fin}(A)$, $\wp_{cofin}(A)$,  $\wp_{fs}(A)$ and $\wp_{fin}(\wp_{fs}(A))$ is Mostowski infinite in FSM.
\item None of the sets $A^{A}_{fs}$, $T_{fin}(A)^{A}_{fs}$ and $\wp_{fs}(A)^{A}_{fs}$ is FSM Mostowksi infinite.
\end{enumerate}
\end{corollary}
\begin{proof}
In the view of Theorem \ref{Tti} it is sufficient to prove that none of 
the sets $A$, $\wp_{fin}(A)$, $\wp_{cofin}(A)$, $\wp_{fs}(A)$, $A+A$, $A 
\times A$, $A^{A}_{fs}$,$T_{fin}(A)^{A}_{fs}$ and $\wp_{fs}(A)^{A}_{fs}$ 
contain infinite uniformly supported subsets. For $A$, $\wp_{fin}(A)$, 
$\wp_{cofin}(A)$ and $\wp_{fs}(A)$ this is obvious since for any finite 
set $S$ of atoms there are at most finitely many subsets of $A$ supported 
by~$S$, namely the subsets of $S$ and the supersets of $A \setminus S$. 
Moreover, $\wp_{fin}(\wp_{fs}(A))$ does not contain an infinite uniformly 
supported subset according to Lemma \ref{lem4} since $\wp_{fs}(A)$ does 
not contain an infinite uniformly supported subset.

Regarding $A^{A}_{fs}$, the things are also similar with Corollary \ref{ti2}(4). According to Lemma \ref{lem'''}, any $S$-supported function $f:A \to A$ should have the property that either $f|_{A \setminus S}=Id$ or $f|_{A \setminus S}$ is an one-element subset of $S$.  For each possible definition of such an $f$ on $S$ we have at most $|S|+1$ possible ways to define $f$ on $A \setminus S$, and so at most $|S|+1$ possible ways to completely define $f$ on $A$. If there was an infinite uniformly $S$-supported sequence of finitely supported functions from $A$ to $A$, there should exist infinitely many finitely supported functions from $S$ to $A$ supported by the same finite set $S$. But this contradicts the fact that $A^{|S|}$ does not contain an infinite uniformly supported subset (this follows by applying finitely many times the result that $X \times X$ does not contain an infinite uniformly supported subset whenever $X$ does not contain an infinite uniformly supported subset). Analyzing the proofs of   Corollary \ref{ti2}(6) and (7), we also conclude that $T_{fin}(A)^{A}_{fs}$ and $\wp_{fs}(A)^{A}_{fs}$ do not contain infinite uniformly supported subsets.

We also have that $A$ is not FSM Tarski II infinite because $\wp_{fs}(A)$ contains no infinite uniformly supported subsets, and so every totally ordered subset (particularly via inclusion) of $\wp_{fs}(A)$ should be finite meaning that it should have a maximal element. Furthermore, we have that there is an equivariant bijection between $\wp_{fs}(A+A)$ and $\wp_{fs}(A)\times \wp_{fs}(A)$. Since $\wp_{fs}(A)$ does not contain an infinite uniformly supported subset, we have that $\wp_{fs}(A) \times \wp_{fs}(A)$ does not contain an infinite uniformly supported subset (the proof is quasi-identical to the one of Theorem \ref{ti1}(6) without taking count on the countability of the related infinite uniformly supported family). Therefore, any infinite totally ordered (via inclusion) uniformly supported family of  $\wp_{fs}(A+A)$ should be finite containing a maximal element. There is an equivariant bijection between  $\wp_{fs}(A)^{A}_{fs}$ and $\wp_{fs}(A \times A)$. Therefore any uniformly supported totally ordered subset of $\wp_{fs}(A \times A)$ should be finite containing a maximal element.
\end{proof}

\begin{corollary}
Let $X$ be a finitely supported subset of an invariant set $Y$ such that $X$ does not contain an infinite uniformly supported subset. Then the set $\wp_{fin}(X)$ is not FSM Mostowski infinite. 
\end{corollary}

\begin{proof}
According to Lemma \ref{lem4}, $\wp_{fin}(X)$ does not contain an infinite 
uniformly supported subset. Thus, by Theorem~\ref{Tti}, $\wp_{fin}(X)$ is 
not FSM Mostowski infinite. 
\end{proof}

\begin{theorem} \label{TTRR} Let $X$ be a finitely supported subset of an invariant set $(Y, \cdot)$. 
\begin{enumerate}
\item If $X$ is FSM Tarski I infinite, then $X$ is FSM Tarski III infinite. The converse does not hold. However if $X$ is FSM Tarski III infinite, then $\wp_{fs}(X)$ is FSM Tarski I  infinite. 
\item If $X$ is FSM Tarski III infinite, then $X$ is FSM Dedekind infinite. The converse does not hold. However if $X$ is FSM Dedekind infinite, then $\wp_{fs}(X)$ is FSM Tarski III  infinite. 
\end{enumerate}
\end{theorem}

\begin{proof}

1. We consider the case when $X$ has at least two elements (otherwise the 
theorem is trivial). Let $X$ be FSM Tarski I infinite. Then $|X \times X|= 
|X|$.  Fix two elements $x_{1}, x_{2} \in X$ with $x_{1}\neq x_{2}$. We 
can define an injection $f:X \times \{0,1\} \to X \times X$ by 
$f(u)=\left\{ \begin{array}{ll} (x,x_{1}) & \text{for}\: u=(x,0)\\ 
(x,x_{2}) & \text{for}\: u=(x,1) \end{array}\right.$. Clearly, by checking 
the condition in Proposition~\ref{2.18'} and using Proposition~\ref{p1}, 
we have that $f$ is supported by $supp(X) \cup supp(x_{1}) \cup 
supp(x_{2})$ (since $\{0,1\}$ is necessarily a trivial invariant set), and 
so $|X\times\{0,1\}|\leq |X \times X|$. Thus, $|X\times\{0,1\}| \leq |X|$. 
Obviously, there is an injection $i: X \to X\times\{0,1\}$ defined by 
$i(x)=(x,0)$ for all $x \in X$ which is supported by $supp(X)$. According 
to Lemma~\ref{lem2}, we get $2|X|=|X \times \{0,1\}|=|X|$.

Let us consider $X=\mathbb{N} \times A$. We make the remark that 
$|\mathbb{N}\times \mathbb{N}|=|\mathbb{N}|$ by considering the 
equivariant injection $h:\mathbb{N} \times \mathbb{N} \to \mathbb{N}$ 
defined by $h(m,n)=2^{m}3^{n}$ and using Lemma \ref{lem2}. Similarly, 
$|\{0,1\}\times \mathbb{N}|=|\mathbb{N}|$ by considering the equivariant 
injection $h':\mathbb{N} \times \{0,1\}\to \mathbb{N}$ defined by 
$h'(n,0)=2^{n}$ and $h'(n,1)=3^{n}$ and using Lemma \ref{lem2}. We have 
$2|X|=2|\mathbb{N}||A|=|\mathbb{N}||A|=|X|$. However, we prove that $|X 
\times X|\neq |X|$. Assume the contrary, and so we have $|\mathbb{N} 
\times (A \times A)|=|\mathbb{N} \times A \times \mathbb{N} \times 
A|=|\mathbb{N} \times A|$. Thus, there is a finitely supported injection 
$g: A \times A \to \mathbb{N} \times A$, and by Proposition \ref{pco2} 
there is a finitely supported surjection $f:\mathbb{N} \times A \to A 
\times A$. Let us consider three different atoms $a,b,c \notin supp(f)$.  
There exists $(i,x) \in \mathbb{N} \times A$ such that $f(i,x)=(a,b)$. 
Since $(a\,b) \in Fix(supp(f))$ and~$\mathbb{N}$ is trivial invariant set, 
we have $f(i,(a\,b)(x))=(a\,b)f(i,x)=(a\,b)(a,b)=((a\,b)(a),(a\,b)(b))=(b,a)$. 
We should have $x=a$ or $x=b$, otherwise $f$ is not a function. Assume 
without losing the generality that $x=a$, which means $f(i,a)=(a,b)$. 
Therefore $f(i,b)=f(i,(a\,b)(a))=(a\,b)f(i,a)=(a\,b)(a,b)=(b,a)$. 
Similarly, since $(a\,c),(b\,c) \in Fix(supp(f))$, we have 
$f(i,c)=f(i,(a\,c)(a))=(a\,c)f(i,a)=(a\,c)(a,b)=(c,b)$ and 
$f(i,b)=f(i,(b\,c)(c))=(b\,c)f(i,c)=(b\,c)(c,b)=(b,c)$. But $f(i,b)=(b,a)$ 
contradicting the functionality of~$f$. Therefore, $X$ is FSM Tarski III 
infinite, but it is not FSM Tarski I infinite.

Now, suppose that $X$ is FSM Tarski III infinite, which means 
$|\{0,1\}\times X|=|X|$. We define the mapping $\psi:\wp_{fs}(X) \times 
\wp_{fs}(X) \to \wp_{fs}(\{0,1\}\times X)$ by $f(U,V)=\{(0,x)\,|\,x \in 
U\} \cup \{(1,y)\,|\,y \in V\}$ for all $U,V \in \wp_{fs}(X)$. Clearly 
$\psi$ is well defined and bijective, and for each $\pi \in Fix(supp(X))$ 
we have $\psi(\pi \star U,\pi \star V)=\pi \star \psi(U,V)$ which means 
$\psi$ is finitely supported. Therefore, $|\wp_{fs}(X) \times 
\wp_{fs}(X)|=|\wp_{fs}(\{0,1\}\times X)|=|\wp_{fs}(X)|$. The last equality 
follows by applying twice Lemma \ref{lemlem} (using the fact that there is 
a finitely supported surjection from $X$ onto $X\times\{0,1\}$ and a 
finitely supported surjection from $X\times\{0,1\}$ onto $X$, we obtain 
there is a finitely supported injection from $\wp_{fs}(X\times\{0,1\})$ 
into $\wp_{fs}(X)$, and a finitely supported injection from $\wp_{fs}(X)$ 
into $\wp_{fs}(X\times\{0,1\})$) and Lemma~\ref{lem3}.

2. Let us assume that $X$ is FSM Tarski III infinite.
Let us consider an element $y_{1}$ belonging to an invariant set (whose action is also denoted by $\cdot$) with $y_{1}\notin X$ (such an element can be, for example, a non-empty element in $\wp_{fs}(X) \setminus X$). Fix  $y_{2} \in X$.  One can define a mapping $f:X \cup \{y_{1}\} \to X \times \{0,1\}$ by $f(x)=\left\{ \begin{array}{ll}
(x,0) & \text{for}\: x \in X\\
(y_{2}, 1) & \text{for}\: x=y_{1} \end{array}\right.$.  Clearly $f$ is injective and it is supported by $S=supp(X) \cup supp(y_{1}) \cup supp(y_{2})$ because for all $\pi$ fixing $S$ pointwise we have $f(\pi \cdot x)=\pi \cdot f(x)$ for all $x \in X \cup \{y_{1}\}$. Therefore, $|X \cup \{y_{1}\}| \leq |X \times \{0,1\}|=|X|$, and so there is a finitely supported injection $g:X \cup \{y_{1}\} \to X$. The mapping $h:X \to X$ defined by  $h(x)=g(x)$ is injective, supported by $supp(g) \cup supp(X)$, and $g(y_{1}) \in X \setminus h(X)$, which means $h$ is not surjective. It follows that $X$ is FSM Dedekind infinite.

Let us consider $X=A \cup \mathbb{N}$. Since $A$ and $\mathbb{N}$ are 
disjoint, we have that $X$ is an invariant set (similarly as in 
Proposition~\ref{p1}). Clearly $X$ is FSM Dedekind infinite. Assume, by 
contradiction, that $|X|=2|X|$, that is $|A \cup 
\mathbb{N}|=|A+A+\mathbb{N}|=|(\{0,1\}\times A) \cup \mathbb{N}|$. Thus, 
there is a finitely supported injection $f:(\{0,1\}\times A) \cup 
\mathbb{N}\to A \cup \mathbb{N}$, and so there exists a finitely supported 
injection $f:(\{0,1\}\times A) \to A \cup \mathbb{N}$. We prove that 
whenever $\varphi:A \to A \cup \mathbb{N}$ is finitely supported and 
injective, for $a \notin supp(\varphi)$ we have $\varphi(a) \in A$. 
Assume, by contradiction, that there is $a \notin supp(\varphi)$ such that 
$\varphi(a)\in \mathbb{N}$. Since $supp(\varphi)$ is finite, there exists 
$b \notin supp(\varphi)$, $b \neq a$. Thus, $(a\,b) \in 
Fix(supp(\varphi))$, and so $\varphi(b)=\varphi((a\,b)(a))=(a\,b)\diamond 
\varphi(a)=\varphi(a)$ since $(\mathbb{N}, \diamond)$ is a trivial 
invariant set. This contradicts the injectivity of $\varphi$.  We can 
consider the mappings $\varphi_{1},\varphi_{2}: A \to A \cup \mathbb{N}$ 
defined by $\varphi_{1}(a)=f(0,a)$ for all $a \in A$ and 
$\varphi_{2}(a)=f(1,a)$ for all $a \in A$, that are injective and 
supported by $supp(f)$. Therefore, $f(\{0\} \times A)=\varphi_{1}(A)$ 
contains at most finitely many element from $\mathbb{N}$, and $f(\{1\} 
\times A)=\varphi_{2}(A)$ also contains at most finitely many element from 
$\mathbb{N}$. Thus, $f$ is an injection from $(\{0,1\}\times A)$ to $A 
\cup Z$ where $Z$ is a finite subset of $\mathbb{N}$. It follows that 
$f(\{0\} \times A)$ contains an infinite subset of atoms $U$, and $f(\{1\} 
\times A)$ contains an infinite subset of atoms $V$. Since $f$ is 
injective, it follows that $U$ and $V$ are infinite disjoint subsets of 
$A$, which contradicts Proposition \ref{p111} stating that $A$ is 
amorphous.

Now, if $X$ is FSM Dedekind infinite, we have that there is a finitely 
supported injection $h$ from $X$ onto a finitely supported proper subset 
$Z$ of $X$. Consider an element $y_{1}$ belonging to an invariant set with 
$y_{1}\notin X$. We can define an injection $h': X \cup \{y_{1}\} \to X$ 
by taking $h'(x)=h(x)$ for all $x \in X$ and $h'(y_{1})=b$ with $b \in X 
\setminus Z$. Clearly $h'$ is supported by $supp(h) \cup supp(y_{1}) \cup 
supp(b)$.  Since there also exists an $supp(X)$-supported injection from 
$X$ to $X \cup \{y_{1}\}$, according to Lemma \ref{lem3}, one can define a 
finitely supported bijection $\psi$ from $X$ to $X \cup \{y_{1}\}$. 
According to Lemma \ref{lemlem} the mapping $g:\wp_{fs}(X \cup \{y_{1}\}) 
\to \wp_{fs}(X)$ defined by $g(V)=f^{-1}(V)$ for all $V \in \wp_{fs}(X 
\cup \{y_{1}\})$ is finitely supported and injective. Therefore, 
$2^{|X|}\geq 2^{|X|+1} =2\cdot 2^{|X|}$ which in the view of Lemma~\ref{lem3} 
leads to the conclusion that $\wp_{fs}(X)$ is FSM Tarski III infinite. 
\end{proof}

\begin{corollary} \label{ti4} 
The following sets are FSM usual infinite, but they are not FSM Tarski I infinite, 
nor FSM Tarski III infinite.
\begin{enumerate}
\item The invariant set $A$.
\item The invariant set $\wp_{fs}(A)$.
\item The invariant sets $\wp_{fin}(A)$ and $\wp_{cofin}(A)$.
\item The set $\wp_{fin}(X)$ where $X$ is a finitely supported subset of an invariant set containing no infinite uniformly supported subset.
\end{enumerate} 
\end{corollary}

\begin{proof}
The result follows directly because the related sets are not FSM Dedekind 
infinite, according to Theorem~\ref{ti1} and Corollary~\ref{ti2}.
\end{proof}

\begin{corollary}Let $X$ be an infinite finitely supported subset of an invariant set. Then $\wp_{fs}(\wp_{fs}(\wp_{fs}(X)))$ is FSM Tarski III infinite and, consequently, $\wp_{fs}(\wp_{fs}(\wp_{fs}(\wp_{fs}(X))))$ is FSM Tarski I infinite. 
\end{corollary}

\begin{proof}Since $\wp_{fs}(\wp_{fs}(X))$ is FSM Dedekind infinite, as in the proof of Theorem \ref{TTRR}(2) one can prove $|\wp_{fs}(\wp_{fs}(X))|+1=|\wp_{fs}(\wp_{fs}(X))|$. The result now follows directly using arithmetic properties of FSM cardinalities proved above. 
\end{proof}
 
In a future work we intend to prove an even stronger result claiming that $\wp_{fs}(\wp_{fs}(X))$ is FSM Tarski III infinite and, consequently, $\wp_{fs}(\wp_{fs}(\wp_{fs}(X)))$ is FSM Tarski I infinite, whenever $X$ is an infinite finitely supported subset of an invariant set.

\begin{corollary} \label{CAN1} 
The sets $A^{\mathbb{N}}_{fs}$ and $\mathbb{N}^{A}_{fs}$ are FSM Tarski I infinite, and so they are also Tarski III infinite. 
\end{corollary}
\begin{proof}There is an equivariant bijection $\psi$ between $(A^{\mathbb{N}})^{2}_{fs}$ and $A^{\mathbb{N} \times \{0,1\}}_{fs}$ that associates to each Cartesian pair $(f,g)$ of mappings from $\mathbb{N}$ to $A$ a mapping $h:\mathbb{N} \times\{0,1\} \to A$ defined as follows.  \[ h(u)=\left\{ \begin{array}{ll}

f(n) & \text{if}\:   u=(n,0) \\
\\
g(n) & \text{if}\:  u=(n,1) \end{array}\right., \] 
The equivariance of $\psi$ follows from Proposition \ref{2.18'} because if $\pi \in S_{A}$ we have $\psi(\pi \widetilde {\star} f, \pi \widetilde {\star} g)=h'$ where $h'(n,0)=(\pi \widetilde {\star} f)(n)=\pi(f(n))$ and $h'(n,1)=(\pi \widetilde {\star} g)(n)=\pi(g(n))$. Thus, $h'(u)=\pi(h(u))$ for all $u \in \mathbb{N} \times \{0,1\}$ which means $h'=\pi \widetilde {\star} h=\pi \widetilde {\star} \psi (f,g)$. 

There also exists an equivariant bijection $\varphi$  between $(\mathbb{N}^{A})^{2}_{fs}$ and $(\mathbb{N} \times \mathbb{N})^{A}_{fs}$ that associates to each Cartesian pair $(f,g)$ of mappings from $A$ to $\mathbb{N}$ a mapping $h: A \to \mathbb{N} \times \mathbb{N}$ defined by $h(a)=(f(a),g(a))$ for all $a \in A$.  The equivariance of $\varphi$ follows from Proposition \ref{2.18'} because if $\pi \in S_{A}$ we have $\varphi(\pi \widetilde {\star} f, \pi \widetilde {\star} g)=h'$ where $h'(a)=((\pi \widetilde {\star} f)(a), (\pi \widetilde {\star} g)(a))=(f(\pi^{-1}(a)), g(\pi^{-1}(a)))=h(\pi^{-1}(a))=(\pi \widetilde {\star}h)(a)$ for all $a \in A$, and so $h'=\pi \widetilde {\star} h=\pi \widetilde {\star} \varphi(f,g)$.
Therefore  $|(A^{\mathbb{N}})^{2}_{fs}|=|A^{\mathbb{N} \times \{0,1\}}_{fs}|= |A^{\mathbb{N}}_{fs}|$. Therefore, $|(\mathbb{N}^{A})^{2}_{fs}|=|(\mathbb{N} \times \mathbb{N})^{A}_{fs}|=|\mathbb{N}^{A}_{fs}|$ according to Proposition \ref{pco1}(3) and Lemma \ref{lem2} (we used $|\mathbb{N} \times \mathbb{N}|=|\mathbb{N}|$). 
\end{proof}

\begin{theorem} Let $X$ be a finitely supported subset of an invariant set $(Y, \cdot)$. 
 If $\wp_{fs}(X)$ is FSM Tarski I infinite, then $\wp_{fs}(X)$ is FSM Tarski III infinite. The converse does not hold.
\end{theorem}
\begin{proof}
The direct implication is a consequence of Theorem \ref{TTRR}(1). Thus, we focus on the proof of the invalidity of the reverse implication.

Firstly we make the remark that whenever $U,V$ are finitely supported subsets of an invariant set with $U \cap V=\emptyset$, we have that there is a finitely supported (by $supp(U) \cup supp(V)$) bijection from $\wp_{fs}(U\cup V)$ into $\wp_{fs}(U) \times \wp_{fs}(V)$ that maps each $X \in  \wp_{fs}(U\cup V)$ into the pair $(X \cap U, X \cap V)$. Analogously, whenever $B,C$ are invariant sets there is  an equivariant bijection from $\wp_{fs}(B) \times \wp_{fs}(C)$ into $\wp_{fs}(B+C)$ that maps each pair $(B_{1},C_{1}) \in \wp_{fs}(B) \times \wp_{fs}(C)$ into the set $\{(0,b)\,|\,b \in B_{1}\} \cup \{(1,c)\,|\,c \in C_{1}\}$. This follows directly by verifying the conditions in Proposition \ref{2.18'}.

Let us consider the set $A \cup \mathbb{N}$ which is FSM Dedekind infinite. According to Theorem \ref{TTRR}(2), we have that $\wp_{fs}(A \cup \mathbb{N})$ is FSM Tarski III infinite. We prove that it is not FSM Tarski I infinite. Assume, by contradiction that $|\wp_{fs}(A \cup \mathbb{N}) \times \wp_{fs}(A \cup \mathbb{N})|=|\wp_{fs}(A \cup \mathbb{N})|$ which means $|\wp_{fs}(A + \mathbb{N}+A + \mathbb{N})|=|\wp_{fs}(A \cup \mathbb{N})|$, and so  $|\wp_{fs}(A +A + \mathbb{N})|=|\wp_{fs}(A \cup \mathbb{N})|$. Thus, according to Proposition \ref{pco1}(4), there is a finitely supported injection from  $\wp_{fs}(A + A)$ to $\wp_{fs}(A \cup \mathbb{N})$, which means there is a finitely supported injection from $\wp_{fs}(A) \times \wp_{fs}(A)$ to $ \wp_{fs}(A) \times \wp_{fs}(\mathbb{N})$, and so there is a finitely supported injection from  $A \times A$ to $ \wp_{fs}(A) \times \wp_{fs}(\mathbb{N})$. According to Proposition \ref{pco2}, there should exist a finitely supported surjection   $f: \wp_{fs}(A)  \times \wp_{fs}(\mathbb{N}) \to A \times A$.  Let us consider two atoms $a,b\notin supp(f)$ with $a 
\neq b$.  It follows that $(a\, b) \in Fix(supp(f))$. Since $f$ is surjective, there exists $(X,M) \in \wp_{fs}(A) \times \wp_{fs}(\mathbb{N})$ 
such that $f(X,M)=(a,b)$. According to Proposition \ref{2.18'} and because $\mathbb{N}$ is a trivial invariant set meaning that $(a\,b) \star M = M$, we have $f((a\,b) \star X,M)=f((a\,b) \otimes (X,M))=(a\,b) 
\otimes f(X,M)=(a\,b) \otimes (a,b)=((a\,b)(a), (a\,b)(b))=(b,a)$. Due to the functionality of $f$ we should have $((a\,b) \star X,M) \neq (X,M)$, which means $(a\,b) \star X \neq X$. 

We prove that if both $a,b \in supp(X)$, then $(a\,b)\star X=X$. Indeed, suppose $a,b \in supp(X)$. Since $X \in \wp_{fs}(A)$, from Proposition \ref{p111} we have that  $X$ is either finite or cofinite. If $X$ is finite, then $supp(X)=X$, and so $a,b \in X$. Therefore, $(a\,b) \star X=\{(a\,b)(x)\,|\,x \in X\}=\{(a\,b)(a)\} \cup \{(a\,b)(b)\} \cup \{(a\,b)(c)\,|\,c \in X \setminus\{a,b\}\}=\{b\} \cup \{a\} \cup (X \setminus \{a,b\})=X$. Now, if $X$ is cofinite, then $supp(X)=A \setminus X$, and so $a,b \in A \setminus X$. Since $a,b \notin X$, we have $a,b \neq   x$ for all $x \in X$, which means $(a\,b)(x)=x$ for all $x \in X$, and again $(a\,b)\star X=X$.

Thus, one of $a$ or $b$ does not belong to $supp(X)$. Assume $b \notin 
supp(X)$. Let us consider $c\neq a,b$, $c \notin supp(f)$, $c \notin 
supp(X)$. Then $(b\, c) \in Fix(supp(X))$, and so $(b\,c)\star X=X$. 
Moreover, $(b\, c) \in Fix(supp(f))$, and by Proposition~\ref{2.18'} we 
have $(a,b)=f(X,M)=f((b\,c) \star X,M)=f((b\,c)\otimes (X,M))=(b\,c) 
\otimes f(X,M)=(b\,c) \otimes (a,b)=(a,c)$ which is a contradiction 
because $b\neq c$. 
\end{proof}

\begin{proposition}
Let $X$ be a finitely supported subset of an invariant set $(Y, \cdot)$. If $X$ is FSM Tarski III infinite, then there exists a finitely supported bijection $g:\mathbb{N} \times X \to X$. The reverse implication is also valid. 
\end{proposition}

\begin{proof}By hypothesis, there is a finitely supported bijection $\varphi:\{0,1\} \times X \to X$. Let us consider the mappings $f_{1},f_{2}: X \to X$ defined by $f_{1}(x)=\varphi(0,x)$ for all $x \in X$ and $f_{2}(x)=\varphi(1,x)$ for all $x \in X$, that are injective and supported by $supp(\varphi)$ according to Proposition \ref{2.18'}. Since $\varphi$ is injective we also have $Im(f_{1}) \cap Im(f_{2})=\emptyset$, and because $\varphi$ is surjective we get $Im(f_{1}) \cup Im(f_{2})=X$. We prove by induction that the $n$-times auto-composition of $f_{2}$, denoted by $f_{2}^{n}$, is supported by $supp(f_{2})$ for all $n\in \mathbb{N}$. For $n=1$ this is obvious. So assume that  $f_{2}^{n-1}$ is supported by $supp(f_{2})$. By Proposition \ref{2.18'} we must have $f_{2}^{n-1}(\sigma \cdot x)=\sigma \cdot f_{2}^{n-1}(x)$ for all $\sigma \in Fix(supp(f_{2}))$ and $x \in X$. Let us fix $\pi \in Fix(supp(f_{2}))$. According to Proposition \ref{2.18'}, we have $f_{2}^{n}(\pi \cdot x)=f_{2}(f_{2}^{n-1}(\pi \cdot x))=f_{2}(\pi\cdot f_{2}^{n-1}(x))=\pi \cdot f_{2}(f_{2}^{n-1}(x))=\pi \cdot f_{2}^{n}(x)$ for all $x \in X$, and so $f_{2}^{n}$ is finitely supported from Proposition \ref{2.18'}. Define $f:\mathbb{N} \times X \to X$ by $f((n,x))=f_{2}^{n}(f_{1}(x))$. Let $\pi \in Fix(supp(f_{1}) \cup supp(f_{2}))$. According to Proposition \ref{2.18'} and because $(\mathbb{N}, \diamond)$ is a trivial invariant set we get $f(\pi \otimes (n,x))=f((n, \pi \cdot x))=f_{2}^{n}(f_{1}(\pi \cdot x))=f_{2}^{n}(\pi \cdot f_{1}(x))=\pi \cdot f_{2}^{n}(f_{1}(x))=\pi \cdot f((n,x))$ for all $(n,x) \in \mathbb{N} \times X$,  which means $f$ is supported by $supp(f_{1}) \cup supp(f_{2})$. We prove the injectivity of $f$. Assume $f((n,x))=f((m,y))$ which means $f_{2}^{n}(f_{1}(x))=f_{2}^{m}(f_{1}(y))$. If $n>m$ this leads to $f_{2}^{n-m}(f_{1}(x))=f_{1}(y)$ (since $f_{2}$ is injective) which is in contradiction with the relation $Im(f_{1}) \cap Im(f_{2})=\emptyset$. Analogously we cannot have $n<m$. Thus, $n=m$ which leads to $f_{1}(x)=f_{1}(y)$, and so $x=y$ due to the injectivity of $f_{1}$. Therefore, $f$ is injective. Since we obviously have a finitely supported injection from $X$ into $\mathbb{N} \times X$ (e.g $x \mapsto (0,x)$ which is supported by $supp(X)$), in the view of Lemma \ref{lem2} we can find a finitely supported bijection between $X$ and $\mathbb{N} \times X$. 

The reverse implication is almost trivial. There is a finitely supported injection from $\{0,1\} \times X$ into $\mathbb{N} \times X$. If there is a finitely supported injection from  $\mathbb{N} \times X$ into $X$, then there is a finitely supported injection from $\{0,1\} \times X$ into $X$. The desired result follows from Lemma \ref{lem2}. 
\end{proof}

\section{Countability} \label{countable}

\begin{definition} Let $Y$ be a finitely supported subset of an invariant set $X$. Then $Y$ is \emph{countable in FSM (or FSM countable)} if there exists a finitely supported onto mapping $f: \mathbb{N} \to Y$. 
\end{definition}

\begin{proposition} Let $Y$ be a finitely supported countable subset of an invariant set $(X, \cdot)$. Then $Y$ is uniformly supported.
\end{proposition}

\begin{proof}There exists a finitely supported onto mapping $f: \mathbb{N} \to Y$. Thus, for each arbitrary $y\in Y$, there exists $n \in \mathbb{N}$ such that $f(n)=y$. According to Proposition \ref{2.18'}, for each $\pi \in Fix(supp(f))$ we have $\pi \cdot y=\pi \cdot f(n)=f(\pi \diamond n)=f(n)=y$, where $\diamond$ is the necessarily trivial action on $\mathbb{N}$. Thus, $Y$ is uniformly supported by $supp(f)$.
\end{proof}

\begin{proposition} Let $Y$ be a finitely supported subset of an invariant set $X$. Then $Y$ is countable in FSM if and only if there exists a finitely supported one-to-one mapping $g: Y \to  \mathbb{N}$.
\end{proposition}
\begin{proof} Suppose that $Y$ is countable in FSM. Then there exists a finitely supported onto mapping $f: \mathbb{N} \to Y$. We define $g: Y \to  \mathbb{N}$ by $g(y)=min[f^{-1}(\{y\})]$, for all $y \in Y$. According to Proposition \ref{2.18'}, $g$ is supported by $supp(f) \cup supp(Y)$. Obviously, $g$ is one-to-one. Conversely, if there exists a finitely supported one-to-one mapping $g: Y \to  \mathbb{N}$, then $g(Y)$ is supported is equivariant as a subset of the trivial invariant set $\mathbb{N}$. Thus,  there exists a finitely supported bijection $g: Y \to  g(Y)$, where $g(Y) \subseteq \mathbb{N}$. We define $f: \mathbb{N} \to Y$ by \[ f(n)=\left\{ \begin{array}{ll}

g^{-1}(n) & \text{if}\:   n \in g(Y) \\
\\
t & \text{if}\:  n \in \mathbb{N} \setminus g(Y) \end{array}\right., \] where $t$ is a fixed element of $Y$. According to Proposition \ref{2.18'}, we have that $f$ is supported by $supp(g) \cup supp(Y) \cup supp(t)$. Moreover, $f$ is onto. 
\end{proof}

\begin{proposition} \label{cou} Let $Y$ be an infinite, finitely supported, countable subset of an invariant set $X$. Then  there exists a finitely supported bijective mapping $g: Y \to  \mathbb{N}$.
\end{proposition}

\begin{proof}
First we prove that for any infinite subset $B$ of $\mathbb{N}$, there is an injection from $\mathbb{N}$ into $B$. Fix such a $B$. It follows that $B$ is well ordered. Define $f:\mathbb{N} \to B$ by: $f(1)=min(B)$, $f(2)=min(B \setminus f(1))$, and recursively
$f(m)= min(B \setminus \{f(1),f(2),...,f(m-1)\})$ for all $m \in \mathbb{N}$ (since $B$ is infinite). Since $\mathbb{N}$ is well ordered, choice is not involved. Obviously since both $B$ and $\mathbb{N}$ are trivial invariant sets, we have that $f$ is equivariant. Since $B$ is a subset of $\mathbb{N}$ we also have an equivariant injective mapping $h:B \to \mathbb{N}$. According to Lemma \ref{lem1}, there is an equivariant bijection between $B$ and $\mathbb{N}$ (we can even prove that $f$ is bijective). 

Since $Y$ is countable, there exists a finitely supported one-to-one 
mapping $u: Y \to \mathbb{N}$. Thus, the mapping $u: Y \to u(Y)$ is 
finitely supported and bijective. Since $u(Y) \subseteq N$, we have that 
there is an equivariant bijection $v$ between $u(Y)$ and~$\mathbb{N}$, and 
so there exists a finitely supported bijective mapping $g: Y \to 
\mathbb{N}$ defined by $g=v \circ u$. 
\end{proof}

From \cite{book} we know that the (in)consistency of the choice principle 
$\textbf{CC(fin)}$ in FSM is an open problem, meaning that we do not know 
whether this principle is consistent or not in respect of the FSM axioms. 
A relationship between countable union principles and countable choice 
principles is presented in ZF in \cite{herrlich}. Below we prove that such 
a relationship is preserved in FSM.

\begin{definition}
\begin{enumerate}
\item The Countable Choice Principle for finite sets in FSM \textbf{CC(fin)} has the form ``Given any invariant set $X$, and 
any countable family $\mathcal{F}=(X_{n})_{n}$ of finite subsets of $X$ such 
that the mapping $n\mapsto X_{n}$ is finitely supported, there exists a 
finitely supported choice function on $\mathcal{F}$."
\item The Countable Union Theorem for finite sets in FSM, \textbf{CUT(fin)}, has the form ``Given any invariant set $X$ and any  countable family $\mathcal{F}=(X_{n})_{n}$ of finite subsets of $X$  such 
that the mapping $n\mapsto X_{n}$ is finitely supported, then there exists a finitely supported onto mapping  $f: \mathbb{N} \to \underset{n}{\cup}X_{n}$"
\item The Countable Union Theorem for $k$-element sets in FSM, \textbf{CUT(k)}, has the form ``Given any invariant set $X$ and any  countable family $\mathcal{F}=(X_{n})_{n}$ of $k$-element subsets of $X$ such that the mapping $n\mapsto X_{n}$ is finitely supported, then there exists a finitely supported onto mapping  $f: \mathbb{N} \to \underset{n}{\cup}X_{n}$"
\item The Countable Choice Principle for sets of $k$-element sets in FSM, 
\textbf{CC(k)} has the form ``Given any invariant set~$X$ and any 
countable family $\mathcal{F}=(X_{n})_{n}$ of $k$-element subsets of $X$ 
in FSM such that the mapping $n\mapsto X_{n}$ is finitely supported, there 
exists a finitely supported choice function on $\mathcal{F}$."
\end{enumerate}
\end{definition} 

\begin{proposition} 
In FSM, the following equivalences hold. 
\begin{enumerate}
\item \textbf{CUT(fin)} $\Leftrightarrow$ \textbf{CC(fin)};
\item \textbf{CUT(2)} $\Leftrightarrow$ \textbf{CC(2)};
\item \textbf{CUT(n)} $\Leftrightarrow$ \textbf{CC(i)} for all $i \leq n$.
\end{enumerate}
\end{proposition}

\begin{proof}
1. Let us assume that  \textbf{CUT(fin)} is valid in FSM.   We consider the finitely supported countable family $\mathcal{F}=(X_{n})_{n}$  in FSM, where each $X_{n}$ is a non-empty finite  subset of an invariant set $X$ in FSM. 

From \textbf{CUT(fin)}, there exists a finitely supported onto mapping  $f: \mathbb{N} \to \underset{n}{\cup}X_{n}$. Since $f$ is onto and each $X_{n}$ is non-empty, we have that $f^{-1}(X_{n})$ is a non-empty subset of $\mathbb{N}$ for each $n \in \mathbb{N}$. Consider the function $g: \mathcal{F} \to \cup \mathcal{F}$, defined by $g(X_{n})=f(min[f^{-1}(X_{n})])$. We claim that $supp(f) \cup supp(n \mapsto X_{n})$ supports $g$. Let $\pi \in Fix(supp(f) \cup supp(n \mapsto X_{n}))$. According to Proposition \ref{2.18'}, and because $\mathbb{N}$ is a trivial invariant  set and each element $X_{n}$ is supported by $supp(n \mapsto X_{n})$, we have $\pi \cdot g(X_{n})=\pi \cdot f(min[f^{-1}(X_{n})])= f(\pi \diamond min[f^{-1}(X_{n})])=f(min[f^{-1}(X_{n})])=g(X_{n})=g(\pi \star X_{n})$, where by $\star$ we denoted the $S_{A}$-action on $\mathcal{F}$, by $\cdot$ we denoted the $S_{A}$-action on $\cup \mathcal{F}$ and by $\diamond$ we denoted the trivial action on $\mathbb{N}$. Therefore, $g$ is finitely supported. Moreover, $g(X_{n}) \in X_{n}$, and so $g$ is a choice function on $\mathcal{F}$. 

Conversely, let 
$\mathcal{F}=(X_{n})_{n}$ be a countable family of finite subsets of $X$ 
such that the mapping $n\mapsto X_{n}$ is finitely supported. Thus, each $X_{n}$ is supported by the same set $S=supp(n \mapsto X_{n})$. Since each $X_{n}$ is finite (and the support of a finite set coincides with the union of the supports of its elements), as in the proof of Lemma  \ref{lem4}, we have that $Y=\underset{n \in \mathbb{N}}\cup X_{n}$
is uniformly supported by $S$. Moreover, the countable sequence $(Y_{n})_{n \in \mathbb{N}}$ defined by $Y_{n}=X_{n} \setminus \underset{m<n}\cup X_{m}$ is a uniformly supported (by $S$) sequence of pairwise disjoint uniformly supported sets with $Y=\underset{n \in \mathbb{N}}\cup Y_{n}$. Consider the infinite family $M \subseteq \mathbb{N}$ such that all the terms of $(Y_{n})_{n \in M}$ are non-empty. 

For each $n \in M$, the set $T_{n}$ of total orders on $Y_{n}$ is finite, non-empty, and uniformly supported by $S$. Thus, by applying \textbf{CC(fin)} to $(T_{n})_{n \in M}$, there is a choice function $f$ on $(T_{n})_{n \in M}$ which is also supported by $S$. Furthermore, $f(T_{n})$ is supported by $supp(f) \cup supp(T_{n})=S$ for all $n \in M$. One can define a uniformly supported (by $S$) total order relation on $Y$ (which is also a well order relation on $Y$) as follows

$x \leq y$ if and only if $\left\{ \begin{array}{ll}

x \in Y_{n} \; \text{and}\; y \in Y_{m} \;\text{with}\; n<m  \\

\text{or}\\
x,y \in Y_{n}\; \text{and}\; xf(T_{n})y \end{array}\right.$.

Clearly, if $Y$ is infinite, then there is an $S$-supported order isomorphism between $(Y, \leq)$ and $M$ with the natural order, which means, in the view of Proposition \ref{cou}, that $Y$ is countable.

2. As in the above item $\textbf{CUT(2)} \Rightarrow \textbf{CC(2)}$.

For proving $\textbf{CC(2)} \Rightarrow \textbf{CUT(2)}$, let $\mathcal{F}=(X_{n})_{n}$ be a countable family of 2-element subsets of $X$ such that the mapping $n\mapsto X_{n}$ is finitely supported. According to \textbf{CC(2)} we have that there exists a finitely supported choice function $g$ on $(X_{n})_{n}$. Let $x_{n}=g(X_{n}) \in X_{n}$.  As in the above item, we have that  $supp(n \mapsto X_{n})$ supports $x_{n}$ for all $n \in \mathbb{N}$.  

For each $n$, let $y_{n}$ be the unique element of $X_{n}\setminus \{x_{n}\}$.  Since for any $n$ both $x_{n}$ and $X_{n}$ are supported by the same set $supp(n \mapsto X_{n})$, it follows that $y_{n}$ is also supported by $supp(n \mapsto X_{n})$ for all $n \in \mathbb{N}$. 

Define $f: \mathbb{N} \to \underset{n}{\cup}X_{n}$ by $ f(n)=\left\{ \begin{array}{ll}

x_{\frac{n}{2}} & \text{if}\:   n \;\text{is even}\\
\\
y_{\frac{n-1}{2}} & \text{if}\:  n \;\text {is odd} \end{array}\right.$. We can equivalently describe $f$ as being defined by $f(2k)=x_{k}$ and $f(2k+1)=y_{k}$. Clearly, $f$ is onto. Furthermore, because all $x_{n}$ and all $y_{n}$ are uniformly supported by $supp(n \mapsto X_{n})$, we have that $f(n)=\pi \cdot f(n)$, for all $\pi \in Fix(supp(n \mapsto X_{n}))$ and all $n \in \mathbb{N}$. Thus, according to Proposition \ref{2.18'}, we obtain that $f$ is also supported by $supp(n \mapsto X_{n})$, and so $\underset{n}{\cup}X_{n}$ is FSM countable. 

3. As in the proof of the first item.
\end{proof}

We can easily remark that under $\textbf{CC(fin)}$ a finitely supported subset $X$ of an invariant set is FSM Dedekind infinite if and only if $\wp_{fin}(X)$ is FSM Dedekind infinite. 

\begin{proposition}Let $Y$ be a finitely supported countable subset of an invariant set $X$. Then the set  $\underset{n \in \mathbb{N}}{\cup}Y^{n}$ is countable, where $Y^{n}$ is defined as the n-time Cartesian product of $Y$. 
\end{proposition}

\begin{proof}Since $Y$ is countable, we can order it as a sequence $Y=\{x_{1}, \ldots, x_{n}, \ldots\}$. The other sets of form $Y^{k}$ are \emph{uniquely} represented in respect of the previous enumeration of the elements of $Y$. Since $Y$ is finitely supported and countable, all the elements of $Y$ are supported by the same set $S$ of atoms. Thus, in the view of Proposition \ref{p1}, for each $k \in \mathbb{N}$, all the elements of $Y^{k}$ are supported by $S$. Fix $n \in \mathbb{N}$. On $Y^{n}$ define the S-supported strict well order relation $\sqsubset$ by: $(x_{i_{1}}, x_{i_{2}}, \ldots, x_{i_{n}}) \sqsubset (x_{j_{1}}, x_{j_{2}}, \ldots, x_{j_{n}})$ if and only if $\left\{ \begin{array}{ll}

i_{1} < j_{1}  \\

\text{or}\\
i_{1}=j_{1}\; \text{and}\; i_{2}<j_{2}\\

 \text{or}\\
\ldots \\
\text{or}\\
i_{1}=j_{1}, \ldots, i_{n-1}=j_{n-1}\; \text{and}\; i_{n}<j_{n} \end{array}\right.$.

Now, define an $S$-supported strict well order relation $\prec$ on $\underset{n \in \mathbb{N}}{\cup}Y^{n}$ by

$u \prec v$ if and only if $\left\{ \begin{array}{ll}

u \in Y^{n} \; \text{and}\; v \in Y^{m} \;\text{with}\; n<m  \\

\text{or}\\
u,v \in Y_{n}\; \text{and}\; u \sqsubset v \end{array}\right.$.

Therefore, there exists an $S$-supported order isomorphism between $(\underset{n \in \mathbb{N}}{\cup}Y^{n}, \prec)$ and $(\mathbb{N}, <)$.
\end{proof}

\section{Conclusion}

It is known that, when an infinite family of elements having no internal 
structure is considered by weakening some axioms of the ZF set theory, the 
results in ZF may lose their validity. According to Theorem 5.4 in 
\cite{hal}, multiple choice principle and Kurepa's antichain principle are 
both equivalent to the axiom of choice in ZF. However, in Theorem~9.2 
of~\cite{jech} it is proved that multiple choice principle is valid in the 
Second Fraenkel Model, while the axiom of choice fails in this model. 
Furthermore, Kurepa's maximal antichain principle is valid in the Basic 
Fraenkel Model, while multiple choice principle fails in this model. 
This means that the following two statements (that are valid in ZF) 
\emph{`Kurepa's principle implies axiom of choice'} and \emph{`Multiple 
choice principle implies axiom of choice'} fail in Zermelo Fraenkel set 
theory with atoms.

FSM is related to set theory with atoms, however in our approach $A$ is 
considered as a ZF set (without being necessary to modify the axioms of 
foundation or of extensionality), and invariant sets are defined as sets 
with group actions. Additionally, FSM involves an axiom of finite support 
which states that only atomic finitely supported structures (under a 
canonical hierarchical set-theoretical construction) are allowed in the 
theory. Therefore, there is indeed a similarity between the development of 
permutation models of set theory with atoms and FSM, but this framework is 
developed over the standard ZF in the form `usual sets together with 
actions of permutation groups' without being necessary to consider an 
alternative set theory.  The goal of this paper is to answer to a natural 
question whether the theorems involving the usual/non-atomic ZF sets 
remain valid in the framework of atomic sets with finite supports modulo 
canonical permutation actions. It is already known that there exist 
results that are consistent with ZF, but the are invalid when replacing 
`non-atomic structure' with `atomic finitely supported structure'. The ZF 
results are not valid in FSM unless we are able to reformulate them with 
respect to the finite support requirement. The proofs of the~FSM results 
should not brake the principle that any structure has to be finitely 
supported, which means that the related proofs should be \emph{internally 
consistent in FSM} and not retrieved from ZF. The methodology for moving 
from ZF into~FSM is based on the formalization of FSM into higher order 
logic (and this is not a simple task due to some important limitations) or 
on the hierarchical construction of supports using the $S$-finite support 
reasoning that actually represents an hierarchical method for defining the 
support of a structure using the supports of the sub-structures of the 
related structure. Since any structure has to be finitely supported in 
FSM, specific results (that are not derived from ZF) can also be obtained.

In this paper we study infinite cardinalities of finitely supported 
structures. The preorder relation~$\leq$ on FSM cardinalities defined by 
involving finitely supported injective mappings is antisymmetric, but not 
total. The preorder relation~$\leq^{*}$ on FSM cardinalities defined by 
involving finitely supported surjective mappings is not antisymmetric, nor 
total. Thus, Cantor-Schr{\"o}der-Bernstein theorem (in which cardinalities 
are ordered by involving finitely supported injective mappings) is 
consistent with the finite support requirement of FSM. However, the dual 
of Cantor-Schr{\"o}der-Bernstein theorem (in which cardinalities are 
ordered by involving finitely supported surjective mappings) is not valid 
for finitely supported structures. Several other specific properties of 
cardinalities are presented in Theorem \ref{cardord1}.

The idea of presenting various approaches regarding `infinite' belongs to 
Tarski who formulates several definitions of infinite in \cite{tarski24}. 
The independence of these definitions was later proved in set theory with 
atoms in \cite{levy1}. Such independence results can be transferred into 
classical ZF set theory by employing Jech-Sochor's embedding theorem 
stating that permutation models of set theory with atoms can be embedded 
into symmetric models of ZF, and so a statement which holds in a given 
permutation model of set theory with atoms and whose validity depend only 
on a certain fragment of that model, also holds in some well-founded model 
of ZF. In this paper we reformulate the definitions of (in)finiteness from 
\cite{tarski24} internally into FSM, in terms of finitely supported 
structures. The related definitions for `FSM infinite' are introduced in 
Section \ref{chap9}.  We particularly mention FSM usual infinite, FSM 
Tarski (of three types) infinite, FSM Dedekind infinite, FSM Mostowski 
infinite, FSM Kuratowski infinite, or FSM ascending infinite. We were able 
to establish comparison results between them and to present relevant 
examples of FSM sets that satisfy certain specific infinity 
properties. These comparison results are proved internally in FSM, by 
employing only finitely supported constructions. Some of the results are 
obtained by using the classical translation technique from ZF into FSM 
involving the $S$-finite support principle, while many other properties 
(especially those revealing uniform supports) are specific to~FSM.  
We also provide connections with FSM (uniformly) amorphous sets. We 
particularly have focused on the notion of FSM Dedekind infinity, and 
we proved a full characterization of FSM Dedekind infinite sets. For 
example, we were able to prove that $T_{fin}(A)$, 
$\wp_{fin}(\wp_{fs}(A))$, $A^{A}_{fs}$, $\wp_{fin}(A^{A}_{fs})$, 
$(A^{n})^{A}_{fs}$ (for a fixed $n \in \mathbb{N}$), $T_{fin}(A)^{A}_{fs}$ 
and $\wp_{fs}(A)^{A}_{fs}$ are not FSM Dedekind infinite (nor FSM 
Mostowski infinite), while $\wp_{fs}(\wp_{fin}(A))$ and 
$T^{\delta}_{fin}(A)$ are FSM Dedekind infinite. The notion of 
`countability' is described in FSM in Section \ref{countable}, where we 
present connections between countable choice principles and countable 
union theorems within finitely supported sets.

In Figure \ref{fig:1} we point out some of the relationships between the 
FSM definitions of infinite. The `red arrows' symbolize \emph{strict} 
implications (of from $p$ implies $q$, but $q$ does not imply $p$), while 
`black arrows' symbolize implications for which we have not proved yet if 
they may be strict or not (analyze this in respect of Remark 
\ref{remarema}). Blue arrows represent equivalences.

\begin{figure}[h]

\centering
\begin{tikzpicture}
[node distance = 1.1cm, auto,font=\footnotesize,
every node/.style={node distance=2.1cm},
comment/.style={rectangle, inner sep= 4pt, text width=3cm, node distance=0.25cm, font=\scriptsize\sffamily},
force/.style={rectangle, draw, fill=black!10, inner sep=4pt, text width=3cm, text badly centered, minimum height=1.0cm, font=\bfseries\footnotesize\sffamily}] 

\node [force] (TI) {X is FSM Tarski I infinite};
\node [force, below of=TI] (TIII) {X is FSM Tarski III infinite};
\node [force, text width=3cm, left=1cm of TIII] (PTI) {The FSM powerset of X is FSM Tarski I infinite};
\node [force, below of=TIII] (D inf) {X is FSM Dedekind infinite};
\node [force, text width=3cm, left=1cm of D inf] (PTIII) {The FSM powerset of X is FSM Tarski III infinite};
\node [force, text width=3cm, right=1cm of TIII] (NTIII) {There is a finitely supported bijection between X and the Cartesian product of the set of positive integers and X};
\node [force, below of=D inf] (M inf) {X is FSM Mostowski infinite};
\node [force, text width=3cm, right=1cm of M inf] (PD inf) {The finite powerset of X is FSM Dedekind infinite};
\node [force, text width=3cm, right=1cm of PD inf] (PDD inf) {The finite powerset of X contains an infinite uniformly supported subset};
\node [force, text width=3cm, left=1cm of M inf] (A inf) {X is FSM ascending infinite};
\node [force, below of=M inf] (cont) {X contains an infinite uniformly supported subset};
\node [force, text width=3cm, left=1cm of cont] (TII) {X is FSM Tarski II infinite};
\node [force, text width=3cm, right=1cm of cont] (nua) {X is FSM non uniformly amorphous};
\node [force, below of=cont] (usual) {X is FSM usual infinite};
\node [force, text width=3cm, right=1cm of nua] (n-a) {X is FSM non amorphous};
\node [force, below of=cont] (usual) {X is FSM usual infinite};
\node [force, text width=3cm, right=1cm of usual] (ci) {X is FSM covering infinite};
\node [force, text width=3cm, left=1cm of usual] (PPI) {The FSM powerset of the finite powerset of X is FSM Dedekind infinite};
\node [force, text width=3cm, right=1cm of ci] (PAI) {The finite powerset of X is FSM ascending infinite};
\node [force, text width=3cm, right=1cm of D inf] (PPPP) {The powerset of X is FSM Dedekind infinite};
\node [force, text width=3cm, right=1cm of PPPP] (PPPP') {There exists a finitely supported surjection from X onto the set of positive integers};

\path[->, thick, red] 
(TI) edge (TIII)
(PTI) edge (PTIII)
(TIII) edge (D inf)
(D inf) edge (A inf)
(M inf) edge (TII)
(TII) edge (usual)
(cont) edge (usual)
(n-a) edge (usual)
(nua) edge (n-a)
(D inf) edge (PPPP)
(PPPP) edge (n-a)
(D inf) edge (PTIII)
(PD inf) edge (PPPP);

\path[->,thick] 
(TIII) edge (PTI)
(D inf) edge (M inf)
(D inf) edge (PD inf)
(PD inf) edge (PDD inf)
(A inf) edge (TII)
(M inf) edge (cont)
(PD inf) edge (nua)
(M inf) edge (nua)
(PTIII) edge (A inf);

\path[->, thick, blue] 
(TIII) edge (NTIII)
(NTIII) edge (TIII)
(PDD inf) edge (cont)
(cont) edge (PDD inf)
(ci) edge (usual)
(usual) edge (ci)
(PPPP) edge (A inf)
(A inf) edge (PPPP)
(usual) edge (PPI)
(PPI) edge (usual)
(PAI) edge (ci)
(ci) edge (PAI)
(PPPP) edge (PPPP')
(PPPP') edge (PPPP);

\end{tikzpicture} 
\caption{FSM relationship between various forms of infinity}
\label{fig:1}
\end{figure}
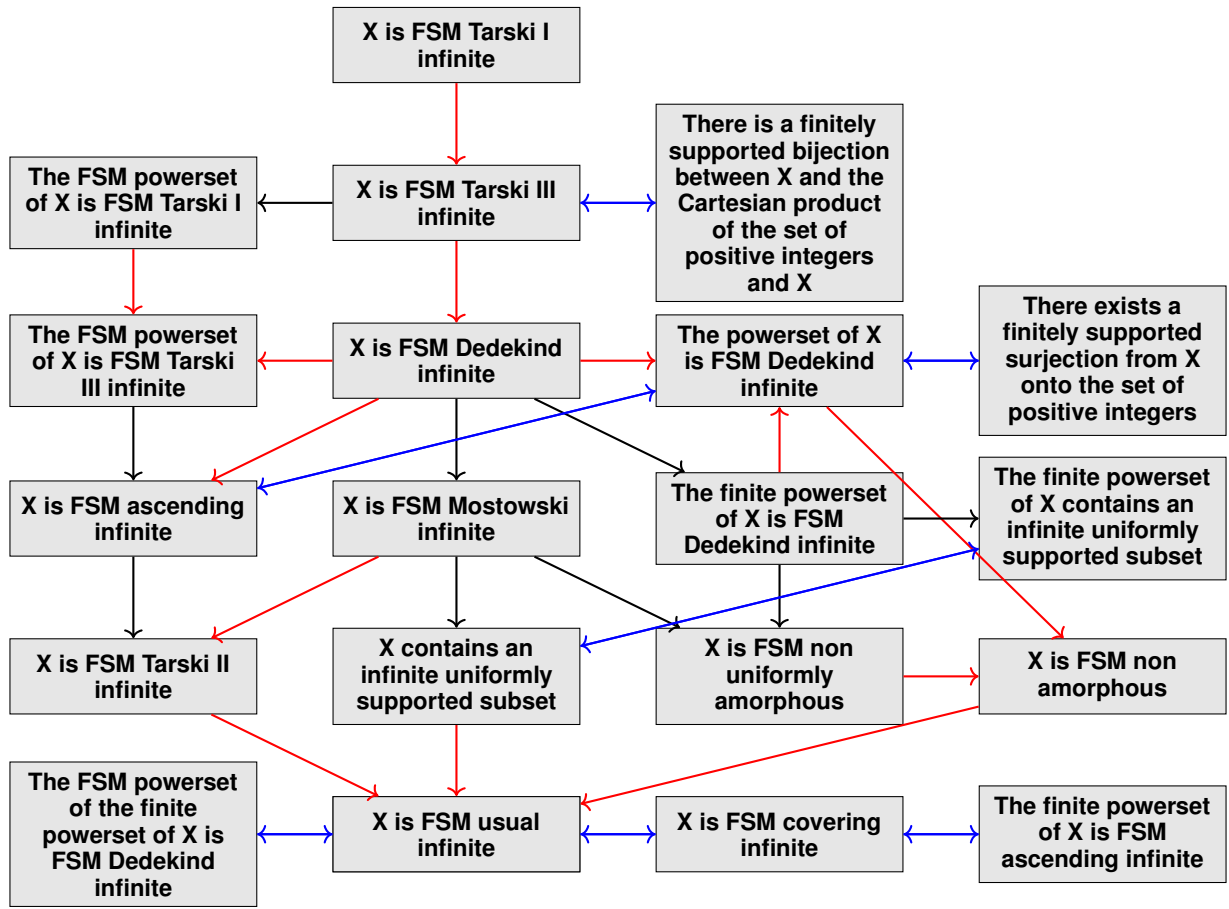

\begin{center}
 \begin{tabular}
{|l|l|l|l|l|l|l|l|p{0.5cm}|} \hline	Set & Tarski I inf & Tarski III inf & Ded. inf & Most. inf & Asc. inf & Tarski II inf & Non-amorph. \\ 
\hline	$A$ & No & No & No & No & No & No & No \\ 
\hline	$A+A$ & No & No & No & No & No & No & Yes \\ 
\hline	$A \times A$ & No & No & No & No & No & No & Yes \\ 
\hline	$\wp_{fin}(A)$ & No & No & No & No & Yes & Yes & Yes \\ 
\hline	$T_{fin}(A)$ & No & No & No & No & Yes & Yes & Yes \\ 
\hline	$\wp_{fs}(A)$ & No & No & No & No & Yes & Yes & Yes \\ 
\hline	$\wp_{fin}(\wp_{fs}(A))$ & No & No & No & No & Yes & Yes & Yes  \\ 
\hline	$A^{A}_{fs}$ & No & No & No & No & Yes & Yes & Yes  \\ 
\hline	$T_{fin}(A)^{A}_{fs}$ & No & No & No & No & Yes & Yes & Yes  \\ 
\hline	$\wp_{fs}(A)^{A}_{fs}$ & No & No & No & No & Yes & Yes & Yes  \\ 
\hline	$A \cup \mathbb{N}$ & No & No & Yes & Yes  & Yes & Yes & Yes  \\ 
\hline	$A \times \mathbb{N}$ & No & Yes & Yes & Yes  & Yes & Yes & Yes \\ 
\hline	$\wp_{fs}(A \cup \mathbb{N})$ & No & Yes & Yes & Yes  & Yes & Yes & Yes \\ 
\hline	$\wp_{fs}(\wp_{fs}(A))$  & ? & Yes & Yes & Yes  & Yes & Yes & Yes \\ 
\hline	$A^{\mathbb{N}}_{fs}$ & Yes & Yes & Yes & Yes  & Yes & Yes & Yes \\ 
\hline	$\mathbb{N}^{A}_{fs}$ & Yes & Yes & Yes & Yes  & Yes & Yes & Yes \\ 

\hline 
\end{tabular} 

\end{center}

\qquad In this final table we present the forms of infinity satisfied by the classical FSM sets.

\end{document}